\documentclass{amsart}
\usepackage[utf8]{inputenc}
\usepackage{amsmath, amsthm, amssymb, color, mathtools, amsfonts, mathrsfs,hyperref,enumitem, footmisc,url}

\usepackage{appendix}

\usepackage[alphabetic,abbrev,nobysame]{amsrefs}

\usepackage[noabbrev, capitalise]{cleveref}
\Crefname{appsec}{Appendix}{Appendices}
\crefformat{footnote}{#2\footnotemark[#1]#3}
\numberwithin{equation}{section}

\usepackage{scalerel}

\usepackage{bm}

\newtheorem{thm}{Theorem}[section]
\newtheorem{lem}[thm]{Lemma}
\newtheorem{prop}[thm]{Proposition}
\newtheorem{cor}[thm]{Corollary}

\newtheorem{introtheorem}{Theorem}

\theoremstyle{definition}
\newtheorem{ex}[thm]{Example}
\theoremstyle{definition}
\newtheorem{notation}[thm]{Notation}
\theoremstyle{definition}
\newtheorem{defn}[thm]{Definition}
\theoremstyle{remark}
\newtheorem{rmk}[thm]{Remark}

\DeclareMathOperator*{\unprod}{\scalerel*{\underline{\prod}}{\sum}}

\newcommand{\br}{\underline{\cdot}}

\newcommand{\inner}[1]{\left\langle #1 \right\rangle}
\newcommand{\Set}[1]{\left\lbrace #1 \right\rbrace}
\newcommand{\un}[1]{\underline{#1}}
\newcommand{\stirling}[2]{\left[\genfrac{}{}{0pt}{}{#1}{#2}\right]}

\newcommand{\lMod}[1]{#1\text{-}\mathsf{Mod}}
\newcommand{\flkmod}[2]{#1\text{-}\mathsf{mod}_{#2}}
\newcommand{\flRmod}[2]{#1\text{-} \mathsf{proj}_{#2}}

\newcommand{\rComod}[1]{\mathsf{Comod}\text{-}#1}
\newcommand{\flmod}[1]{#1\text{-}\mathsf{proj}}
\newcommand{\fplmod}[1]{#1\text{-}\mathsf{mod}}

\newcommand{\op}{\mathrm{op}}
\newcommand{\colim}{\operatorname{colim}}
\newcommand{\cop}{\mathrm{cop}}
\newcommand{\Hom}{\operatorname{Hom}}
\newcommand{\iHom}{\underline{\operatorname{Hom}}}

\newcommand{\ev}{\operatorname{ev}}
\newcommand{\Fun}{\operatorname{Fun}}
\newcommand{\coev}{\operatorname{coev}}
\newcommand{\uHom}{\underline{\operatorname{Hom}}}
\newcommand{\uEnd}{\underline{\operatorname{End}}}
\newcommand{\Id}{\operatorname{Id}}
\newcommand{\lex}{\mathrm{lex}}

\newcommand{\Nat}{\operatorname{Nat}}
\newcommand{\one}{\mathds{1}}

\newcommand{\Vect}{\mathsf{Vect}_\Bbbk}

\newcommand{\cC}{\mathcal{C}}
\newcommand{\cD}{\mathcal{D}}
\newcommand{\cE}{\mathcal{E}}

\newcommand{\cP}{\mathcal{P}}
\newcommand{\cR}{\mathcal{R}}

\newcommand{\cT}{\mathcal{T}}

\newcommand{\wcC}{\widehat{\mathcal{C}}}
\newcommand{\wcD}{\widehat{\mathcal{D}}}
\newcommand{\wF}{\widehat{F}}

\newcommand{\wcP}{\widehat{\mathcal{P}}}

\newcommand{\mI}{\mathbb{I}}
\newcommand{\mC}{\mathbb{C}}
\newcommand{\mK}{\mathbb{K}}
\newcommand{\mQ}{\mathbb{Q}}
\newcommand{\mO}{\mathbb{O}}
\newcommand{\mZ}{\mathbb{Z}}
\newcommand{\mN}{\mathbb{Z}}

\newcommand{\dq}[1]{\operatorname{det}_q(#1)}
\newcommand{\dnu}[1]{\operatorname{det}_\nu(#1)}
\newcommand{\brdq}[1]{\operatorname{\underline{det}}_q(#1)}
\newcommand{\brdeps}[1]{\operatorname{\underline{det}}_\epsilon(#1)}

\newcommand{\ov}[1]{\overline{#1}}

\newcommand{\scalar}{\sigma}
\newcommand{\monom}{\operatorname{mon}}
\newcommand{\intf}{\textrm{int}}

\allowdisplaybreaks

\title{Presentations for small reflection equation algebras of type A}
\author[J. Cooke]{Juliet Cooke}
\author[R. Laugwitz]{Robert Laugwitz}
\address{School of Mathematical Sciences,
University of Nottingham, University Park, Nottingham, NG7 2RD, UK}
\email{robert.laugwitz@nottingham.ac.uk}
\usepackage[en-GB]{datetime2}
\date{June 12, 2025}

\usepackage[all]{xy}
\usepackage{dsfont}

\begin{document}

\begin{abstract}
We give presentations, in terms of the generators and relations, for the reflection equation algebras of type $GL_n$ and $SL_n$, i.e., the covariantized algebras of the dual Hopf algebras of the small quantum groups of $\mathfrak{gl}_n$ and $\mathfrak{sl}_n$. Our presentations display these algebras as quotients of the infinite-dimensional reflection equation algebras of types $GL_n$ and $SL_n$ by identifying additional relations that correspond to twisting the nilpotency and unipotency relations of the finite-dimensional quantum function algebras. The presentations are valid for appropriately defined integral forms of these algebras.
\end{abstract}


\maketitle

\vspace{-10pt}

\section{Introduction}

The reflection equation algebra is a fundamental structure in quantum algebra where it arises in several contexts. 
In quantum integrable systems, the reflection equation is a consistency equation, analogous to the quantum Yang--Baxter equation, but involving reflection at a wall, see \cite{Che}, \cite{Kul} and references therein. The \emph{reflection equation (RE) algebra} is an algebra associated to the reflection equation \cite{KuS}. The RE algebra has a K-matrix that defines a braided module category structure  on its representations \cites{KoSt,Kol}. In fact, the RE algebra is the universal (co)module algebra with a K-matrix \cites{BBJ2,LWY}.

The RE algebra can be constructed from the \emph{FRT algebra} (after Fadeev--Reshe\-tikhin--Takhtajan \cites{FRT,RTF}), also called the \emph{quantum function algebra}, by a cocycle twist \cites{DM,DKM} or by \emph{transmutation} \cites{Maj1,Majid95book}. To state the twisting result, the FRT algebra $A_\cC$ is viewed as a bimodule algebra over a quasitriangular Hopf algebra $H$ with R-matrix $R$, i.e., as an algebra in $\cC^\op\boxtimes \cC$ for the braided monoidal category $\cC$ of representations of $H$. Then, the RE algebra $B_\cC$ can be constructed the $2$-cocycle twist of $A_\cC$ by $R_{13}R_{23}$ as a bimodule algebra \cite{DM}. To obtain the transmutation description of $B_\cC$, $A_\cC$ is realized as the dual $H^\circ$ of the Hopf algebra $H$. Then, the product of $A_\cC$ is changed to the covariantized product 
\begin{align}
a \br b &=  a_{(2)} b_{(3)} \cR\left(a_{(3)} \otimes S \left(b_{(1)}\right)\right) \cR\left(a_{(1)} \otimes b_{(2)}\right),
\end{align}
of \cite{Majid95book}*{Section~7.4}, for the dual R-matrix $\cR$. This equips $B_\cC$ with the structure of a Hopf algebra \emph{internal} to $\cC$ \cite{Majid95book}*{Section~9.4}.

Using category theory, there are different universal constructions used to define the FRT and RE algebras associated to a ribbon category $\cC$. The first universal construction uses the coend 
$$
B_\cC=\int^{X\in \cC} X^*\otimes X
$$
of Lyubashenko \cites{Lyu1,Lyu2,LM}, using the tensor product $\otimes$ of $\cC$ and the left dual $X^*$. This universal construction also reveals structure of $B_\cC$ as a Hopf algebra internal to $\cC$ in a natural way. The FRT algebra has a similar definition as 
$$A_\cC=\int^{X\in \cC} X^*\boxtimes X,$$
constructed as an algebra object in the Deligne tensor product $\cC^\op\boxtimes \cC$, see Section \ref{sec:AC-coend}. If is a category of representations over a Hopf algebra $H$, both $A_\cC$ and $B_\cC$ can be modelled on the (finitary) dual $H^\circ$ but their algebra structures differ.
More generally, $B_\cC\cong T(A_\cC)$ is the image of $A_\cC$ under the tensor functor $T\colon \cC^\op\boxtimes \cC\to \cC$ obtained from the tensor product of $\cC$, see \cite{EGNO}*{Exercise~8.25.7}. Equivalently, $B_\cC$ can be defined, via a generalization of Tannaka--Krein reconstruction \cite{Majid95book}*{Section~9} called \emph{braided reconstruction}, as the object representing the functor 
$$\Nat(\Id_\cC,\Id_\cC\otimes (-))\colon \cC\to \Vect, \quad \text{i.e.,}
\quad 
\Hom_{\cC}(B_\cC,Y)\cong \Nat(\Id_\cC,\Id_\cC\otimes Y),
$$
for objects $Y$ of $\cC$. The FRT algebra can also be constructed as the \emph{internal} endomorphism algebra $\uEnd(\one)$ of the tensor unit with respect to the regular $\cC$-bimodule structure on $\cC$ and is called the \emph{canonical algebra} in \cites{EGNO}.
The relationship between these universal constructions is known but we review them, with coherent choices of conventions, in Section \ref{sec:FRT-RE-algebra}. To ensure existence of the RE and FRT algebras via these universal constructions, it is, in general, necessary to pass to the completion $\wcC$ of $\cC$ under directed colimits, see Appendix~\ref{sec:appendix}.

The RE algebra, with its universal property as a coend, is of central importance in the construction of invariants of links and $3$-dimensional cobordisms from tensor categories, see e.g. \cites{Lyu1,TV,DGGPR}. 
Moreover, the reflection equation algebra $B_\cC$ appeared as a key tool in the algebraic evaluation of factorization homology on punctured surfaces obtained by gluing handles \cites{BBJ,BBJ2}. The factorization homology is then equivalent to the category of modules over a braided tensor product of copies of $B_\cC$ internal to $\cC$ determined by the surface \cite{BBJ}*{Theorem~5.14}. This application motivates giving an explicit presentation in order to compute factorization homology by classifying representations of the braided products of copies of $B_\cC$. 

The main cases of interest of the FRT and RE algebra constructions consider the ribbon category $\cC$ of finite-dim\-ensio\-nal type I representations of a quantized enveloping algebra $U_q(\mathfrak{g})$. For a generic parameter $q$, the FRT and RE algebras are quadratic algebra and presentations by generators and relations are well-known, see e.g. \cites{Majid95book,DL,JW} for details.
 The FRT algebra $A_\cC$ is isomorphic to the quantum function algebra $O_q(G)$, which has generators $x_i^j$, and the RE algebra is denoted by $B_q(G)$ with generators $u_i^j$, where $i,j=1,\ldots, r=\operatorname{rank}\mathfrak{g}$.  
We review these presentations, for a suitable integral form of these algebras, in Section~\ref{sec:q-fun-alg} for the FRT algebras and in Section~\ref{sec:generic-RE} for the RE algebra. In the $GL_n$-case, and for integral forms, the presentation of the FRT algebra is obtained from \cite{Tak} where non-degenerate pairings with the (integral form of) the corresponding universal enveloping algebras are constructed. These non-degenerate pairings are key to identifying the FRT algebras with the abstractly defined coends. 

In this paper, we investigate the analogues of the FRT and RE algebras specializing $q$ to a primitive $\ell$-th root of unity $\epsilon$, where the order $\ell$ is odd and $\cC$ is the category of representations of the small quantum group $u_\epsilon(\mathfrak{g})$. We restrict ourselves to type A, where $G=SL_n$ (or $GL_n$) and $\mathfrak{g}=\mathfrak{sl}_n$ (or $\mathfrak{gl}_n$). In these cases, the finite-dimensional FRT algebra, called \emph{small quantum function algebra}, is denoted by $o_\epsilon(G)$ (see Definitions~\ref{pres-oqgln} and \ref{pres-oqsln}). Its generators $x_i^j$ are subject to the usual relations from the generic case as well as 
\begin{equation}\label{eq:intro-small-o}
\left(x^i_j\right)^\ell\quad (1\leq i\neq j\leq n),\quad \text{and}\quad  \left(x^i_i\right)^{\ell}-1.
\end{equation}

\subsection{Statement of results}

The main result gives the following presentations for the finite-dimensional (\emph{small}) reflection equation algebras $b_\epsilon(G)$ associated to $G=GL_n, SL_n$ for a root of unity $\epsilon$ of odd order $\ell>2$.\footnote{The assumption that $\ell$ is odd ensures existence of a braiding on the categories of $u_q(\mathfrak{g})$-modules and is needed for certain results of Section~\ref{sec:qgroup-defns}.}

\begin{introtheorem}[See Theorem~\ref{thm:main-theorem}]
The algebra $b_\epsilon(GL_n)$ is generated by $u^k_l$ for all $1\leq l,k\leq n$ subject to the relations \eqref{eq:BqMn-rel-1}--\eqref{eq:BqMn-rel-4} as well as the additional relations 
\begin{align}
\left(u^k_l\right)^{\br \ell}&=0,\\
\sum_{\lambda\models \ell} \scalar_\epsilon (\lambda) \sum_{(\beta_1,\ldots, \beta_{\ell+1})\in V^k(\lambda)} u^{\beta_1}_{\beta_2}\br \ldots \br u^{\beta_\ell}_{\beta_{\ell+1}}&=1,\label{eq:introformula}
\end{align}
for all $1\leq k\neq l\leq n$. The algebra $b_\epsilon(SL_n)$ is the quotient of $b_\epsilon(GL_n)$ by the additional relation
\begin{align}\label{eq:intro-det}
\brdeps{n}=\sum_{\sigma \in S_n} (-\epsilon)^{l(\sigma)}\epsilon^{e(\sigma)}u^{n}_{\sigma(n)}\br \ldots \br u^1_{\sigma(1)}&=1,
\end{align}
where $l(\sigma)$ is the number of inversions in $\sigma$ and $e(\sigma)=|\Set{i=1,\ldots, n~|~ \sigma(i)>i}|$.
\end{introtheorem}

Equation \eqref{eq:introformula} is a closed combinatorial formula involving the following notation. 
\begin{itemize}
\item A \emph{composition} of \emph{weight} $\ell$ is a sequence $\lambda=(\lambda_1,\ldots, \lambda_k)\in \mZ_{\geq 1}^k$ of positive integers such that $\sum_{i=1}^k \lambda_i=\ell$. Its length is $|\lambda|=k$.
\item We define the \emph{$\epsilon$-scalar} associated to the composition $\lambda$,
    \begin{align}
        \scalar_\epsilon(\lambda)= \frac{\prod_{j=1}^{\ell-1} \left( 1 - \epsilon^{-2(\ell-j)} \right)}{ \prod_{k=1}^{|\lambda|-1} \left( 1 - \epsilon^{-2 \left( \ell  - \sum_{j=1}^{k} \lambda_{|\lambda|+1-j} \right)} \right)}.
    \end{align}      
We shown in Lemma \ref{lem:lambda-scalar} these $\epsilon$-scalars satisfies the recursion
$$ \scalar_\epsilon\left(\lambda\right)= \Big(\prod_{j=1}^{\lambda_{|\lambda|-1}} \big( 1 - \epsilon^{-2(\ell-j)} \big) \Big) \scalar_\epsilon (\lambda_1,\ldots, \lambda_{|\lambda|-1})\quad \in\mZ[\epsilon,\epsilon^{-1}].$$
\item The indexing set of the second summation is 
$$V^k(\lambda)=\left\{(\beta_1,\ldots, \beta_{\ell+1})\in \{1,\ldots, k\}~\middle| ~\beta_{\sum_{i=1}^x \lambda_i + 1} = k, \text{ for $x=0,\ldots,|\lambda|$}\right\}.$$
\end{itemize}

The theorem is proved by twisting the additional relations \eqref{eq:intro-small-o} of the small quantum function algebras $A=o_\epsilon(GL_n)$ or $o_\epsilon(SL_n)$ via a certain \emph{twisting map}
\begin{equation}
\Psi\colon A\to \un{A}, \quad \Psi(1)=1, \quad \Psi\left(x^i_j\right)=u^i_j,
\end{equation}
which fixes the generators, see \eqref{eq:twistingmap}. This approach is commonly used to obtain presentations for RE algebras at generic parameters \cites{KS,JW}.
It turns out that most additional, non-quadratic, relations twist in an easy way but twisting the relation $(x_k^k)^\ell=1$ leads to the combinatorial formula in Equation~\eqref{eq:introformula}. 
The formula for the quantum determinant of the RE algebra, used in Equation \eqref{eq:intro-det} is obtained from the results of \cite{JW} who computed the center of the reflection equation algebra at generic $q$.

When either $n$ or $\ell$ are small, the presentation simplifies.
In the case when $n=2$, Equation~\eqref{eq:introformula} becomes 
\begin{align}
\sum_{\lambda\models \ell} \scalar_\epsilon (\lambda) \monom(\lambda_{|\lambda|})\br \monom(\lambda_{|\lambda|-1})\br \ldots \br \monom(\lambda_1)&=1,
\end{align}
where for an integer $l\geq 1$, and $a=u^1_1$, $b=u^1_2$, $c=u^2_1$, $d=u^2_2$,
$$\monom(l)=\begin{cases}
d, & \text{if $l=1$,}\\
c\br a^{\br (l-2)}\br b, & \text{if $l>1$.}
\end{cases}$$
The presentations for $b_\epsilon(GL_2)$ and $b_\epsilon(SL_2)$ with this simplified combinatorial formula are obtained in Corollary~\ref{cor:bepGL2-SL2}.

The complexity of Equation~\eqref{eq:introformula} grows exponentially in the order of $\ell$. For $\ell=3$ it is still easy to compute all coefficients in a formula. For any $n\geq 2$, the sets $V^k(\lambda)$ from Definition \ref{def:k-extension} are in this case
\begin{gather*}
V^k(1,1,1)=\Set{(k,k,k,k)},\qquad V^k(1,2)=\Set{(k,k,i,k)~|~1\leq i< k}, \\ V^k(2,1)=\Set{(k,i,k,k)~|~1\leq i< k},\quad
V^k(3)=\Set{(k,i,j,k)~|~1\leq i,j< k}
\end{gather*}
and we compute
\begin{align}
\begin{split}
1=\left(u^k_k\right)^{\br 3}&+\left(1-\epsilon^{-2}\right) \sum_{i<k} u^k_i \br u^i_k\br u^k_k + \left(1-\epsilon^{-4}\right)\sum_{i<k} u^k_k\br u^k_i\br u^i_k \\ &+ \left(1-\epsilon^{-2}\right)\left(1-\epsilon^{-4}\right)\sum_{i,j<k}u^k_i\br u^i_j\br u^j_k.
\end{split}
\end{align}
This can be combined with Lemma~\ref{lem:BqMn-relations} (a presentation for the quadratic algebras $B_q(GL_n)$) 
to give a full presentation for $b_{\epsilon}(GL_n)$. If, for example, $n=3$, adding the relation
\begin{align*}
1=\brdeps{3}=&~u^3_3\br u^2_2 \br u^1_1- q^{2}u^3_3\br u^2_1 \br u^1_2- q^{2}u^3_1\br u^2_2 \br u^1_3\\
&- q^{2}u^3_2\br u^2_3 \br u^1_1+ q^{3}u^3_2\br u^2_1 \br u^1_3+ q^{4}u^3_1\br u^2_3 \br u^1_2.
\end{align*}
gives a full presentation for $b_{\epsilon}(SL_3)$, where $\epsilon=e^{2\pi i/3}$.

\subsection{Summary of content}

Section \ref{sec:FRT-RE-algebra} provides an overview that relates the FRT algebra and the reflection equation algebras associated to a ribbon category $\cC$. We fix a consistent set of conventions and describe the equivalence between defining the FRT algebra as coend and via interal homs, as well as the reflection equation algebra defined via braided reconstruction by Majid \cite{Majid95book}, as coend by Lyubashenko \cites{Lyu1,Lyu2,LM}, and as internal homs by \cites{EGNO, BBJ}. If $\cC$ is given by representations of a (possibly infinite rank) Hopf algebras over a non-zero commutative ring $R$, these constructions are linked to concrete models involving the dual of a Hopf algebra.

Section \ref{sec:qgroup-defns} contains definitions on quantized enveloping algebras and quantum function algebras associated to $GL_n$ and $SL_n$ following Takeuchi \cite{Tak}. Quantum groups and quantum function algebras are displayed as duals via non-degenerate pairings. We link the quantum function algebras defined here to the abstractly defined FRT algebras from Section~\ref{sec:FRT-RE-algebra}.

Section \ref{sec:results} contains the main results of the paper giving presentations for the reflection equation algebras of type $GL_n$ and $SL_n$ at roots of unity and their associated finite-dimensional quotients, again, including integral forms.

Finally, Appendix~\ref{sec:appendix} contains a summary of the concepts and results from the theory of locally finitely presentable (LFP) categories required in order to form cocompletions of tensor categories under filtered colimits. Here, we take up some ideas of \cites{Lyu3} and describe how to associate to an $R$-linear tensor category $\cC$ with right exact (or, cocontinuous) tensor product in both components, a cocompletion that adds all directed colimits. This is necessary in order to provide an ambient category in which the FRT algebra and RE algebras $A_\cC$ and $B_\cC$ exist. This appendix was added as an exposition tailored to the study of possibly infinite-dimensional quantum groups on these topics was hard to find in the literature.

We remark that the contents of Section~\ref{sec:FRT-RE-algebra} and Appendix~\ref{sec:appendix} are written in larger generality and are not required to read Section~\ref{sec:results} \emph{if} one takes Majid's definition of the covariantized algebra of the dual as a definition of the RE algebra (see Definition~\ref{defn:cov-alg}). From this point of view, Section~\ref{sec:results} can be read independently with reference back to the definition of the relevant quantum groups in Section~\ref{sec:qgroup-defns} as required. The other sections are required to link the RE algebra to other models defined by various universal properties. 

\subsection*{Acknowledgements}

This research was funded by the University of Nottingham through a Nottingham Research Fellowship. R.~L. thanks Alex Schenkel for helpful conversations on locally finitely presentable categories. The authors thank an anonymous reviewer for careful reading of the manuscript that corrected several typos and many helpful comments.

\section{The FRT algebra and the reflection equation algebra}
\label{sec:FRT-RE-algebra}

In this section we review the construction of the FRT algebra $A_\cC$  and the reflection equation  algebra $B_\cC$ associated to a braided tensor category $\cC$. This section surveys different approaches to these objects by relating the constructions of Majid \cite{Majid95book} via braided reconstruction theory, Lyubashenko \cites{Lyu1,Lyu2,LM} via coends, and \cites{EGNO, BBJ} via internal homs. We will define the FRT algebra $A_\cC$ as an object in $\wcC^\op\boxtimes \wcC$, for the cocompletion $\wcC$ of $\cC$, and the reflection equation algebra as its image 
$$B_\cC=T(A_\cC)$$
 under the tensor functor $T\colon \wcC^\op\boxtimes \wcC\to \wcC$ given by the tensor product. The cocompletion $\wcC$ is discussed in Appendix~\ref{sec:appendix}.

We are particularly interested in the case when $H$ a Hopf algebra over an non-zero commutative ring $R$ and $\cC=\flRmod{H}{R}$, the full tensor subcategory of the monoidal category of $H$-modules which are finitely generated projective over $R$.
Throughout this section, we will describe concrete models for $A_\cC$ and $B_\cC$ in this case.

Note that if $\cC$ is a finite abelian tensor category over a field $\Bbbk$, e.g., when $H$ is finite-dimensional, $A_\cC$ and $B_\cC$ are objects in $\cC^\op\boxtimes \cC$, respectively, $\cC$, see Lemma~\ref{lem:finite-case1} and there is no need to work with the cocompletion $\wcC$ of $\cC$. In this section, we keep a more general setup in order to work with (possibly infinite rank) Hopf algebra over rings such as $R=\mZ[q,q^{-1}]$ in later parts of this article.

In the following, we will assume that $\cC$ is a braided $R$-linear tensor category, cf. Definition \ref{def:tensor}. 
In addition, we assume, when necessary, that $\cC$ has duals and a \emph{ribbon twist} (or \emph{balancing}) $\theta\colon \Id_\cC\to \Id_\cC$ such that 
$$\theta_{X \otimes Y} = (\theta_X \otimes \theta_Y)   \Psi_{Y,X}  \Psi_{X,Y}, \qquad (\theta_X)^* = \theta_{X^*}, \qquad \theta_\one = \Id_\one. $$
Thus, $\cC$ comes with a monoidal natural isomorphism 
$\tau\colon \Id_\cC\to (-)^{**}$, i.e, a \emph{pivotal structure}, given by 
\begin{equation}
\label{eq:pivotal-structure}    
\tau_X:=(\ev_X\otimes \Id_{X^{**}})(\Id_{X^*}\otimes \Psi_{X^{**},X})(\coev_{X^*}\otimes \Id_X)\theta_X.
\end{equation}
We summarize this set of assumptions by calling $\cC$ an \emph{$R$-linear ribbon category}. For example, if $R$ is a commutative ring, the the category $\flmod{R}$ of finitely-generated projective $R$-modules is an $R$-linear ribbon category together with the relative tensor product $\otimes_R$, the symmetric braiding, and trivial ribbon structure $\theta=\Id$.

\subsection{The FRT algebra \texorpdfstring{$A_\cC$}{AC} as a coend}\label{sec:AC-coend}

In this section, we start by introducing the FRT algebra $A_\cC$ as a coend.

We denote the opposite category of $\cC$ (with opposite composition of morphisms) by $\cC^\vee$ and denote the monoidal category with opposite tensor product by $\cC^\op$. We consider the bifunctor
$$\cC^\vee \times \cC\to \cC^\op\boxtimes \cC, \quad (X,Y) \mapsto X^*\boxtimes Y.$$
Its \emph{coend} is the following coequalizer
$$
A_\cC=\int^{X\in \cC}X^*\boxtimes X=\operatorname{Coeq}\Bigg(
{
\xymatrix{ \displaystyle\coprod_{Y,Z,f\in \Hom_\cC(Y,Z)}Z^*\boxtimes Y\quad \ar@<5pt>[r]^-{f^*\boxtimes \Id_Y}\ar@<-5pt>[r]_-{\Id_{Z^*}\boxtimes f}& \quad \displaystyle\coprod_{X\in \cC}X^*\boxtimes X}
}
\Bigg),
$$
see e.g. \cite{FS} for a survey on coends. The coend $A_\cC$ exists as an object in the cocompletion $\wcC^\op\boxtimes \wcC$, see Appendix \ref{sec:tensor-cocompletion}.
By definition, $A_\cC$ comes equipped with universal component maps $i_X\colon X^*\boxtimes X\to A_\cC$, for every object $X\in \cC$, which make all diagrams of the form 
\begin{equation}\label{eq:naturality-iX}
\xymatrix{
Z^*\boxtimes Y\ar[d]^{\Id\boxtimes f}\ar[rr]^{f^*\boxtimes \Id}&&Y^*\boxtimes Y\ar[d]^{i_Y}\\
Z^*\boxtimes Z\ar[rr]^{i_Z}&& A_\cC
},
\end{equation}
for $f\in \Hom_\cC(Y,Z)$, commute. The object $A_\cC$ of $\wcC^\op\boxtimes \wcC$, with the maps $i_X$, is universal in the following sense. If $A$ is another object in $\wcC^\op\boxtimes \wcC$ with maps $j_X\colon X^*\boxtimes X\to A$ that make the same diagrams as in Equation \eqref{eq:naturality-iX} commute with $j_X,j_Y$ in place of $i_X,j_Y$, then there exists a unique morphism $c\colon A_\cC\to A$ such that $c\circ i_X=j_X$,
for any object $X\in \cC$.

The induced map $m$ in the diagram
$$
\xymatrix{
& (Y^*\otimes X^*) \boxtimes (X\otimes Y)\ar@{=}[dl] \ar@{=}[dr]^{\sim}& \\
(X^*\boxtimes X)\otimes (Y^*\boxtimes Y)\ar[d]^{i_X\otimes i_Y}\ar[rr]&&(X\otimes Y)^* \boxtimes (X\otimes Y)\ar[d]^{i_{X\otimes Y}}\\
A_\cC\otimes A_\cC\ar[rr]^m&& A_\cC
},
$$
together with the unit map $u$ obtained from 
$$
\xymatrix{
\one^* \boxtimes \one \ar@{=}[dr] \ar[rr]^{i_\one}&& A_\cC\\
&\one \boxtimes \one\ar[ur]^{u}&
}
$$
make $A_\cC$ an algebra object in $\wcC^\op\boxtimes \wcC$.

\begin{defn}\label{defn:AC}
We call the algebra $A_\cC$ in $\wcC^\op\boxtimes \wcC$ the \emph{Faddeev--Reshetikhin--Takhtajan (FRT) algebra} associated to $\cC$. 
\end{defn}

\begin{ex}\label{ex:finiteHopf}
Denote by $\flmod{R}$ the category of finitely generated projective $R$-modules. We denote the tensor product $\otimes_R$ simply by $\otimes$ and note that $\flmod{R}$ is a rigid $R$-linear tensor category. Its cocompletion is the $R$-linear tensor category $\lMod{R}$, see Example~\ref{ex:Rmod-compl}.

Assume that $H$ is a Hopf algebra in $\flmod{R}$ and let $\cC=\flRmod{H}{R}$ be the category of $H$-modules which are finitely generated projective  over $R$.
Then $\cC$ has duals, given by the duals in $\flmod{R}$ where, for an object $V$, $H$ acts by 
$$(h\cdot f)(v)=f(S(h)\cdot v),$$
for $h\in H$, $v\in V$, $f\in V^*=\Hom_R(V,R)$.
The (left) dual $H^*$ is also a Hopf algebra in the same monoidal category. We always assume that the antipode of a Hopf algebra is an isomorphism. The cocompletion $\wcC$ can be identified with the $R$-linear tensor category $\lMod{H}_R$ of \emph{all} $H$-modules over $R$, see Example~\ref{ex:Cocompletion-H-projR}.

The dual $H^*$ provides a concrete model for $A_\cC$. To describe its structure, we identify $\wcC^\op$ with $\lMod{H^{\cop}}_R$, i.e., left modules over the Hopf algebra $H^\cop$ with opposite coproduct $\Delta^\cop(h)=h_{(2)}\otimes h_{(1)}$, where the coproduct of $H$ is $\Delta(h)=h_{(1)}\otimes h_{(2)}$, using Sweedler's notation.

The component maps of the coend are now given by 
\begin{equation}\label{eq:componentmapsH}
i_V\colon V^*\boxtimes V\to H^*, \quad f\boxtimes v \mapsto \left(h\mapsto f(h\cdot v)\right),
\end{equation}
for any left $H$-module $V$.
One checks that $i_V$ is a morphism of $H^\cop\otimes H$-modules via the action 
\begin{align}\label{eq:Hstarbimodule}
(k_1\otimes k_2)\cdot g(h)=g(S(k_1)h k_2), \qquad \forall g\in H^*, h,k_1,k_2\in H,
\end{align}
on $H^*$ and
$$(k_1\otimes k_2)\cdot (f\boxtimes v)=(k_1\cdot f)\boxtimes (k_2\cdot v)$$
on $V^*\boxtimes V$. 

The product $m$ on $H^*$ is given by 
$$(\alpha \beta)(h)=m(\alpha\otimes \beta)(h)=\alpha(h_{(1)})\beta(h_{(2)}).$$
One checks that this product indeed is a morphism 
$$m\colon H^*\otimes H^*\to H^*$$
of left $H^\cop\otimes H$-modules, with respect to the action specified in Equation \eqref{eq:Hstarbimodule}. The unit of $H^*$ is given by $1_{H^*}=\varepsilon$, using the counit $\varepsilon\colon H\to R$ of $H$ and it follows from $1_{H^*}=\varepsilon$ being a morphism of algebras that $1_{H^*}$ is an $H^\cop\otimes H$-invariant element.
\end{ex}

\begin{lem}\label{lem:finiteHopf}
Assume that $H$ is a Hopf algebra in $\flmod{R}$ and consider $\cC=\flRmod{H}{R}$. Then $A_\cC\cong H^*$ are isomorphic as algebras in $\cC^\op\boxtimes \cC$.
\end{lem}
\begin{proof}
The above considerations in Example \ref{ex:finiteHopf} show that there is a homomorphism of $H^\cop\otimes H$-module algebras $A_\cC\to H^*$. First note that the algebra $H^*$ is generated by elements in the image of the component maps $i_V$. Namely, take $V=H$ to be the regular $H$-module. Then $i_H(1\otimes f)=(h\mapsto f(h))=f$, for any $f\in H^*$. Now, one shows that $H^*$ satisfies the universal property of the coend together with the component maps $i_V$. 
\end{proof}

\begin{ex}\label{ex:filteredHopf}
Consider the category $\cC=\flRmod{H}{R}$ for a Hopf algebra in $\lMod{R}$ which has a filtration $H=\cup_{i\in \mN}H_i$, where the $H_i$ are finitely-generated projective as $R$-modules such that $m\colon H_i\otimes_R H_j \to H_{i+j}$. In this case, $A_\cC$ can be identified with the finitary dual $H^\circ$ of $H$. Here, $H^\circ$ is the subalgebra of $\Hom_R(H,R)$ which is the directed union of images of the component maps $i_V$ from Equation \eqref{eq:componentmapsH}. Each element of $H^\circ$ can be expressed as a \emph{coordinate function}
\begin{equation}\label{coordinate function}
c_{f,v}^V:=i_V(f\otimes v)\in H^\circ,
\end{equation}
cf., e.g., \cite{BG}*{Chapter I.7}.
The product $\cdot$ and unit $1$ of $H^\circ$ are given by 
\begin{align}
c_{f,v}^V\cdot c_{g,w}^W=c_{g\otimes f,v\otimes w}^{V\otimes W},\qquad 1=c_{\Id_{\one},1}^{\one},\label{eq-product-cVs}
\end{align}
where $g\otimes f$ is regarded as an element in $W^*\otimes V^*=(V\otimes W)^*$.

The left $H^\cop\otimes H$-module structure on $H^\circ$ is given by 
\begin{align}
(k_1\otimes k_2)\cdot c_{f,v}^V=c_{k_1\cdot f,k_2\cdot v}^V.
\end{align}
\end{ex}

When studying the representation theory of quantum groups, is sometimes convenient to consider subcategories $\cC\subseteq \flRmod{H}{R}$. In this case, $A_\cC$ is the subalgebra of $H^\circ$ generated by all coordinate functions on objects $V$ of $\cC$. The following lemma will be used in such contexts.

In the following, a set of \emph{tensor generators} for an $R$-linear tensor category $\cC$ is a set of objects $\left\{X_i\right\}_{i\in I}$ such that every object of $\cC$ is in the \emph{Karoubian envelope} of the full subcategory on tensor products of the ${X_i}_{i\in I}$. That is, every object of $\cC$ can be obtained by taking tensor products, finite direct sums, duals, and direct summands of the objects $\left\{X_i\right\}_{i\in I}$. 

\begin{lem}\label{lem:tensorgenerators}
Let $H$ be a Hopf algebra $H$ as in Example~\ref{ex:filteredHopf}. Let $\cC$ be full tensor subcategory of $\flRmod{H}{R}$ with a set of tensor generators $\left\{X_i\right\}_{i\in I}$. Then $A_\cC$ is isomorphic as an algebra to the subalgebra of $H^\circ$ generated by coordinate functions $c_{f,v}^{X_i}$, varying over $i\in I$, $f\in X_i^*, v\in X_i$.
\end{lem}

Forgetting the $H^\cop\otimes H$-module structure on $A_\cC$ we obtain a Hopf algebra structure for $A_\cC$ as an object of $\lMod{R}$. 
\begin{lem}\label{lem:AC-Hopf}
In the setup of Lemma \ref{lem:tensorgenerators}, the subalgebra $A_\cC$ of $H^\circ$ is an $R$-sub-bialgebra with coproduct and counit given by 
\begin{gather}
\Delta(c_{f,v}^V)=\sum_{\alpha}c_{f,v_\alpha}^V\otimes c_{f^\alpha,v}^V,\qquad \varepsilon(c_{f,v}^V)=f(v)=\ev_V(f\otimes v),
\end{gather}
where $\left\{v_\alpha\right\}\subset V$ and $\left\{f_\alpha\right\}\subset V^*$ are dual bases for $V$ an object of $\cC$.

If $\cC$ is a rigid category, then $A_\cC$ is a Hopf subalgebra with antipode given by
\begin{gather}\label{eq:antipode-matrix-coeff}
S(c_{f,v}^V)=c_{\tau_V(v),f}^{V^*}\; ,
\end{gather}
where $\tau_V$ denotes the pivotal structure of $\flmod{R}$. Moreover, if $H$ is quasitriangular with R-matrix $R=R^{(1)}\otimes R^{(2)}\in H\otimes H$, then $H^\circ$ is dual quasitriangular with dual R-matrix
\begin{equation}\label{eq:R-matrix-mat-coeff}
    R(c_{f,v}\otimes c_{g,w})=f(R^{(1)}\cdot v)g(R^{(2)}\cdot w).
\end{equation}
\end{lem}

Lemma~\ref{lem:tensorgenerators} shows that to find the generators for $A_\cC$ one can take the coordinate functions $x^i_j=c_{f_i,v_j}$ obtained from  basis vectors $v_j$ and dual basis vectors $f_i$ of the tensor generators. In order to determine the relations among the generators $x^i_j$, we utilize pairings similarly to, e.g.,  \cite{Majid95book}. Assume that we have found an algebra $A$ with generators $\ov{x}^i_j$ and a surjective homomorphism of algebras 
$$A\twoheadrightarrow A_\cC\subseteq H^\circ,$$
such that $\ov{x}^i_j\mapsto x^i_j$.  
Then $A\cong A_\cC$ if and only if the composite map $\phi\colon A\to H^\circ$ is injective. This, in turn, is equivalent to the Hopf algebra pairing 
$$\langle -,-\rangle\colon A\otimes H\to R,\qquad \ov{x}^i_j\otimes h\mapsto \phi(\ov{x}^i_j)(h)=f^i(h\cdot v_j),$$
having trivial left radical. This pairing is extended to products of the generators in the left component via
\begin{equation}\label{eq:extend-pairing}
\langle ab,h \rangle = \langle a,h_{(1)} \rangle\langle b,h_{(2)} \rangle, \qquad \langle 1,h \rangle=\varepsilon(h).
\end{equation}

In the finite case, $A_\cC$ can be identified with the dual of $H$ as a Hopf algebra.

\begin{lem}\label{lem:finiteHopf-AC}
Assume that $H$ is a Hopf algebra in $\flmod{R}$ and consider $\cC=\flRmod{H}{R}$. Then $A_\cC\cong H^*$ are isomorphic as Hopf algebras in $\flmod{R}$.
\end{lem}
\begin{proof}
This follows from Lemmas~\ref{lem:finiteHopf} and \ref{lem:AC-Hopf}.
\end{proof}

The following important examples will be generalized to the rank $n$-case, and more general base rings, in Section~\ref{sec:qgroup-defns}.

\begin{ex}[$O_q(SL_2)$]\label{ex:sl2}
Let $\Bbbk$ be a field of characteristic zero. Consider the quantum group 
$U_q(\mathfrak{sl}_2)$, for $q\in \Bbbk\setminus \{0,1,-1\}$, generated as a $\Bbbk$-algebra by 
$E,F,K,K^{-1}$ subject to the relations 
\begin{gather*}
KK^{-1}=K^{-1}K, \quad KEK^{-1}=q^2E,\quad  KFK^{-1}=q^{-2}F, \quad  [E,F]=\frac{K-K^{-1}}{q-q^{-1}},
\end{gather*}
with coproduct determined by 
\begin{gather*}
\Delta(K)=K\otimes K, \quad \Delta(E)=E\otimes K + 1\otimes E, \quad \Delta(F)=F\otimes 1 +K^{-1}\otimes F.
\end{gather*}
The presentation here comes from
\cite{KS}*{Section~3.1}.

Consider  the two-dimensional $U_q(\mathfrak{sl}_2)$-module $V=\Bbbk\langle v_1,v_2\rangle$, with actions 
\begin{gather*}
K\cdot v_1=qv_1, \quad K\cdot v_2=q^{-1}v_2, \\
F\cdot v_1=v_2, \quad F\cdot v_2=0
\\
E\cdot v_1=0, \quad E\cdot v_2=v_1.
\end{gather*}
This module (denoted by $T_{1/2}$ in \cite{KS}*{Section~3.2}) is a tensor generator for the tensor subcategory $\cC=\cC_q(\mathfrak{sl}_2)$ of $\flRmod{U_q(\mathfrak{sl}_2)}{\Bbbk}$ given by so-called \emph{type I modules}, cf. \cite{BG}*{Section I.6.12}. 
Thus, by Lemma \ref{lem:tensorgenerators}, the algebra $A_\cC$ is generated as an algebra by the image of $i_V\colon V^*\boxtimes V\to H^*$. 
Denoting the dual basis by $\left\{f_1,f_2\right\}$, we set
\begin{align}x^i_{j}:=i_V(f_i\otimes v_j), \text{ and } a:=x^1_{1}, \quad b:=x^1_{2}, \quad c:=x^2_{1}, \quad d:=x^2_{2}.\label{eq:abcd-def}
\end{align}
We define $O_q(SL_2)$ to be the Hopf algebra generated by $a,b,c,d$ as an algebra, subject to relations
\begin{gather}
\begin{gathered}
ab=q^{-1}ba, \quad
ac=q^{-1}ca,\quad 
bd=q^{-1}db,\quad cd=q^{-1}dc, \quad 
bc = cb, \\
ad - da = \left(q^{-1} - q\right)bc,
\end{gathered}\label{eq:AR-rels}
\end{gather}
plus the additional relation 
\begin{gather}
ad-q^{-1}bc=1.\label{eq:qdet-rel}
\end{gather}
This algebra appears in \cite{Majid95book}*{Proposition~4.2.6}. Note that the algebra generated by $a,b,c,d$ satisfying Equation \eqref{eq:AR-rels} but not Equation \eqref{eq:qdet-rel} is the algebra $O_q(M_2)$ from \cite{Majid95book}*{Example~4.2.5}. That is, it is obtained as the algebra $A(R)$ for the matrix solution
$$R=q^{-1/2}\begin{pmatrix}
q&0&0&0\\0&1&q-q^{-1}&0\\0&0&1&0\\0&0&0&q
\end{pmatrix}
$$
of the quantum Yang--Baxter equation.
The above discussion implies that  $O_q(SL_2)$ is an algebra object in $\wcC^\op\boxtimes \wcC\simeq \lMod{U_q(\mathfrak{sl}_2)^\cop\otimes U_q(\mathfrak{sl}_2)}_\Bbbk$. In particular, $O_q(SL_2)$ has a $U_q(\mathfrak{sl}_2)^\cop\otimes U_q(\mathfrak{sl}_2)$-module structure given by
\begin{gather}
(K\otimes 1)a=q^{-1}a,\quad (K\otimes 1) b=q^{-1}b,\quad (K\otimes 1) c=qc,\quad (K\otimes 1) d=q d,\\
(1\otimes K)a=qa,\quad (1\otimes K) b=q^{-1}b,\quad (1\otimes K) c=qc,\quad (1\otimes K) d=q^{-1}d,\\
(E\otimes 1)a=-qc,\quad (E\otimes 1) b=-qd,\quad (E\otimes 1) c=0,\quad (E\otimes 1) d=0,\\
(1\otimes E)a=0,\quad (1\otimes E) b=a,\quad (1\otimes E) c=0,\quad (1\otimes E) d=c,\\
(F\otimes 1)a=0,\quad (F\otimes 1) b=0,\quad (F\otimes 1) c=-q^{-1}a,\quad (F\otimes 1) d=-q^{-1}b,\\
(1\otimes F)a=b,\quad (1\otimes F) b=0,\quad (1\otimes F) c=d,\quad (1\otimes F) d=0.
\end{gather}
\end{ex}

\begin{lem}
Let $\cC=\cC_q(\mathfrak{sl}_2)$ and assume that $q$ is not a root of unity in $\Bbbk$. Then $A_\cC$ is isomorphic to $O_q(SL_2)$ as an algebra in $\wcC^\op\boxtimes \wcC$, i.e, as a $U_q(\mathfrak{sl}_2)^\cop\otimes U_q(\mathfrak{sl}_2)$-module algebra. 
\end{lem}
\begin{proof}
Since $V$ is a tensor generator for the category $\mathcal{C}_q(\mathfrak{sl}_2)$, Lemma \ref{lem:tensorgenerators} shows that $a,b,c,d$ generate $A_\cC$. By the discussion after Lemma \ref{lem:AC-Hopf}, it remains to show that the pairing $\langle ~,~\rangle\colon O_q(SL_2)\otimes U_q(\mathfrak{sl}_2)\to \Bbbk$ with the only non-zero values on generators given by 
\begin{gather}
\begin{gathered}
\inner{a,K}=c^V_{f_1,v_1}(K)=q, \qquad \inner{d,K}=c^V_{f_2,v_2}(K)=q^{-1},\\
\inner{a,1}=1, \qquad \inner{d,1}=1,\\
\inner{b,E}=c^V_{f_1,v_2}(E)=1, \qquad \inner{c,F}=c^V_{f_2,v_1}(F)=1
\end{gathered} \label{pairings-OUsl2}
\end{gather}
is non-degenerate. 
One now checks that all relations from \eqref{eq:AR-rels}--\eqref{eq:qdet-rel} are in the left radical of the pairing. These are all relations as the resulting pairing on $O_q(SL_2)\otimes U_q(\mathfrak{sl}_2)$ is non-degenerated provided that $q$ is \emph{not} a root of unity, see e.g.  \cite{KS}*{Section~4.1}.
\end{proof}

Following, e.g., \cite{KS}*{Section~6.1.2}, consider the quantum group 
$U_q(\mathfrak{gl}_2)$, for $q\in \Bbbk\setminus \{0,1,-1\}$, generated as a $\Bbbk$-algebra by 
$E,F,J_i^{\pm}$, for $i=1,2$, subject to the relations 
\begin{gather*}
J_iJ_i^{-1}=J_i^{-1}J_i=1, \quad J_iEJ_i^{-1}=q^{(-1)^i}E,\quad  J_iFJ_i^{-1}=q^{(-1)^{i+1}}F, \\  [E,F]=\frac{J_1J_{2}^{-1}-J_1^{-1}J_2}{q-q^{-1}},
\end{gather*}
with coproduct determined by 
\begin{gather*}
\Delta(J_i)=J_i\otimes J_i, \quad \Delta(E)=E\otimes J_1J_2^{-1} + 1\otimes E, \quad \Delta(F)=F\otimes 1 +J_1^{-1}J_2\otimes F.
\end{gather*}
We note that $U_q(\mathfrak{sl}_2)$ is the Hopf subalgebra of $U_q(\mathfrak{gl}_2)$ generated by $E,F,J_1J_2^{-1}$.
Example \ref{ex:sl2} can, equally well, be extended to the quantum group of $\mathfrak{gl}_2$. 

\begin{ex}[$O_q(GL_2)$]\label{ex:gl2}
Similarly to \cite{Zh} (or \cite{KS}*{Sections~8.4, 9.4}) consider the type I simple highest weight module $L_{\epsilon_1}$ of $U_q(\mathfrak{gl}_2)$. We denote the highest weight vector $v_1$. We denote $E:=E_1$, $F:=F_1$, and obtain the action 
\begin{align*}
J_1\cdot v_1&=qv_1, &J_2\cdot v_1&=v_1, &
E\cdot v_1&=0, & F\cdot v_1&=v_2,\\
J_1\cdot v_2&=v_2, & J_2\cdot v_2&=qv_2,&
E\cdot v_2&=v_1, & F\cdot v_2&=0.
\end{align*}
Note that the restriction of $L_{\epsilon_1}$ to $U_q(\mathfrak{sl}_2)$ recovers the representation $V$ from Example \ref{ex:sl2}. Hence, we similarly define $a,b,c,d$ as in \eqref{eq:abcd-def}. It follows that the resulting subalgebra $A$ of $U_q(\mathfrak{gl}_2)^{\circ}$ satisfies the relations \eqref{eq:AR-rels} but not \eqref{eq:qdet-rel}.
However, the quantum determinant 
\begin{equation}\label{eq:qdet2}
\dq{2} :=ad-q^{-1}bc
\end{equation}
 is invertible as an element in $U_q(\mathfrak{gl}_2)^{\circ}$. This follows since $\dq{2}$ is grouplike in $U_q(\mathfrak{gl}_2)^{\circ}$ by \cite{KS}*{Section~9.2.2}, i.e., defines a character on $U_q(\mathfrak{gl}_2)$.

We define $O_q(GL_2)$ as the localization
$$O_q(GL_2)=O_q(M_2)[\dq{2}]^{-1} = O_q(M_2)[t]/(t\dq{2}-1).$$
A pairing $\inner{-,-}\colon O_q(GL_2)\otimes U_q(\mathfrak{gl}_2)\to \Bbbk$ is given on generators by the non-zero values on generators
\begin{align*}
\inner{a,J_1}&=c^V_{f_1,v_1}(J_1)=q, & \inner{a,J_2}&=1,\\ 
\inner{d,J_1}&=1 , & \inner{d,J_2}&=q,\\
\inner{a,1}&=1, & \inner{d,1}&=1,\\
\inner{b,E}&=c^V_{f_1,v_2}(E)=f_1(Ev_2)=1, & \inner{c,F}&=1,
\end{align*}
and extended via Equation \eqref{eq:extend-pairing} to products of generators and hence to the localization.

Under restriction along the inclusion of $U_q(\mathfrak{sl}_2)\hookrightarrow U_q(\mathfrak{gl}_2)$ in the right component, the pairing $\inner{-,-}$ becomes degenerate with left radical given by the relation $\dq{2}=1$. Hence, it factors over the non-degenerate pairing from Example \ref{ex:sl2}. 
\end{ex}

\begin{lem} Let $\cC=\cC_q(\mathfrak{gl}_2)$. Then $A_\cC$ is isomorphic to $O_q(GL_2)$ as an algebra in $\cC^\op\boxtimes \cC$. 
\end{lem}
\begin{proof}
By \cite{Tak}*{4.4. Theorem}, setting $R=\Bbbk$, (see also \cite{KS}*{Theorem~18, Section~9.4}), the above pairing $\inner{-,-}\colon O_q(GL_2)\otimes U_q(\mathfrak{gl}_2)\to \Bbbk$ is non-degenerate. This implies the claim.
\end{proof}

\subsection{The algebra \texorpdfstring{$A_\cC$}{AC} as an internal hom object}

We now return to the general setting when $\cC$ is an $R$-linear braided tensor category with duals. We will see that the algebra $A_\cC$ in $\wcC^\op\boxtimes \wcC$ can, equivalently, be described as an internal hom object. For this, consider the functor 
\begin{align}\label{functor-T}
T\colon \cC^\op\boxtimes \cC\to \cC, \quad X\boxtimes Y\mapsto X\otimes Y,
\end{align}
for the exterior tensor product (Kelly tensor product) $\boxtimes=\boxtimes^{\mathrm{f}}_R$ (see Appendix \ref{sec:tensor-cocompletion}).
This functor $T$, 
together with the structural isomorphisms,
$$\xymatrix{
T((X_1\boxtimes Y_1)\otimes (X_2\boxtimes Y_2))=X_2\otimes X_1 \otimes Y_1\otimes Y_2\ar@<-10pt>[d]^{\mu^T_{X_1\boxtimes Y_1,X_2\boxtimes Y_2}:= \Psi_{X_2,X_1\otimes Y_1}\otimes \Id}\\
T(X_1\boxtimes Y_1)\otimes T(X_2\boxtimes Y_2)=X_1\otimes Y_1 \otimes X_2\otimes Y_2
}
$$
is a strong monoidal functor. Now, $T$ has a right adjoint $R\colon \cC\to \cC^\op\boxtimes \cC$ if and only if the internal homs $\uHom(\one,Z)$ exist for all objects $Z$ in $\cC$, cf. e.g., \cite{EGNO}*{Example~7.9.10}. More generally, the internal hom exists as a functor 
$$R=\uHom(\one, -)\colon \cC\to \wcC^\op\boxtimes\wcC.$$ Indeed, this right adjoint $R(Y)$ is given by the object
$$R(Y)\colon (\cC^\op\boxtimes \cC)^\vee\to \lMod{R},\qquad V\boxtimes W\mapsto \Hom_{\cC}(V\otimes W,Y)$$
in $\wcC^\op\boxtimes \wcC$.
As the functor $\widehat{T}\colon \wcC^\op\boxtimes \wcC\to \wcC$ is, by assumption, cocontinuous, $R(Y)$ is a  left  exact contravariant functor  and hence defines an object in $\wcC^\op\boxtimes \wcC$.

\smallskip

Note that the above tensor functor $T$ turns $\cC$ into a left $\cC^\op\boxtimes \cC$-module category via the categorical action given by 
$$(\cC^\op\boxtimes \cC)\times \cC\to \cC, \quad (X\boxtimes Y,Z)\mapsto T(X\boxtimes Y)\otimes Z=X\otimes Y\otimes Z.$$
This way, one shows that the coend $A_\cC$ is equivalent to an internal hom object.

\begin{lem}
The algebra $A_\cC$ is isomorphic as an algebra in $\wcC^\op\boxtimes \wcC$ to the internal hom object $\uEnd_{\cC^\op\boxtimes \cC}(\one)$. 
\end{lem}
\begin{proof}
The internal hom $\uEnd(\one)=\uEnd_{\cC^\op\boxtimes \cC}(\one)$ satisfies the universal property that there exists a natural isomorphism
$$\xymatrix{
\Hom_{\wcC^\op\boxtimes \wcC}(Y\boxtimes Z, \uEnd(\one))\ar@<5pt>[rr]^-{\alpha_{Y,Z}}&&\ar@<5pt>[ll]^-{\beta_{Y,Z}}\Hom_\cC(Y\otimes Z,\one).
}$$
We show that $A_\cC$ satisfies this universal property.

To construct $\alpha_{Y,Z}$ we use that $Y,Z\in \cC$ are compact objects in the cocompletion $\wcC$ see Section~\ref{sec:appendix} and hence $\Hom_{\wcC^\op\boxtimes \wcC}(Y\boxtimes Z, -)$ commutes with filtered (or, equivalently, directed) colimits. Thus, $\alpha_{Y,Z}$ is the colimit of the component maps
$$\Hom_{\cC^\op\boxtimes \cC}(Y\boxtimes Z, X^*\boxtimes X)\to \Hom_\cC(Y\otimes Z,\one), \quad f\boxtimes g\mapsto \ev_X(f\otimes g).$$

Using the natural isomorphisms 
$$\Hom_{\cC}(Y\otimes Z,\one)\cong \Hom_{\cC}(Z,{}^*Y),\quad  f\mapsto f'=(\Id_{{}^*Y}\otimes f)(\coev^r_{Y}\otimes \Id_Z),$$ 
obtained from the right dual ${}^*Y$, we define $\beta_{Y,Z}(f)$, for $f\colon Y\otimes Z\to \one$,  as the diagonal composition in the commutative diagram
$$
\xymatrix{
& Z^*\boxtimes Z\ar[dr]^{i_Z}&\\
Y\boxtimes Z=({}^*Y)^*\boxtimes Z\ar[ur]^{(f')^*\boxtimes \Id}\ar[dr]_{\Id\boxtimes f'}\ar[rr]^-{\beta_{Y,Z}(f)}&&A_\cC.\\
& Y\boxtimes {}^*Y\ar[ur]_{i_{{}^*Y}}&
}
$$
One checks that $\alpha$ and $\beta$ indeed define mutually inverse natural transformations. 

A general object $B$ in $\cC^\op\boxtimes \cC$ is a finite colimit $B=\colim X_i\boxtimes Y_i$  \cite{Kel}*{Theorem~5.35}. A morphism in $\Hom_{\wcC^\op\boxtimes \wcC}(B,A_\cC)$ is determined by the data of a compatible system of morphisms of morphisms $X_i\boxtimes Y_i\to A_\cC$ which correspond to morphisms $X_i\otimes Y_i\to \one$ using the natural isomorphisms $\alpha_{X_i,Y_i}$ and $\beta_{X_i,Y_i}$. As $T$ is colimit preserving, this system of morphisms determines a unique morphism $T(B)\to \one$.  Thus, we have natural isomorphisms 
$$\xymatrix{
\Hom_{\wcC^\op\boxtimes \wcC}(B,A_\cC)\ar@/^/[rr]^{\alpha_B}&&\ar@/^/[ll]^{\beta_B}\Hom_\cC(T(B),\one)
}$$
and $A_\cC$ satisfies the universal property of $\uEnd(\one)$.

It remains to check that the algebra structure induced on $A_\cC$ as in Definition \ref{defn:AC} coincides with the one induced on $\uEnd(\one)$ similarly to \cite{EGNO}*{Section~7.9}. Note that any morphism $f\colon A_\cC\to B$ in $\wcC^\op\boxtimes \wcC$  is determined by a compatible diagram of morphisms $f_X\colon X^*\boxtimes X\to B$ such that $f_X=f\circ i_X$. 
Clearly, $\Id_X\circ i_X=i_X$. Thus, we obtain a morphism $\ev_{\one, \one}$ determined by the commutative diagrams
$$\xymatrix{
X^*\otimes X\ar[dr]^{T(i_X)}\ar[rr]^{\ev_X}&&\one\\
&T(A_\cC)\ar[ru]^{\ev_{\one,\one}}&
}$$
for all objects $X$ in $\cC$ and $\alpha_B(f)=\ev_{\one,\one}(f)$ for all $f$.

Denote by 
$$m'\colon A_\cC\otimes A_\cC\to A_\cC$$
the multiplication map induced by the universal property of the internal hom $A_\cC=\uEnd(\one)$. It is determined as the unique morphism making the diagram
$$
\xymatrix{
T(A_\cC\otimes A_\cC)\ar[d]^{T(m')}\ar[rr]^-{\mu^T_{A_\cC,A_\cC}}&&T(A_\cC)\otimes T(A_\cC)\ar[rr]^{\Id\otimes \ev_{\one,\one}}&&T(A_\cC)\ar[d]^{\ev_{\one,\one}}\\
T(A_\cC)\ar[rrrr]^{\ev_{\one,\one}}&&&&\one
}
$$
commute. By precomposing with $$T((X^*\boxtimes X)\otimes (Y^*\boxtimes Y))=X^*\otimes X\otimes Y^*\otimes Y\xrightarrow{T(i_X\otimes i_Y)} T(A_\cC\otimes A_\cC)$$
we see that the multiplication $m$
 defined before Definition \ref{defn:AC} has this property from the commutative diagram below.
$$
\xymatrix{
&&Y^*\otimes X^*\otimes X\otimes Y\ar@/_/[dll]|-{T(i_X\otimes i_Y)}\ar@/_2pc/[dddll]|-{T(i_{X\otimes Y)}}\ar@/_4pc/[ddd]|-{\ev_{X\otimes Y}}\ar[d]^{\mu^T}\ar@/^/[drr]|-{T(i_X\otimes i_Y)}&&\\
T(A_\cC\otimes A_\cC)\ar[dd]_{T(m)}&&X^*\otimes X\otimes Y^*\otimes Y\ar@/^/[rrd]|-{T(i_X)\otimes T(i_Y)}\ar[d]^{\Id\otimes \ev_Y}&& T(A_\cC\otimes A_\cC)\ar[d]^{\mu^T}\\
&&X^*\otimes X\ar[d]^{\ev_X}\ar@/^/[rrd]|-{T(i_X)}&&T(A_\cC)\otimes T(A_\cC)\ar[d]^{\Id\otimes \ev_{\one,\one}}\\
T(A_\cC)\ar[rr]^{\ev_{\one,\one}}&&\one&& T(A_\cC)\ar[ll]_{\ev_{\one,\one}}.
}
$$
Here, the left top diagram commutes by definition of $m$, and we have used that 
$$\ev_{X\otimes Y}=\ev_Y(\Id\otimes \ev_X)\mu^T_{X^*\boxtimes X,Y^*\boxtimes Y}$$
identifying $(X\otimes Y)^*=Y^*\otimes X^*$. Thus, $m=m'$. Finally, for the unit $u\colon \one \to A_\cC$ from before Definition \ref{defn:AC} satisfies $T(u)=T(i_\one)$ and is hence equal to the unit $u'\colon \one \to A_\cC=\uEnd(\one)$ obtained from the universal property of $\uEnd(\one)$.
\end{proof}

In particular, $A_\cC$ is isomorphic to $R(\one)$ with its natural structure as an algebra object in $\wcC^\op\boxtimes \wcC$ obtained using the lax monoidal structure of $R$. We remark that $R(\one)=\uEnd_{\cC^\op\boxtimes\cC}(\one)$ is the \emph{canonical algebra} of \cite{EGNO}*{Section 7.18}.
 In some cases, it is not necessary to pass to cocompletions and $A_\cC$ is an object of $\cC$. 

\begin{lem}\label{lem:finite-case1}
If $\cC$ is a $\Bbbk$-linear finite tensor category in the sense of \cite{EGNO}, then $A_\cC$ exists as an object in $\cC^\op\boxtimes \cC$.
\end{lem}
\begin{proof}
This follows from the isomorphism $A_\cC\cong \uEnd_{\cC^\op\boxtimes\cC}(\one)$ and existence of internal homs in $\Bbbk$-linear finite tensor categories \cite{EGNO}*{Section~7.9}.
\end{proof}

\subsection{The reflection equation algebra \texorpdfstring{$B_\cC$}{BC}}\label{sec:BC}
We will now introduce the main object of interest in this paper, the reflection equation algebra $B_\cC$, as the image of the FRT algebra $A_\cC$ under the tensor product functor. This approach follows \cite{EGNO}*{Exercise~8.25.7} and agrees with the definition of $B_\cC$ as a coend.

The tensor functor $T\colon \cC^\op\boxtimes \cC\to \cC$ from \eqref{functor-T} has an extension 
$$\widehat{T}\colon \wcC^\op\boxtimes \wcC\to \wcC$$
to cocompletions which preserves colimits, see Proposition~\ref{prop:wcC}. Thus there is a canonical isomorphism 
$$\widehat{T}(A_\cC)=\widehat{T}\bigg(\int^{X\in \cC}X^*\boxtimes X\bigg)\cong \int^{X\in \cC}X^*\otimes X$$
of coends, which is an isomorphism of algebras in $\wcC$.

\begin{defn}\label{defn:BC}
We denote 
$$B_\cC:=\int^{X\in \cC}X^*\otimes X\,\in\, \wcC $$
and call $B_\cC$ the \emph{reflection equation algebra} associated to $\cC$.
\end{defn}

We will now show how $B_\cC$ satisfies the universal properties postulated in \cite{Majid95book}*{Chapter~9}, cf. \cite{Sch}*{2.1.9. Lemma}.

\begin{lem}\label{lem:representability1}
For any object $Y$ of $\cC$, there are natural isomorphisms
$$\Hom_{\wcC}(B_\cC,Y)\cong \Nat(\Id_\cC,\Id_\cC\otimes Y).$$
\end{lem}
\begin{proof}
Denote the component map of the coend by
$$d_Y\colon Y^*\otimes Y\to \int^{X\in \cC} X^*\otimes X=B_\cC.$$
Then, by the universal property, we have an identification of $\Hom_{\wcC}(B_\cC,Y)$ with the set 
$$
\Set{(g_X\colon X^*\otimes X\to Y)_{X\in \cC}\;\vert \; g_X(f^*\otimes \Id_X)=g_Z(\Id_{Z^*}\otimes f),\, \forall f\colon X\to Z}.
$$
Now, using the coevaluation map $\coev_X\colon \one \to X\otimes X^*$, we define 
$$h_X:=(\Id_X\otimes g_X)(\coev_X\otimes \Id_X)\colon X\to X\otimes Y,$$
and obtain a natural transformation $h=(h_X)_{X\in \cC}\colon \Id_\cC\to \Id_\cC\otimes Y$. Conversely, given such a natural transformation $h\colon \Id_\cC\to \Id_\cC\otimes Y$, one uses $\ev_X\colon X^*\otimes X\to \one$ to obtain an element $(g_X)_{X\in \cC}$ in the colimit and hence a morphism from the coend.
\end{proof}

In other words, $B_\cC$ represents the functor
$$\cC\to \lMod{R}, \quad Y \mapsto \Nat(\Id,\Id\otimes Y).$$
This condition, together with the higher representability, is imposed in \cite{Majid95book}*{Section~9.4}. Here, we consider the functors 
$$\Id_\cC^{\otimes n}\otimes Y \colon \cC^{\boxtimes n}\to \cC, \quad X_1\boxtimes \ldots \boxtimes X_n\mapsto X_1\otimes \ldots \otimes X_n\otimes Y.$$

\begin{lem}\label{lem:representabilityn}
For any object $Y$ of $\cC$ and $n\geq 0$, there are natural isomorphisms
$$\Hom_{\wcC}(B_\cC^{\otimes n},Y)\cong \Nat(\Id_\cC^{\otimes n},\Id_\cC^{\otimes n}\otimes Y)$$
of functors $\cC\to \lMod{R}$.
\end{lem}
\begin{proof}
Since $\otimes$ preserves colimits, $B_\cC^{\otimes n}$ is a colimit with a universal property with respect to the component maps 
$$d_{Y_1}\otimes \ldots d_{Y_n}\colon Y_1^*\otimes Y_1\otimes\ldots \otimes Y_n^*\otimes Y_n\to B_\cC^{\otimes n}.$$
Using the braiding of $\cC$ one obtains a natural isomorphism
$$\Hom_\cC(Y_1^*\otimes Y_1\otimes\ldots \otimes Y_n^*\otimes Y_n,Y)\cong \Hom_\cC(Y_1\otimes\ldots \otimes Y_n,Y_1\otimes \ldots \otimes Y_n\otimes Y).$$
These isomorphisms associate to a morphism $B_\cC^{\otimes n}\to Y$ a natural transformation $\Id^{\otimes n}_\cC\to \Id^{\otimes n}_\cC\otimes Y$, component-wise. Now, proceeding as in Lemma~\ref{lem:representability1} yields the claimed natural isomorphism.
\end{proof}

Note that $B_\cC$ is an algebra object in $\wcC$. In fact, Majid \cite{Maj1}, \cite{Majid95book}*{Section~9.4.2} showed that $B_\cC$ is a (quasitriangular) Hopf algebra object in the cocompletion $\wcC$. To display its Hopf algebra structure, consider:
\begin{enumerate}
\item the natural transformation $$\delta\colon \Id_\cC\to \Id_\cC\otimes B_\cC$$
corresponding to the identity $\Id\colon B_\cC\to B_\cC$ in Lemma \ref{lem:representability1};
\item the natural transformation 
$\mu\colon \Id_\cC\otimes \Id_\cC\to \Id_\cC\otimes \Id_\cC\otimes B_\cC$
given by 
$$\mu_{X\otimes Y}=\delta_{X\otimes Y};$$
\item the natural transformation $\delta^*\colon \Id_\cC\to \Id_\cC\otimes B_\cC$ given by 
$$\delta^*_X=(\Id_X\otimes \ev_X)(\Id_{X\otimes X^*}\otimes \Psi_{B_\cC,X})(\Id\otimes \delta_{X^*}\otimes \Id)(\coev_X\otimes \Id_X)$$
\end{enumerate}

\begin{prop}[Majid]\label{prop:coend-Hopf}
Assume that $\cC$ is a ribbon category. Then
$B_\cC$ is a Hopf algebra in $\wcC$ with product $\mu_{B_\cC}$, unit $u_{B_\cC}$, coproduct $\Delta_{B_\cC}$, unit $\varepsilon_{B_\cC}$, and antipode $S_{B_\cC}$ given component-wise by 
\begin{align*}
\mu_{B_\cC} (d_{X}\otimes d_Y)&=d_{X\otimes Y}(\Psi_{X^*\otimes X,Y^*}\otimes \Id_Y)\colon (X^*\otimes X)\otimes (Y^*\otimes Y)\to B_{\cC},\\
u_{B_\cC}&=d_\one\colon \one \to B_\cC,\\
\Delta_{B_\cC} d_X&=(d_X\otimes d_X)(\Id_{X^*}\otimes \coev_{X}\otimes \Id_X)\colon X^*\otimes X\to B_\cC\otimes B_\cC,\\
\varepsilon_{B_\cC} d_X&=\coev_X\colon X^*\otimes X\to \one,\\
S_{B_\cC} d_X&=d_{X^*}(\phi_X\otimes \Id_{X^*})\Psi_{X^*,X}\colon X^*\otimes X\to B_\cC,
\end{align*}
where $\phi_X=(\ev_X\otimes \Id_{X^{**}})(\Id_{X^*}\otimes \Psi_{X^{**},X})(\coev_{X^*}\otimes \Id_X)$ is the Drinfeld isomorphism.
\end{prop}
\begin{proof}
Using the ideas of braided reconstruction theory \cite{Majid95book}*{Chapter~9}, we define the Hopf algebra structure by the
\begin{itemize}
\item product given by the map $\mu_{B_\cC}\colon B_\cC\otimes B_\cC\to B_\cC$ corresponding to $\mu$ under Lemma \ref{lem:representabilityn},
\item unit $u_{B_\cC}=\delta_{\one}\colon \one\to B_\cC$,
\item coproduct $\Delta_{B_\cC}=\delta_{B_\cC}$,
\item counit $\varepsilon_{B_\cC}\colon B_\cC\to \one$ corresponding to $\Id_\cC$ under Lemma \ref{lem:representability1},
\item antipode $S_{B_\cC}\colon B_\cC\to B_\cC$ corresponding to $\delta^*$ under Lemma \ref{lem:representability1}.
\end{itemize}
To describe these structures on the coend, we use the isomorphisms from Lemma \ref{lem:representability1} and \ref{lem:representabilityn}, under which 
$\delta_X\colon X\to X\otimes B_\cC$ corresponds to the component map $d_X\colon X^*\otimes X\to B_\cC$. 
\end{proof}
The Hopf algebra structure on $B_\cC$ displayed in Proposition \ref{prop:coend-Hopf} here matches that of \cite{DGGPR}*{Section~2.4}.

In the spirit of \cite{Majid95book}*{Section~9}, the Hopf algebra $B_\cC$ is obtained via comodule reconstruction from $\cC$. Here, we denote by $ \rComod{B_\cC}(\cC)$ the monoidal category of right $B_\cC$-comodules internal to $\wcC$ where the objects underlying the comodules are contained in the full subcategory $\cC\subseteq \wcC$. 
There is a fully faithful functor of monoidal categories 
$$\cC\to \rComod{B_\cC}(\cC)$$
which sends an object $X$ of $\cC$ to itself with canonical $\cC$-coaction given by the natural transformation
$$\delta_X\colon X\to X\otimes B_\cC$$
corresponding to the identity morphism $\Id_{B_\cC}$ under Lemma \ref{lem:representability1}.

\smallskip

The Hopf algebra structure on $B_\cC$ corresponds to that of Majid's covariantized Hopf algebra.
\begin{defn}[{\cite{Majid95book}*{Example~9.4.10}}]\label{defn:cov-alg}
Let $A$ be a dual quasitriangular $R$-Hopf algebra over the integral domain $R$ with dual R-matrix $\cR\colon A\otimes A \to R$. Define the \emph{covariantized Hopf algebra} $\un{A}$ to be the same coalgebra as $A$, and the same unit, but with \emph{braided product} and antipode given by 
\begin{align}
\un{\mu}(a\otimes b)=a \br b &=  a_{(2)} b_{(3)} \cR\left(a_{(3)} \otimes S (b_{(1)})\right) \cR\left(a_{(1)} \otimes b_{(2)}\right)\label{eq:brproduct}
\\
\un{S}(a)&=S(a_{(2)})\cR(S^2(a_{(3)})S(a_{(1)})\otimes a_{(4)}).\label{eq:brantipode}
\end{align}
\end{defn}

Note that in these formula, we denote elements of $\un{A}$ by the same symbols as the corresponding elements of $A$. We will later require the following results.

\begin{lem}
The product of $A$ is recovered from the braided product $\un{\mu}$ of $\un{A}$ by the formula  
\begin{align}\label{eq:brproduct2}
\begin{split}
    a b &= \cR\left(S (a_{(1)}), b_{(1)}\right)  \cR \left( a_{(3)}, b_{(2)}\right) a_{(2)} \br b_{(3)}\\
    &= \cR^{-1}\left(a_{(1)}, b_{(1)}\right)  \cR \left( a_{(3)}, b_{(2)}\right) a_{(2)} \br b_{(3)}.
\end{split}
\end{align}
\end{lem}
\begin{proof}
This is easily verified, see e.g.~\cite{KS}*{Section 10.3.1, (86)}, and using \cite{Majid95book}*{Lemma~2.2.2} for the second equality.
\end{proof}

\begin{lem}
If $g\in A$ is grouplike, then $g$ is invertible in $\un{A}$.
\end{lem}
\begin{proof}
If $g$ is grouplike in $A$, then $\Delta(g)=g\otimes g$ and $S(g)=g^{-1}=S^{-1}(g)$. Now, computing $g\br \,g^{-1}$ using the definition of product in \Cref{eq:brproduct} shows that $g$ is invertible with respect to the braided product. 
\end{proof}

\begin{lem}\label{lem:cov-is-BC}
Let $\cC$ be a tensor subcategory of $\flRmod{H}{R}$, for $H$ a ribbon $R$-Hopf algebra as in Example~\ref{ex:filteredHopf}. Then $B_\cC$ is isomorphic as a Hopf algebra object in $\cC$ to the covariantized Hopf algebra $\un{A_\cC}$, for $A_\cC\leq H^\circ$ as defined in Lemmas \ref{lem:tensorgenerators}--\ref{lem:AC-Hopf}.
\end{lem}
\begin{proof}
Consider the maps $d_V\colon V^*\otimes V\to H^*$, $f\otimes v\mapsto c_{f,v}^V$, defined as in \eqref{eq:componentmapsH}, for $V$ a finitely generated projective $H$-module. Then $d_V$ is a morphism of $H$-modules with respect to the coadjoint action 
\begin{align}
(k\cdot f)(h)&=f(S(k_{(1)})hk_{(2)}), \qquad h,k\in H,
\end{align}
of $H$ on $H^*$. We now need to check that the Hopf algebra structure on $\un{H}^*$ satisfies the equations from Proposition \ref{prop:coend-Hopf} with respect to the maps $d_V$.

Clearly, $d_\one=u$. Further, 
$$\varepsilon_{H^*} d_V(f\otimes v)=i_V(f\otimes v)(1)=f(v)=\coev_V(f\otimes v).$$
Next, we check for the coproduct that
$$\big(\Delta_{H^*}d_V(f\otimes v)\big)(h\otimes k)=f((hk)\cdot v)=f(h\cdot (k\cdot v))$$
is equal to 
$$\left((d_V\otimes d_V)\left(\sum_\alpha f\otimes v_\alpha\otimes f^\alpha \otimes v\right)\right)(h\otimes k)=f(h\cdot v_\alpha)f^\alpha(k\cdot v)=f(h\cdot(k\cdot v)),$$
where $\coev_V=\sum_\alpha v_\alpha\otimes f^\alpha$ for dual bases $\Set{v_\alpha}$ of $V$, $\Set{f^\alpha}$ of $V^*$ as $R$-modules.

To verify that the product on $B_\cC$ corresponds to the braided product from \eqref{eq:brproduct}, we compute with the $R$-matrix $R=R^{(1)}\otimes R^{(2)}\in H\otimes H$, that
\begin{align*}
d_{V\otimes W}&(\Psi_{V^*\otimes V,W^*}\otimes \Id_W)(f\otimes v\otimes g\otimes w)(h)\\
&=d_{V\otimes W}(R^{(2)}\cdot g\otimes (R^{(1)})_{(1)}\cdot f \otimes (R^{(1)})_{(2)}\cdot v\otimes w)(h)\\
&=((R^{(1)})_{(1)}\cdot f)(h_{(1)}(R^{(1)})_{(2)}\cdot v) (R^{(2)}\cdot g)(h_{(2)}\cdot w)
\\
&=f(S((R^{(1)})_{(1)})h_{(1)}(R^{(1)})_{(2)}\cdot v) (R^{(2)}\cdot g)(h_{(2)}\cdot w)\\
&=f(S((R^{(1)})_{(1)})h_{(1)}(R^{(1)})_{(2)}\cdot v) g(S(R^{(2)})h_{(2)}\cdot w)\\
&=f(S(R_1^{(1)})h_{(1)}R_2^{(1)}\cdot v) g(S(R_2^{(2)})S(R_1^{(2)})h_{(2)}\cdot w)\\
&=f(R_1^{(1)}h_{(1)}R_2^{(1)}\cdot v) g(S(R_2^{(2)})R_1^{(2)}h_{(2)}\cdot w).
\end{align*}
Here, the  first and second equalities follow from the definitions. The third and fourth equalities apply the action of $H$ on the duals, via the antipode, cf. Equation \eqref{eq:Hstarbimodule}. The sixth equality applies the $R$-matrix axiom 
\begin{equation}
    \label{eq:R-matrix}
(\Delta\otimes \Id)(R)=(R^{(1)})_{(1)}\otimes (R^{(1)})_{(2)}\otimes R^{(2)}= R_1^{(1)}\otimes R_2^{(1)}\otimes R_1^{(2)}R_2^{(2)}=R_{13}R_{23},
\end{equation}
where $R_i=R_i^{(1)}\otimes R_i^{(2)}$, for $i=1,2$, indicates two copies of the $R$-matrix. Here, we also use that 
$$S(R_1^{(2)}R_2^{(2)})=S(R_2^{(2)})S(R_1^{(2)}),$$
since the antipode $S$ is an anti-algebra morphism. The last equality uses invariance of the $R$-matrix under application of $S\otimes S$, c.f. \cite{Majid95book}*{Lemma~2.1.2}.
The last expression in the above block of equalities equals, by definition of the maps $c^V, c^W$ in Equation \eqref{coordinate function},
\begin{align*}
c_{f,v}^V&(R_1^{(1)}h_{(1)}R_2^{(1)})c_{g,w}^W(S(R_2^{(2)})R_1^{(2)}h_{(2)})\\
&={c_{f,v}^V}_{(1)}(R_1^{(1)}){c_{f,v}^V}_{(2)}(h_{(1)}){c_{f,v}^V}_{(3)}(R_2^{(1)}){c_{g,w}^W}_{(1)}(S(R_2^{(2)})){c_{g,w}^W}_{(2)}(R_1^{(2)}){c_{g,w}^W}_{(3)}(h_{(2)})\\
&=R({c_{f,v}^V}_{(1)}\otimes {c_{g,w}^W}_{(2)})
R({c_{f,v}^V}_{(3)}\otimes S({c_{g,w}^W}_{(1)})){c_{f,v}^V}_{(2)}(h_{(1)})
{c_{g,w}^W}_{(3)}(h_{(2)})
\\
&=\un{\mu}(c_{f,v}^V\otimes c_{g,w}^W)(h)\\&=
\big(\un{\mu}(d_V\otimes d_W)(f\otimes v\otimes g\otimes w)\big)(h).
\end{align*}
Here, we used the conventions that for $f,g\in H^*$, the coproduct is given by 
$$(f_{(1)}\otimes f_{(2)})(h\otimes g)=f_{(1)}(h)f_{(2)}(g).$$

To check that the antipode is given by the formula in \eqref{eq:brantipode}, we compute
\begin{align*}
\big(d_{V^*}&(\phi_V\otimes \Id_{V^*})\Psi_{V^*,V}(f\otimes v)\big)(h)\\
&=\big(d_{V^*}(\phi_V\otimes \Id_{V^*})(R^{(2)}\cdot v\otimes R^{(1)}\cdot f)\big)(h)\\
&=\big(d_{V^*}(S^2(R_2^{(1)}) R_2^{(2)}R_1^{(2)}\cdot v\otimes R_1^{(1)}\cdot f)\big)(h)\\
&=(hR_1^{(1)}\cdot f)(S^2(R_2^{(1)}) R_2^{(2)}R_1^{(2)}\cdot v)\\
&=f(S(R_1^{(1)})S(h)S^2(R_2^{(1)}) R_2^{(2)}R_1^{(2)}\cdot v)\\
&=f(S({R^{(1)}}_{(2)})S(h)S^2({R^{(1)}}_{(1)}) R^{(2)}\cdot v).
\end{align*}
Here, the first equality uses the definition of the braiding and the second equality computes the Drinfeld isomorphism $\phi_V$ in terms of the action of the R-matrix. The third equality applies the definition of $d_{V^*}$, followed by the action of $H$ on the dual in the forth equality. The fifth equality applies the R-matrix axiom \eqref{eq:R-matrix}. 
We need to equate the above to the following. 
\begin{align*}
   \un{S}d_V(f\otimes v)(h)&=(\un{S}c_{f,v}^V)(h)\\
   &=R\big(S^2({c_{f,v}^V}_{(3)})S({c_{f,v}^V}_{(1)})\otimes {c_{f,v}^V}_{(4)}\big) S({c_{f,v}^V}_{(2)})(h)\\
   &= R\big(S^2(c_{f^\beta,v_\gamma}^V)S(c_{f,v_\alpha}^V)\otimes c_{f^\gamma,v}^V\big)S(c_{f^\alpha,v_\beta}^V)(h)\\
   &= R\big(c_{\tau f^\beta, \tau v_\gamma}^{V^{**}}c_{\tau v_\alpha, f}^{V^*}\otimes c_{f^\gamma,v}^V\big)c_{\tau v_\beta, f^\alpha}^{V^*}(h)\\
   &= R\big(c_{\tau v_\alpha\otimes \tau f^\beta, \tau v_\gamma\otimes f}^{V^{**}\otimes V^*}\otimes c_{f^\gamma,v}^V\big)c_{\tau v_\beta, f^\alpha}^{V^*}(h)\\
   &= (\tau f^\beta({R^{(1)}}_{(1)}\cdot \tau v_\gamma))(\tau v_\alpha({R^{(1)}}_{(2)}\cdot f)) f^\gamma(R^{(2)}\cdot v)c_{\tau v_\beta, f^\alpha}^V(h)\\
   &= f^\beta(S^2({R^{(1)}}_{(1)})\cdot  v_\gamma)f(S({R^{(1)}}_{(2)})\cdot v_\alpha) f^\gamma(R^{(2)}\cdot v)f^\alpha(S(h)\cdot v_\beta)\\
   &= f^\beta(S^2({R^{(1)}}_{(1)})R^{(2)}\cdot v)f(S({R^{(1)}}_{(2)})\cdot v_\alpha) f^\alpha(S(h)\cdot v_\beta)\\
   &= f(S({R^{(1)}}_{(2)})\cdot v_\alpha) f^\alpha(S(h)S^2({R^{(1)}}_{(1)})R^{(2)}\cdot v)\\
      &= f(S({R^{(1)}}_{(2)})S(h)S^2({R^{(1)}}_{(1)})R^{(2)}\cdot v).
\end{align*}
The second equality applies the definition of the braided antipode from Equation \eqref{eq:brantipode}. In the third equality, we apply the coproduct of $H^\circ$, see Equation \eqref{eq-product-cVs}, followed by several applications of the antipode of $H^\circ$ of Equation \eqref{eq:antipode-matrix-coeff}, with the pivotal structure $\tau$ of $\flmod{R}$. The fifth equation applies the product of $H^\circ$, followed by the dual R-matrix from Equation \eqref{eq:R-matrix-mat-coeff}, where the first tensor factor $R^{(1)}$ acts on the tensor product using the coproduct of $H$. Now, the seventh equality evaluates the evaluation pairings on dual spaces in terms of the original pairing $\ev_V\colon V^*\otimes V\to R$, $f\otimes v\mapsto f(v)$, making use of the equality $(h\cdot f)(v)=f(S(h)\cdot (v))$ for the dual action. The remaining equalities use the duality axiom
$v_\alpha f^\alpha(v)=v$ repeatedly. This completes the proof of the antipode formula and concludes the comparison of Hopf algebra structures. 
\end{proof}

We will later use the case of a Hopf algebra $H$ which is finitely generated projective over $R$, when $B_\cC$ is isomorphic to the dual as a coalgebra, with covariantized product (see also \cite{TV}*{Section~6.4}).

\begin{ex} Consider the finite rank case, i.e., that $H$ is a Hopf algebra in $\flmod{R}$ with dual Hopf algebra $H^*$ and $\cC=\flRmod{H}{R}$, then $B_\cC=\un{H}^*$ as a Hopf algebra in $\cC$.
\end{ex}

\begin{ex}\label{eq:BqMn-GL2-SL2}
As a first concrete key example, used later in this paper, we continue Examples~\ref{ex:gl2} and \ref{ex:sl2} and consider the algebras $B_q(GL_2):=\un{O_q(GL_2)}$ and $B_q(SL_2):=\un{O_q(SL_2)}$. As a first step, one computes the algebra $B_q(M_2):=\un{O_q(M_2)}$, based on the dual quasi-triangular bialgebra $O_q(M_2)$ from \cite{Majid95book}*{Example~4.2.5} which, even though not a Hopf algebra, can still be deformed to the algebra $B_q(M_2)$ since the dual R-matrix is convolution invertible.  The algebra $B_q(M_2)$ is generated by elements $a,b,c,d$ subject to the relations 
\begin{gather}
b\br a=q^2a\br b,\qquad  c\br a=q^2a\br c, \qquad a\br d=d\br a,\label{eq:BqM2-rel1} \\
c\br d-d\br c=\left(1-q^{-2}\right)c\br a, \quad d\br b-b\br d = \left(1-q^{-2}\right)a\br b,\label{eq:BqM2-rel2}
\\ c\br b-b\br c=\left(1-q^{2}\right)(d-a)\br a.\label{eq:BqM2-rel3}
\end{gather}
This algebra appears in, e.g.\ in \cite{Majid95book}*{Example~4.3.4}. The algebra $B_q(GL_2)$ is now given by localizing at the central element $a\br d-q^2c\br b$ of $B_q(M_2)$, and the algebra $B_q(SL_2)$ is the quotient $B_q(M_2)/\left(a\br d-q^2c\br b-1\right)$.
\end{ex}


\section{Quantum groups and quantum function algebras}
\label{sec:qgroup-defns}

\subsection{Definitions of quantum groups}
\label{sec:q-groups}

In this section, we define the quantum groups $U_q(\mathfrak{sl}_n)$ and $U_q(\mathfrak{gl}_n)$ together with their small analogues $u_\epsilon(\mathfrak{sl}_n)$ and $u_\epsilon(\mathfrak{gl}_n)$, for an odd root of unity $\epsilon\in \mC$, as well as the associated quantum function algebras (or FRT algebras) $O_q(SL_n)$ and $O_q(GL_n)$. Following \cite{Tak} we show that these are dually paired Hopf algebras, paying particular attention to integral forms for the quantum function algebras. 

To fix notation, we consider the rings and fields 
\begin{equation}\label{not:rings-fields}
\mI:=\mZ[q,q^{-1}],\qquad \Bbbk:=\mQ(q), \qquad \mK:=\mC(q),
\end{equation}
for a generic parameter $q$.
We assume that 
\begin{center}
$\epsilon$ is a primitive $\ell$-th root of unity, for $\ell$ odd,
\end{center}
and denote by $\nu$ the image of $q$ in ring of cyclotomic integers, i.e.
$$\mO=\mZ[\nu]=\mZ[q,q^{-1}]/(\Psi_\ell(q))\xrightarrow{\sim}\mZ[\epsilon]\subset \mC.$$
Our presentations follows \cite{Tak}*{Section~3.2}, see also \cite{KS}*{Section~6.1.2} or \cite{Zh}*{Section~2}.

\begin{defn}[$U_q(\mathfrak{gl}_n)$ and $U_q(\mathfrak{sl}_n)$]
$U_q(\mathfrak{gl}_n)$ is generated as a $\mK$-algebra (or, just as a $\Bbbk$-algebra) by $E_i,F_i, J_j^{\pm}$, for $i=1,\ldots, n-1$ and $j=1,\ldots, n$ such that 
\begin{gather*}
J_iJ_j=J_jJ_i, \qquad J_iJ_i^{-1}=J_i^{-1}J_i=1,\\
J_iE_jJ_i^{-1}=q^{\delta^i_j-\delta^i_{j+1}}E_j, \qquad 
J_iF_jJ_i^{-1}=q^{-\delta^i_j+\delta^i_{j+1}}F_j,\\
[E_i,F_r]=\delta^i_r\frac{J_iJ_{i+1}^{-1}-J_i^{-1}J_{i+1}}{q-q^{-1}}, \qquad [E_i,E_j]=[F_i,F_j]=0 \quad (|i-j|\geq 2),\\
E_i^2E_{i\pm 1}-\left(q + q^{-1}\right)E_iE_{i\pm 1}E_i+E_{i\pm 1}E_i^2=0,\\
F_i^2F_{i\pm 1}-\left(q + q^{-1}\right)F_iF_{i\pm 1}F_i+F_{i\pm 1}F_i^2=0.
\end{gather*}
It is a Hopf algebra with coproduct, counit, and antipode give by 
\begin{gather*}
\Delta(E_i)=E_i\otimes J_{i}J_{i+1}^{-1}+1\otimes E_i,\\
\Delta(F_i)=F_i\otimes 1+J_{i}^{-1}J_{i+1}\otimes F_i,\\
\Delta(J_i)=J_i\otimes J_i, \qquad \varepsilon(E_i)=\varepsilon(F_i)=0, \qquad \varepsilon(J_i)=1,\\
S(E_i)=-E_iJ_i^{-1}J_{i+1}, \qquad S(F_i)=-J_iJ_{i+1}^{-1}F_i,\qquad S(J_i)=J_i^{-1}.
\end{gather*}

Recall from \cite{KS}*{Section~6.1.2} that the Hopf subalgebra of $U_q(\mathfrak{gl}_n)$ generated by $E_i, F_i,K_i:=J_iJ_{i+1}^{-1}$, for $i=1,\ldots, n-1$, is isomorphic to $U_q(\mathfrak{sl}_n)$. A presentation can be derived easily from the above presentation. In particular, we have
\begin{gather*}
K_iE_iK_i^{-1}=q^2E_i,\qquad K_iF_iK_i^{-1}=q^{-2}F_i,\\
K_{i}E_{i+1}K_{i}^{-1}=q^{-1}E_{i+1}, \qquad K_{i}F_{i+1}K_{i}^{-1}=qF_{i+1},\\
K_{i}E_{j}K_{i}^{-1}=E_j, \qquad K_{i}F_{j}K_{i}^{-1}=F_j \quad (j\neq i,i+1),\\
[E_i,F_j]=\delta^i_j\frac{K_i-K_i^{-1}}{q-q^{-1}},\\
\Delta(E_i)=E_i\otimes K_i+1\otimes E_i, \qquad\Delta(F_i)=F_i\otimes 1+K_i^{-1}\otimes F_i.
\end{gather*}
\end{defn}

We define the following elements in $U_q(\mathfrak{gl}_n)$, or $U_q(\mathfrak{sl}_n)$,
\begin{gather}
\stirling{g}{k}_q:=\prod_{i=1}^k\frac{q^{i-k}g-q^{k-i}g^{-1}}{q^i-q^{-i}}, \qquad x^{(k)}=\frac{x^k}{[k]_q!},
\end{gather}
where $g=J_j, K_j$, $x=E_i, F_i$, $k\in \mZ_{\geq 0}$, and
\begin{align}
[k]_q!=[k]_q[k-1]_q\ldots [1]_q, \qquad [k]_q=\frac{q^k-q^{-k}}{q-q^{-1}}.
\end{align}
In particular, $[1]_q=1$, $
\stirling{g}{0}_q=1, \stirling{g}{1}_q=\frac{g-g^{-1}}{q-q^{-1}}$, and $x^{(1)}=x$.
\begin{defn}[Integral forms $U^\intf_q(\mathfrak{gl}_n)$ and $U^\intf_q(\mathfrak{sl}_n)$]
$U^\intf_q(\mathfrak{gl}_n)$ is defined to be the $\mI$-subalgebra of $U_q(\mathfrak{gl}_n)$ generated by $J_j^{\pm 1}$, $\stirling{J_j^{\pm}}{k}_q$, $E_i^{(k)}$, and $F_i^{(k)}$, for $1\leq i\leq n-1$, $1\leq j\leq n$, and $k\in \mZ_{\geq 0}$. Similarly, one defines $U^\intf_q(\mathfrak{sl}_n)$ as the $\mI$-subalgebra of $U^\intf_q(\mathfrak{gl}_n)$ generated by $K_j^{\pm 1}$, $\stirling{K_j^{\pm}}{k}_q$, $E_i^{(k)}$, and $F_i^{(k)}$. In this algebra, we have the following the relations:
\begin{equation}E_i^k=E_i^{(k)}[k]_q!, \qquad F_i^k=F_i^{(k)}[k]_q!
\end{equation}
\end{defn}
$U^\intf_q(\mathfrak{gl}_n)$ and $U^\intf_q(\mathfrak{sl}_n)$ are Hopf algebras over $\mI$ and it is possible to give a full presentation of $U^\intf_q(\mathfrak{gl}_n)$ by generators and relations following \cite{Lus2}*{Theorem~4.5}. 
There is an isomorphism of $\Bbbk$-Hopf algebras between $U_q^\intf(\mathfrak{gl}_n)\otimes_{\mI}\Bbbk$ and $U_q(\mathfrak{gl}_n)$.

We define small quantum groups following \cite{Lus2}*{Section~0.4}.
\begin{defn}[Integral forms of small quantum groups $u_\nu^\intf(\mathfrak{gl}_n)$, $u_\nu^\intf(\mathfrak{sl}_n)$]
Define $u_\nu^\intf(\mathfrak{gl}_n)$ to be the $\mO$-subalgebra of the quotient $$U_q^\intf(\mathfrak{gl}_n)\otimes_\mI \mO/\inner{J^\ell_j-1\mid j=1,\ldots, n}$$ generated by $E_i$, $F_i$, and $J_j$, $1\leq i\leq n-1$, $1\leq j\leq n$. It follows that $u_\nu^\intf(\mathfrak{gl}_n)$ is a Hopf algebra. The Hopf algebra $u_\nu^\intf(\mathfrak{sl}_n)$ is defined as the $\mO$-subalgebra of $u_\nu^\intf(\mathfrak{gl}_n)$ generated by $E_i$, $F_i$, and $K_i$ of the same quotient. 
\end{defn}

\begin{defn}[Quantum groups at odd roots of unity $\epsilon\in \mC$]\label{def:U-at-epsilon}
We define 
\begin{gather*}
U_\epsilon(\mathfrak{gl}_n):=U_q^\intf(\mathfrak{gl}_n)\otimes_\mI \mC,\qquad U_\epsilon(\mathfrak{sl}_n):=U_q^\intf(\mathfrak{sl}_n)\otimes_\mI \mC,\\
u_\epsilon(\mathfrak{gl}_n):=u_\nu^\intf(\mathfrak{gl}_n)\otimes_\mO \mC,\qquad u_\epsilon(\mathfrak{sl}_n):=u_\nu^\intf(\mathfrak{sl}_n)\otimes_\mO \mC,
\end{gather*}
using the ring homomorphism $\mI\twoheadrightarrow \mO, q\mapsto \epsilon,$ and $\mO\hookrightarrow  \mC, \nu\mapsto \epsilon$.
\end{defn}

We will subsequently denote the generators of the small quantum groups corresponding to $E_i,F_i,J_j,K_j$ by lower case symbols $e_i, f_i, j_j, k_j$, respectively. The following result \cite{Lus2}*{Section 5.7} can be used to derive a presentation for $u_\epsilon(\mathfrak{gl}_n)$ or $u_\epsilon(\mathfrak{sl}_n)$.

\begin{lem}\label{lem:small-u-as-quo}
Mapping $E_i\mapsto e_i$, $F_i\mapsto f_i$, $J_j\mapsto j_j$ (or, $K_i\mapsto k_i$)  defines a surjective homomorphism of Hopf algebras $\pi\colon U_\epsilon(\mathfrak{gl}_n)\twoheadrightarrow u_\epsilon(\mathfrak{gl}_n)$ (respectively, $\pi\colon U_\epsilon(\mathfrak{sl}_n)\twoheadrightarrow u_\epsilon(\mathfrak{sl}_n)$). 

The kernel of $\pi$ is given by the ideal generated by 
$E_\alpha^\ell, F_\alpha^\ell, J_j^\ell-1$, for any positive root $\alpha$, $j=1,\ldots, n$ in the $\mathfrak{gl}_n$-case (respectively, with the generating relation $K_j^\ell-1$ instead of $J_j^\ell-1$ in the $\mathfrak{sl}_n$-case).
\end{lem}

Finally, we observe that the standard module for $U_q(\mathfrak{gl}_n)$ and $U_q(\mathfrak{sl}_n)$ can be defined for the integral forms of these quantum groups, cf.\ \cite{Tak}*{Section~3.3}.

\begin{lem}\label{lem:V1-integral}
$U^\intf_q(\mathfrak{gl}_n)$ admits an irreducible rank $n$ module $V_1$ with free $\mI$-basis $v_1,\ldots, v_n$ and action given by 
\begin{align*}
J_i\cdot v_k&=q^{\delta^i_k}v_k, & E_j\cdot v_{k}&=\delta^j_k v_{k-1}, &  F_j\cdot v_k&=\delta^j_k v_{k+1},\\
\stirling{J_i}{l}_q\cdot v_k&=\delta^1_l\delta^i_k v_k, & E_j^{(l)}\cdot v_{k+1}&=\delta^1_l \delta^{j+1}_k v_k, & F_j^{(l)}\cdot v_{k}&=\delta^1_l\delta^j_k v_{k+1}.
\end{align*}
This module is also irreducible when restricted to a $U^\intf_q(\mathfrak{sl}_n)$-module.
\end{lem}
\begin{proof}
One checks that $V_1$ is closed under action by the generators of $U^\intf_q(\mathfrak{gl}_n)$. The Cartan part, generated by the $J_i$ and $\stirling{J_i}{k}_q$ acts by scalars in $\mI$ and $V_1$ decomposes as a direct sum of distinct simple modules with respect to this action. This shows that the $\mI$-action is free with the stated basis. Moreover, the module $V_1$ is a highest weight module with highest weight vector $v_1$. This can be used to show that $V_1$ is irreducible as a $U_q^\intf(\mathfrak{sl}_n)$-module and, hence, as a $U_q^\intf(\mathfrak{sl}_n)$-module.  
\end{proof}

By extension of scalars, one recovers a $n$-dimensional simple $U_q(\mathfrak{gl}_n)$-module and $U_q(\mathfrak{sl}_n)$-module, denoted by $V_1^\mK=V_1\otimes_\mI\mK$.

\begin{lem}
Extension of scalars $V_1^\mO=V_1\otimes_\mI\mO$ along the map of rings $\mI\twoheadrightarrow \mO, q\mapsto \nu$ produces an irreducible $u_\nu^\intf(\mathfrak{gl}_n)$ module which is also irreducible when restricted to an $u_\nu^\intf(\mathfrak{sl}_n)$-module.
\end{lem}

Again, extending scalars gives a simple $n$-dimensional module $V_1^\mC=V_1^\mO\otimes_{\mO}\mC$ of $u_\epsilon(\mathfrak{gl}_n)$ or $u_\epsilon(\mathfrak{sl}_n)$.

\subsection{Definitions of quantum function algebras}\label{sec:q-fun-alg}

We now introduce various versions of quantum function algebras of type $GL_n$ and $SL_n$ following \cite{Tak}.

\begin{defn}[The $q$-matrix algebra $O^\intf_q(M_n)$]\label{defn:OLqMn}
Define $O^\intf_q(M_n)$ to be the $\mI$-algebra generated by $x^{i}_{j}$ subject to relations 
\begin{gather}
x^i_kx^i_j=qx^i_jx^i_k,\qquad x^k_ix^j_i=qx^j_ix^k_i\quad (\forall 1\leq j<k\leq n), \label{eq:Mqn1}\\
\begin{split}
x^i_lx^j_k=x^j_kx^i_l, \quad x^j_lx^i_k-x^i_kx^j_l=\left(q-q^{-1}\right)x^i_lx^j_k\\ (\forall 1\leq i<j\leq n,\quad 1\leq k<l\leq n ).\end{split}
\label{eq:Mqn2}
\end{gather}
\end{defn}

The second type of relations \eqref{eq:Mqn2} can be expressed as 
$$bc=cb, \qquad da-ad=\left(q-q^{-1}\right)bc$$
for any $2\times 2$-minor matrix 
$
\left(
\begin{smallmatrix}
a&b\\
c&d
\end{smallmatrix}
\right)
$ 
of the matrix $\left(x^i_j\right)_{1\leq i,j\leq n}$.

\begin{ex}
If $n=2$, we recover the relations from \Cref{eq:AR-rels}.
\end{ex}

\begin{rmk}
The algebra $O^\intf_q(M_n)$ is often defined with relations summarized in a matrix equation. How to relate to such a presentation is expained, with the conventions used here, e.g., in \cite{Majid95book}*{Exercise~4.1.2, Examples~4.1.3 \& 4.2.5}.
\end{rmk}

For any total ordering of the generators $x^i_j$, the set
$$\left\{ \prod_{1\leq i\leq j\leq n} \left(x^i_j\right)^{n_{ij}}\;\middle|\; n_{ij}\in \mZ_{\geq 0} \right\}$$
is an $\mI$-basis for $O^\intf_q(M_n)$, see \cite{Tak}*{Section 3.1} and  \cite{PW}. Note that, unlike for the quantum groups, the $q$-matrix algebras do not require divided powers in their integral forms.

By extension of scalars, we define versions of the $q$-matrix algebras valued in other rings and fields from \eqref{not:rings-fields}.
\begin{gather}
O_q(M_n):=O^\intf_q(M_n)\otimes_{\mI}\Bbbk,\qquad O_\epsilon(M_n):=O^\intf_q(M_n)\otimes_\mI \mC
\end{gather}
as in Definition \ref{def:U-at-epsilon}, where the latter uses the root of unity $\epsilon$ of odd order $\ell$. Unlike the case of the quantized enveloping algebras in Section \ref{sec:q-groups}, the presentations of the integral form  directly give presentations for these quantum function algebras.

\begin{lem}
$O^\intf_q(M_n)$ is a bialgebra with coproduct and antipode given by
\begin{align}
\Delta\left(x^i_j\right)=\sum_{k=1}^n x^i_k\otimes x^k_j, \qquad \varepsilon\left(x^i_j\right)=\delta^i_j.\label{eq:Mn-coalgebra}
\end{align}
\end{lem}

Define the $q$-determinant by
\begin{equation}
\dq{n}:=\sum_{\sigma\in S_n}(-q)^{-l(\sigma)}x^1_{\sigma(1)}\ldots x^n_{\sigma(n)},\label{eq:qdet}
\end{equation}
where $$l(\sigma)=|\Set{(i,j)~|~ 1\leq i<j\leq n \text{ such that } \sigma(i)>\sigma(j)}|$$ denotes the number of \emph{inversions} of $\sigma$.  
As a special case, we recover $\dq{2}$ from \eqref{eq:qdet2}.
Results from \cites{PW,Tak} imply the following lemma.
\begin{lem}\label{lem:dqn}
The element $\dq{n}$ is not a zero divisor. It is grouplike and central in the bialgebra $O^\intf_q(M_n)$.
\end{lem}
The lemmas hold similarly for extensions of $O^\intf_q(M_n)$  to coefficients in any integral domain that is an $\mI$-algebra, e.g., the fields $\Bbbk$ and $\mK$, see \eqref{not:rings-fields}. 

\begin{defn}[Integral versions $O^\intf_q(GL_n)$, $O^\intf_q(SL_n)$]\label{defn:OqLGLn-SLn}
We define 
\begin{gather*}
O^\intf_q(GL_n):=O^\intf_q(M_n)[t]/(t\dq{n}-1),\\
O^\intf_q(SL_n):=O^\intf_q(M_n)/(\dq{n}-1).
\end{gather*}
Similarly, we define by extension of scalars
\begin{gather}\label{eq:OqGn}
O_q(G_n):=O^\intf_q(G_n)\otimes_{\mI}\Bbbk, \qquad O_\epsilon(G_n):=O^\intf_q(G_n)\otimes_\mI \mC,
\end{gather}
For $G_n=GL_n$ or $SL_n$, using the ring homomorphisms $\mI\twoheadrightarrow \Bbbk, q\mapsto \epsilon,$ and $\mI\hookrightarrow  \mC$. We will also regard $O_q(G_n)$ as a $\mK$-algebra without distinguishing notation.
\end{defn}

\begin{ex}
When $n=3$, a detailed presentation of the algebra $O_q(SL_n)$ is given in \cite{Majid95book}*{Example~4.2.9}.
\end{ex}

\begin{lem}[{\cite{PW}*{(5.3.2) Theorem}, \cite{Tak}}]\label{lem:OqGL-antipode}
The $\mI$-algebra $O_q^\intf(GL_n)$ is a Hopf algebra with sub-bialgebra $O^\intf_q(M_n)$.
\end{lem}

\begin{defn}[Small quantum function algebras, $GL_n$-type]\label{pres-oqgln}
We define $o_\nu^\intf(GL_n)$ to be the  quotient of $O^\intf_q(GL_n)\otimes_\mI \mO$ by the relations 
\begin{equation}\label{eq:small-o-rel}
\left(x^i_j\right)^\ell=0\quad (1\leq i\neq j\leq n),\quad \text{and}\quad  \left(x^i_i\right)^{\ell}=1.
\end{equation}
As this ideal is a Hopf ideal, $o_\nu^\intf(GL_n)$ is a Hopf algebra. 
Further, we define 
$$o_\epsilon(GL_n)=o_\nu^\intf(GL_n)\otimes_\mO\mC.$$
\end{defn}

\begin{lem}[{\cite{Tak}*{Section~5}}]\label{lem:Tak-OGLn-quotient}
$o^\intf_\nu(GL_n)$ is isomorphic to the quotient of $O^\intf_q(M_n)\otimes_\mI \mO$ by the ideal generated by 
\begin{align}
\left(x^i_j\right)^\ell\quad (1\leq i\neq j\leq n),\quad \text{and}\quad  \left(x^i_i\right)^{\ell}-1.\label{eq:rel-small-o}
\end{align}
In particular, $\dq{n}$ is invertible in this quotient.
An $\mO$-basis for $o_\nu(GL_n)$ is given by the set
$$\left\lbrace \prod_{1\leq i,j\leq n} \left(x^i_j\right)^{n_{ij}}\,\middle| \, 0\leq n_{ij}\leq \ell-1 \right\rbrace.$$
\end{lem}
Thus, $o^\intf_\nu(GL_n)$ is a free $\mO$-module of rank $\ell^{n^2}=\ell^{\dim \mathfrak{gl}_n}$. Extending scalars to $\mC$ we see that $o_\epsilon(GL_n)$ is $\ell^{n^2}$-dimensional.

We remark that the additional relations \eqref{eq:rel-small-o} are easier than those for the small quantum groups in Lemma~\ref{lem:small-u-as-quo} which are indexed by positive roots.

\begin{defn}[Small quantum function algebras, $SL_n$-type]\label{pres-oqsln}
We define $o^\intf_\nu(SL_n)$ to be the  quotient of $o^\intf_\nu(GL_n)$ by the ideal generated by 
$$\dnu{n}-1.$$
As this ideal is a Hopf ideal, $o^\intf_\nu(SL_n)$ is a Hopf algebra. 
Further, we define 
$$o_\epsilon(SL_n)=o^\intf_\nu(SL_n)\otimes_\mO\mC.$$
\end{defn}
The Hopf algebra $o^\intf_\nu(SL_n)$ has rank $\ell^{n^2-1}=\ell^{\dim \mathfrak{sl}_n}$ as an $\mO$-module. For completeness, we summarize these  presentations below.

The same presentations apply to $o_\epsilon(GL_n)$ and $o_\epsilon(SL_n)$ when extending scalars to $\mC$ via the algebra homomorphism $\mO\hookrightarrow \mC, q\mapsto \epsilon$.

\subsection{Quantum function algebras as FRT algebras}

We will now identify the FRT algebras from Section \ref{sec:AC-coend} associated to certain categories of representations over quantum groups 
with the corresponding quantum function algebras from Section \ref{sec:q-fun-alg}. In the following, denote 
$$\mathfrak{g}_n=
\mathfrak{gl}_n\text{ or }\mathfrak{sl}_n.$$

We start by considering the integral case. Recall that $\mI=\mZ[q,q^{-1}]$ and consider the $\mI$-Hopf algebra $U_q^\intf(\mathfrak{g}_n)$ and the $\mI$-tensor category category $\cC=\flRmod{U_q^\intf(\mathfrak{g}_n)}{\mI}$ of modules which are finitely-generated projective as $\mI$-modules.

\begin{defn}[Type I representations]
We define the $\mI$-linear category $\cC^\intf_q(\mathfrak{g}_n)$ of \emph{integral type I representations} of $U^\intf_q(\mathfrak{g}_n)$ as the Karoubian tensor subcategory of the category $\cC$ generated by $V_1$ and $V_1^*$ from Lemma \ref{lem:V1-integral}.

Similarly, we define 
$\cC_q(\mathfrak{g}_n)$ and $
\cC_\epsilon(\mathfrak{g}_n)$,
the tensor categories of \emph{type I representations} over  
$U_q(\mathfrak{g}_n)$ and $U_\epsilon(\mathfrak{g}_n)$, respectively, as the tensor subcategories generated by the modules $V_1^\Bbbk$ and $V_1^\mC$, respectively, and their duals.
\end{defn}

In can be shown that the simple modules in $\cC_q(\mathfrak{g}_n)$ are precisely the simple type I representations of $U_q(\mathfrak{g}_n)$ \cite{Jan}*{Section~5.2}. These simple modules are precisely the simple direct summands of tensor powers of $V_1$ and $V_1^*$ \cite{Jan}*{5A.10~Proposition}. 
Hence, $\cC_q(\mathfrak{g}_n)$ corresponds to the (abelian, semisimple) subcategory of $\flRmod{U_q(\mathfrak{g}_n)}{\Bbbk}$ of representations that are direct sums of simple modules of weight $(q^{m_1},\ldots, q^{m_{r}})$, i.e., where 
$$K_{\alpha_i}\cdot v=q^{m_i}v,\qquad m_i\in \mZ,$$
for a highest weight vector $v$, and $K_{\alpha_i},\ldots, K_{\alpha_r}$ a choice of simple roots for $\mathfrak{g}_n$. In the case of $\mathfrak{gl}_n$ we have $r=n$ and $K_{\alpha_i}=J_i$ while for $\mathfrak{sl}_n$ we have $r=n-1$ and $K_{\alpha_i}=K_i$.

\begin{prop}\label{prop:Tak-GLn}
Set $\cC=\cC^\intf_q(\mathfrak{g}_n)$.  Then $A_\cC$ is isomorphic as a Hopf algebra to $O_q^\intf(G_n)$, with $G_n=GL_n$ or $SL_n$, respectively. The isomorphism is given on generators by 
$x^i_j\mapsto c^{V_1}_{f_i,v_j}$ and $S\left(x^i_j\right)\mapsto c^{V_1^*}_{v_i,f_j}$, using the set of free $\mI$-generators $\left\{v_1,\ldots, v_n\right\}$ of $V_1$ from Lemma \ref{lem:V1-integral} and the dual basis $\left\{f_1,\ldots, f_n\right\}$ of $V_1^*$.

\noindent Moreover, this isomorphism is one of $U_q^\intf(\mathfrak{g}_n)$-bimodule algebras. 
\end{prop}
\begin{proof}
First, we consider the $GL_n$-case. By definition, $\cC=\cC^\intf_q(\mathfrak{gl}_n)$ is generated as an $\mI$-tensor category by $V_1$.
We check that the morphism $\phi\colon A_\cC'\to O_q^\intf(GL_n)$ from the statement of the Proposition is a morphism of algebras, i.e., that the $\phi(x_j^i)=c_{f_i,v_j}^{V_1}$ satisfy the relations of $O_q(GL_n)$.
For this, consider the pairing $$\langle ~,~\rangle\colon O_q^\intf(M_n)\otimes U_q^\intf(\mathfrak{gl}_n)\to \mZ[q,q^{-1}], \quad \inner{x^i_j,u}=c_{f_i,v_j}(u)=f_i(u\cdot v_j).$$
This pairing is given on generators by 
\begin{gather*}
\inner{x^i_j,J_k}=\delta^i_j q^{\delta^j_k}, \quad \inner{x^i_j,{\textstyle\stirling{J_k}{l}_q}}=\delta^1_l\delta^j_k\delta^i_j,\\
 \inner{x^i_j,E_k^{(l)}}=\delta^1_l\delta^j_k\delta^i_{j-1},\quad \inner{x^i_j,F_k^{(l)}}=\delta^1_l\delta^j_k\delta^i_{j+1}, 
\end{gather*}
and extends to products via 
$$\inner{xy,h}=\inner{x,h_{(1)}}\inner{y,h_{(2)}},\qquad\inner{x,gh}=\inner{x_{(1)},g}\inner{x_{(2)},h}.$$
This pairing extends to the pairing $O_q^\intf(GL_n)\otimes U_q^\intf(\mathfrak{gl}_n)\to \mI$ of \cite{Tak}*{Section~3.3}, where the inverse of the quantum determinant is evaluated as 
$$\inner{t,u}=\inner{\dq{n},S(u)},$$
using the antipode.

The action of the generators on the dual $V_1^*$ is given by 
$$J_i\cdot f_j=q^{-\delta^i_j}f_j, \quad E_k\cdot f_k=-q^{-1}f_{k+1}, \quad F_k\cdot f_k=-q^{-1}f_{k-1},$$
and zero otherwise. Thus, the coordinate functions associated with $V_1^*$ are, up to multiplication with elements from $\mI$, given by the $\left(x^i_j\right)_{i,j}$ as they take the same values on the generators.
Since $V_1,V_1^*$ are tensor generators for the $\mI$-tensor category $\cC=\mathcal{C}^\intf_q(\mathfrak{gl}_n)$, Lemma \ref{lem:tensorgenerators} shows that the $\left(x^i_j\right)_{i,j}$, and the corresponding dual coordinate functions $c_{f_i,v_j}^{V_1^*}$ identifying $V_1^{**}\cong V_1$, generate $A_\cC$. It is enough to formally adjoin the inverse $t$ of the quantum determinant to generate all dual coordinate functions from the $\left(x^i_j\right)_{i,j}$ and $t$ since $O_q^\intf(GL_n)=O_q^\intf(M_n)[\dq{n}^{-1}]$ is a Hopf algebra. Thus, the map $\phi\colon O_q^\intf(M_n)[\dq{n}^{-1}]\to U_q^\intf(\mathfrak{gl}_n)^\circ$ given by the pairing is surjective.

Now, the pairing is non-degenerated by \cite{Tak}*{4.4.~Theorem}, where this property is referred to as \emph{connectedness}, i.e., injectivity of the map $\phi$. As discussed in Section~\ref{sec:AC-coend}, this implies that $A_\cC$ and $O_q^\intf(GL_n)$ are isomorphic.

The $SL_n$-case follows, similarly, from \cite{Tak}*{4.11.~Theorem} by restricting the pairing to a pairing $O_q^\intf(M_n)\otimes U_q^{\intf}(\mathfrak{sl}_n)$ and observing that $\dq{n}-1$ is in the left radical. 
\end{proof}

Now consider any morphism of rings $\mI\to R$ and denote $U^R=U\otimes_\mI R$ for a Hopf algebra $U$ with scalars extended to $R$. The functor of extension of scalars
$$\lMod{U}_\mI\to \lMod{U^R}, \quad V\mapsto V\otimes_\mI R$$
is left adjoint to restriction and thus preserves all colimits that exist in $\lMod{U}_R$. Since $A_\cC$ is defined as a colimit, we derive that for the subcategory $\cC^R$ generated by $V_1^R$ and its dual, $A_{\cC^R}=A_{\cC}\otimes_\mI R$ is simply given by extension of scalars. 
Thus, we derive the following corollaries from Proposition~\ref{prop:Tak-GLn}.

\begin{cor}
There are isomorphisms of Hopf algebras (and $U$-bimodule algebras) $A_\cC\cong O$, where:
\begin{enumerate}
\item[(i)] $\cC=\cC_q(\mathfrak{gl}_n)$,  $O=O_q(GL_n)$, and $U=U_q(\mathfrak{gl}_n)$,
\item[(ii)]$\cC=\cC_\epsilon(\mathfrak{gl}_n)$, $O=O_\epsilon(GL_n)$, and $U=U_\epsilon(\mathfrak{gl}_n)$,
\item[(iii)] $\cC=\cC_q(\mathfrak{sl}_n)$, $O=O_q(SL_n)$, and  $U=U_q(\mathfrak{sl}_n)$,
\item[(iv)]$\cC=\cC_\epsilon(\mathfrak{sl}_n)$, $O=O_\epsilon(SL_n)$, and $U=U_\epsilon(\mathfrak{sl}_n)$.
\end{enumerate}
\end{cor}
We remark that \cite{Tak}*{Section~4} shows the non-degenerateness of the pairings of quantum function algebras and quantum groups over any base ring containing $\mI$ (e.g., $\mK$, $\mC$ as used in the above Corollary).

We now define $O_\nu(GL_n)^\intf:=O_q(GL_n)^\intf\otimes_\mI\mO$ and $U^\intf_\nu(\mathfrak{gl}_n):=U^\intf_q(\mathfrak{gl}_n)\otimes_\mI\mO$ and define $ u'_\nu(\mathfrak{gl}_n)$ to be the $\mO$-subalgebra of $U^\intf_\nu(\mathfrak{gl}_n)$ generated by the $J_i^{\pm 1}, E_j,F_j$.

\begin{prop} 
The restriction of the pairing $O^\intf_\nu(GL_n)\otimes U^\intf_\nu(\mathfrak{gl}_n)\to \mO$ to $O^\intf_\nu(GL_n)\otimes u'_\nu(\mathfrak{gl}_n)$ has right radical generated by $J_i^{\pm\ell}-1$ and left radical generated by the relations \eqref{eq:small-o-rel}. 
\end{prop}
\begin{proof}
This is proved in \cite{Tak}*{5.4.2.--5.4.4.}.
\end{proof}
The above proposition implies that the induced pairing
$$o_\nu^\intf(GL_n)\otimes u_\nu^\intf(\mathfrak{gl}_n)\to \mO$$
is non-degenerate.
By \cite{Tak}*{Section~5.5}, this pairing induces a non-degenerate pairing 
$$o_\nu^\intf(SL_n)\otimes u_\nu^\intf(\mathfrak{sl}_n)\to \mO,$$
since, when restricting to the Hopf subalgebra $u_\nu^\intf(\mathfrak{sl}_n)\subset u_\nu^\intf(\mathfrak{gl}_n)$, the relation $\dq{n}-1$ will be in the right radical. 

Thus, we have isomorphisms of Hopf algebras over $\mO$,
$$
u_\nu^\intf(\mathfrak{gl}_n)^*\cong o_\nu^\intf(GL_n), \quad \text{and}\quad u_\nu^\intf(\mathfrak{sl}_n)^*\cong o_\nu^\intf(SL_n).
$$
Now, denote by $\cD$ the $\mO$-tensor category  of $\flRmod{u_\nu^\intf(\mathfrak{gl}_n)}{\mO}$ consisting of $u_\nu^\intf(\mathfrak{gl}_n)$-modules which are projective of finite rank over $\mO$. Thus, by Lemma \ref{lem:finiteHopf}, we obtain the following result.

\begin{cor}
\label{cor:Tak-small}
For $\mathfrak{g}_n=\mathfrak{gl}_n$ or $\mathfrak{sl}_n$ and $\cD$ as above, $A_{\cD}$ is isomorphic as an algebra in $\cD^{\op}\boxtimes \cD$ (or as a Hopf algebra)  to $o_\nu^\intf(G_n)$, with $G_n=GL_n$ or $SL_n$, respectively. The isomorphism is given on generators as in Proposition \ref{prop:Tak-GLn}.
\end{cor}
By extension of scalars via $\mO\hookrightarrow \mC$, $\nu\mapsto \epsilon$, we obtain the following. 

\begin{cor}
There is an isomorphism 
$A_\cD\cong o_\epsilon(\mathfrak{g}_n)$ of $u_\epsilon(\mathfrak{g}_n)$-bimodule algebras (and $\Bbbk$-Hopf algebras)
for $\cD=\flkmod{u_\epsilon(\mathfrak{g}_n)}{\Bbbk}$, the abelian tensor category of finite-dimensional $u_\epsilon(\mathfrak{g}_n)$-modules.
\end{cor}


\section{Presentations for small reflection equation algebras}
\label{sec:results}

This section contains the new results of this paper. We will use the presentations from \Cref{pres-oqgln} and \Cref{pres-oqsln} of $o_\epsilon(GL_n)$ and $o_\epsilon(SL_n)$, respectively, to give a presentation for the corresponding RE algebras, denoted by  $b_\epsilon(GL_n)$ and $b_\epsilon(SL_n)$, respectively, for $\epsilon$ is an odd $\ell$-th root of unity. These RE algebras are defined, in general, in \Cref{sec:BC} and are isomorphic to the covariantized Hopf algebras from \Cref{defn:cov-alg} associated to $o_\epsilon(GL_n)$ and $o_\epsilon(SL_n)$, respectively, by \Cref{lem:cov-is-BC}.

\subsection{The reflection equation algebras at generic parameter}\label{sec:generic-RE}

We start by detailing the presentation of the (infinite-dimensional) reflection equation algebras $B_q(GL_n)$ and $B_q(SL_n)$ for a generic parameter $q$. We use capital letters to denote these infinite dimensional RE algebras and lower case for the finite-dimensional small RE algebras to avoid confusion. 

\begin{defn}[Reflection equation algebras, generic integral form]
Recall $\mI=\mZ[q,q^{-1}]$. We denote by $B_q^\intf(M_n)$ the $\mI$-algebra which is the covariantized algebra $\un{A}$ of the $\mI$-Hopf algebra $A=O_q^\intf(M_n)$ from \Cref{defn:OLqMn}. Moreover, we denote by $B_q^\intf(GL_n)$ and $B_q^\intf(SL_n)$ the $\mI$-algebras $\un{A}$ for the $\mI$-Hopf algebras $A=O_q^\intf(GL_n)$, respectively, $A=O_q^\intf(SL_n)$ from \Cref{defn:OqLGLn-SLn}.
\end{defn}

We note that extension of scalars commutes with the process of passing to the covariantized algebra in \Cref{defn:cov-alg} and we denote the reflection equation algebras over $\Bbbk$ (or $\mK$) by 
$B_q(M_n)$, $B_q(GL_n)$, and $B_q(SL_n)$, respectively, see \Cref{eq:OqGn}. The following Lemma gives presentations for these RE algebras.

The dual R-matrix for $O_q^\intf(M_n)$ is defined by 
$$\cR(x^i_j\otimes x^k_l)=R^{ik}_{jl},$$
where, following  \cite{Majid95book}*{Example~4.5.7}, we use 
\begin{equation}\label{eq:R-matrix-sln}
    R^{ik}_{jl} = q^{-\frac{1}{n}}\left(\delta^i_j \delta^k_l
    q^{\delta^j_l}
    + (q - q^{-1}) \delta_{j>l}\delta^i_l \delta^k_j \right),
\end{equation}
where $\delta_{j>l}=1$ if and only if $j>l$ and zero otherwise, and $\delta^i_j$ is the Kronecker delta.
For the inverse dual R-matrix $$\cR^{-1}(a \otimes b)=\cR((S)a \otimes b), $$ 
see \cite{Majid95book}*{Lemma~2.2.2}, we have
\begin{equation}
(R^{-1})^{ik}_{jl} = q^{\frac{1}{n}}\left(\delta^i_j \delta^k_l
q^{-\delta^j_l}
- (q - q^{-1})\delta_{j>l} \delta^i_l \delta^k_j\right).
\end{equation}

Following \cites{KS,JW}, to twist the relations of the FRT algebra $A$ to give relations in the RE algebra $\un{A}$, we consider the following linear maps
\begin{equation}\label{eq:twistingmap}
\Psi\colon A\to \un{A}, \quad \Psi(1)=1, \quad \Psi\left(x^i_j\right)=u^i_j,
\end{equation}
which fixes the generators and is extended iteratively by sending a product $ab$ of elements in $A$ to the element
\begin{equation}
\begin{split}
\Psi(ab)=&\cR^{-1}\left(a_{(1)}, b_{(1)}\right)  \cR \left( a_{(3)}, b_{(2)}\right) \Psi(a_{(2)})\br \Psi(b_{(3)})\\
=&\cR\left(S(a_{(1)}), b_{(1)}\right)  \cR \left( a_{(3)}, b_{(2)}\right) \Psi(a_{(2)})\br \Psi(b_{(3)}).
\end{split}
\label{eq:twist-iteration}
\end{equation}
It follows that $\Psi$ is an isomorphism of free modules over the chosen base ring (e.g.\ $\mI$, $\mC$).
In particular, for the dual R-matrix in Equation \ref{eq:R-matrix-sln}, $\Psi$ is given on quadratic elements by 
\begin{align}\label{eq:twist-quadratic}
\begin{split}
\Psi(x^i_jx^k_l)=& 
    q^{\delta^j_k-\delta^i_k}u^i_j \br u^k_l
    + \delta^k_j (q - q^{-1})\sum_{d<j}
    q^{-\delta^i_k}u^i_d \br u^d_l
\\&
- \delta_{k>i}(q - q^{-1})
        q^{\delta^i_j}u^k_j \br u^i_l
   -\delta_{k>i}\delta^i_j(q - q^{-1})^2  \sum_{b<j} 
    u^k_b \br u^b_l.
    \end{split}
\end{align}
We refer to application of the map $\Psi$ as \emph{twisting}. 

\smallskip

The following presentation is known and given in detail in, e.g.\ \cite{JW}*{Equation (2.4)} and \cite{DL}*{Section~3} with different conventions (see Remark~\ref{rmk:compare-JW}).

\begin{lem}\label{lem:BqMn-relations}
The $\mI$-algebra $B^\intf_q(M_n)$ is generated by $u_i^j$ for $i,j=1,\ldots, n$ subject to the following relations for the braided product:
\begin{align}
u^i_l\br u^i_j-q^{\delta^i_j-\delta^i_l+1}u^i_j \br u^i_l
=&\; 
 \delta^i_j (q^2 - 1)\sum_{d<j}
u^i_d \br u^d_l- \delta^i_l \left(1 - q^{-2}\right)\sum_{d<l}
   u^i_d \br u^d_j,
\label{eq:BqMn-rel-1}
\\ u^k_j \br u^i_j-q^{\delta^j_k-\delta^i_j-1} 
 u^i_j \br u^k_j=& 
\;\delta^k_j \left(1 - q^{-2}\right)\sum_{d<j}
   u^i_d \br u^d_j
-\delta^i_j\left(1-q^{-2}\right) \sum_{b<j} 
    u^k_b \br u^b_j,
\label{eq:BqMn-rel-2}
\\
\begin{split}
 u^k_j \br u^i_l-q^{\delta^k_l-\delta^i_j} 
 u^i_l \br u^k_j=& -(q - q^{-1})
    q^{\delta^i_l-\delta^i_j} u^k_l \br u^i_j
\\
&-\delta^i_l(q - q^{-1})^2 q^{-1}  \sum_{b<l} 
    u^k_b \br u^b_j,
\\&
+\delta^k_l (q - q^{-1})q^{-\delta^i_j}\sum_{d<l}
   u^i_d \br u^d_j\\
&- \delta^i_j \left(1 - q^{-2}\right)\sum_{d<j}
   u^k_d \br u^d_l,
\end{split}
\label{eq:BqMn-rel-3}
\\
\begin{split}
 u^k_l\br u^i_j-q^{\delta^j_k-\delta^i_l} 
 u^i_j \br u^k_l &=
 \delta^k_j (q - q^{-1})q^{-\delta^i_l}\sum_{b<j}
   u^i_b \br u^b_l\\
&- \delta^i_l \left(1 - q^{-2}\right) \sum_{d<l}
   u^k_d \br u^d_j ,
\end{split}\label{eq:BqMn-rel-4}
\end{align}
for all $i<k$, $j<l$.
\end{lem}
\begin{proof}
This is proved by twisting the quadratic relations from \Cref{defn:OLqMn} by the twisting map of Equation \eqref{eq:twist-quadratic}.
The relations from \Cref{defn:OLqMn} to be twisted are: 
\begin{gather}
x^i_jx^i_l=q^{-1}x^i_lx^i_j\qquad (\forall 1\leq j<l\leq n),\label{eq:reltype1}\\
x^i_jx^k_j=q^{-1}x^k_jx^i_j\qquad (\forall 1\leq i<k\leq n)\label{eq:reltype2}\\
x^i_lx^k_j=x^k_jx^i_l\qquad (\forall 1\leq i<k\leq n,\quad 1\leq j<l\leq n )\label{eq:reltype3}\\ 
x^i_jx^k_l-x^k_lx^i_j=(q^{-1}-q)x^i_lx^k_j \qquad (\forall 1\leq i<k\leq n,\quad 1\leq j<l\leq n ),\label{eq:reltype4}
\end{gather}
where $i,j,k,l$ are pairwise distinct.
Using this formula, relations \eqref{eq:reltype1}--\eqref{eq:reltype3} twist to give relations \eqref{eq:BqMn-rel-1}--\eqref{eq:BqMn-rel-3}, respectively.  Twisting relation \eqref{eq:reltype4} gives an equation that simplifies to \eqref{eq:BqMn-rel-4}. 
\end{proof}

\begin{rmk}
Note that the dual quasi-triangular bialgebra $O_q(M_n)$ is \emph{not} a Hopf algebra. However, the covariantized algebra $B_q(M_n)$ is still well-defined, and commonly found in the literature, as the product formula only involves the antipode before applying the dual R-matrix. More precisely, one only needs that $\cR$ has a convolution inverse in the algebra $\Hom_\Bbbk (A\otimes A^{\cop}, \Bbbk)$ for the bialgebra $A$. If the antipode exists, this is given by $\cR(\Id\otimes S)$ but the convolution inverse might exist more generally, and the covariantized algebra $\un{A}$ is well defined. In the case of $A=O_q(M_n)$, which is denoted by $A(R)$ in \cite{Majid95book}, this is the algebra $B(R)$.
\end{rmk}

\begin{rmk}\label{rmk:compare-JW}
We can compare the relations \eqref{eq:BqMn-rel-1}--\eqref{eq:BqMn-rel-4} to \cite{JW}*{Equation (2.4)} by reversing the order of the index set $1,\ldots,n$. This is due to a difference in the form of the R-matrix between the one use in \cite{JW}*{Equation~(2.1)} which follows the conventions of \cite{KS} and \eqref{eq:R-matrix-sln} where we use the convention of \cite{Majid95book}. With this precise comparison to the conventions of \cite{JW}, we can apply their result \cite{JW}*{Corollary~6.4} giving a closed formula for the quantum determinant.
\end{rmk}

\begin{thm}[Jordan--White \cite{JW}]\label{thm:brdqn}
The quantum determinant $\dq{n}$ of \eqref{eq:qdet} corresponds to the element 
\begin{align}
\brdq{n}=\sum_{\sigma \in S_n} (-q)^{l(\sigma)}q^{e(\sigma)}u^{n}_{\sigma(n)}\br \ldots \br u^1_{\sigma(1)} \quad\in\quad B^\intf_q(M_n),\label{eq:brdq-n}
\end{align}
where $e(\sigma)=|\Set{i=1,\ldots, n~|~ \sigma(i)>i}|$.
\end{thm}
Note that $\brdq{n}$ is in the center of $B^\intf_q(M_n)$. Hence, a formally adjoined inverse will also be in the center. Thus, we derive the following complete presentations for the algebras $B^\intf_q(GL_n)$ and $B^\intf_q(SL_n)$.

\begin{cor}\label{cor:presBLq-SL-GL}
We have isomorphisms of $\mI$-algebras 
\begin{align*}
B^\intf_q(GL_n)&\cong B^\intf_q(M_n)[t]/\left(\brdq{n}t-1\right), \\ B^\intf_q(SL_n)&\cong B^\intf_q(M_n)/\left(\brdq{n}-1\right).
\end{align*}
\end{cor} 
\begin{proof}
The quantum determinant $d=\dq{n}$ is a central grouplike element, see Lemma~\ref{lem:dqn}, and so is its inverse $t$ in $O^\intf_q(GL_n)$. We denote the corresponding elements in $B^\intf_q(GL_n)$ by $\Psi(d)$ and $\Psi(t)$. Thus, when twisting we find that 
$$1=\Psi(dt)=\cR^{-1}(d,t)\cR(d,t)\Psi(d)\br \Psi(t)=\Psi(d)\br \Psi(t).$$
Thus, in $B^\intf_q(GL_n)$, the elements $\Psi(d)$ and $\Psi(t)$ are mutual inverses (even though $\Psi$ is not a morphism of algebras). Further, this shows that in $B^\intf_q(SL_n)$, the relation $\un{d}=1$ holds. 
\end{proof}

\begin{rmk}
Denote $G_n=M_n,GL_n$, or $SL_n$. Corollary \ref{cor:presBLq-SL-GL} directly gives presentations for the $\mK$-algebras algebras $B_q(G_n)$ or the $\mC$-algebras algebras $B_\epsilon(G_n):=B_q^\intf(G_n)\otimes_\mI \mC$, where $\epsilon$ is an odd root of unity in $\mC$, by extension of scalars along the map $\mI\to \mC$, $q\mapsto \epsilon$. 
\end{rmk}

\begin{ex}
Consider the case of $n=2$. Directly twisting the quantum determinant gives
\begin{align*}
\Psi(x^1_1x^2_2-q^{-1}x^1_2x^2_1)=&\;u^1_1\br u^2_2-(q^2-1)u^2_1\br u^1_2\\
&-u^1_2\br u^2_1-\left(1-q^{-2}\right)u^1_1\br u^1_1+\left(1-q^{-2}\right)u^2_2\br u^1_1\\
=&\;u^1_1\br u^2_2-(q^2-1)u^2_1\br u^1_2-u^2_1\br u^1_2\\
=&\;u^1_1\br u^2_2-q^{2}u^2_1\br u^1_2\\
=&\;u^2_2\br u^1_1-q^{2}u^2_1\br u^1_2\\
=&\;\brdq{2},
\end{align*}
using the relations
$$
u^2_1\br u^1_2-u^1_2\br u^2_1=\left(1-q^{-2}\right)u^1_1\br u^1_1-\left(1-q^{-2}\right)u^2_2\br u^1_1, 
$$
obtained from \eqref{eq:BqMn-rel-3}, and $u^1_1\br u^2_2=u^2_2\br u^1_1$ obtained from \eqref{eq:BqMn-rel-4}. Now, specializing all the relations from Lemma \ref{lem:BqMn-relations} to the case $n=2$, with $a=u^1_1$, $b=u^1_2$, $c=u^2_1$, and $d=u^2_2$, exactly recovers the presentation of $B_q(M_2)$ given in Example~\ref{eq:BqMn-GL2-SL2}, extending scalars to $\Bbbk$. Moreover, the quantum determinant is exactly the one used there and in \cite{Majid95book}*{Example~4.3.4} and we recover the resulting presentations for $B_q(GL_2)$ and $B_q(SL_2)$.
\end{ex}

\begin{ex}\label{ex:BqSL2}
In the case when $n=3$, the quantum determinant is given by 
\begin{align}\label{eq:brdq3}
\begin{split}
\brdq{3}=&\;u^3_3\br u^2_2 \br u^1_1- q^{2}u^3_3\br u^2_1 \br u^1_2- q^{2}u^3_1\br u^2_2 \br u^1_3\\
&- q^{2}u^3_2\br u^2_3 \br u^1_1+ q^{3}u^3_2\br u^2_1 \br u^1_3+ q^{4}u^3_1\br u^2_3 \br u^1_2.
\end{split}
\end{align}
Adding the relation $\brdq{3}=1$ to the algebra generated by the $u^i_j$, $1\leq i,j\leq 3$, subject to relations \eqref{eq:BqMn-rel-1}--\eqref{eq:BqMn-rel-4} gives a presentation for $B^\intf_q(SL_3)$.
\end{ex}

%

\subsection{Small reflection equation algebra at a root of unity}

This section contains Theorem~\ref{thm:main-theorem}, the main result of this paper. 
For this, let $\epsilon\in \mC$ denote a primitive $\ell$-th root of unity for $\ell$ odd. 

\begin{defn}[small reflection equation algebras]
We define the \emph{small reflection equation algebras} of type $GL_n$ and $SL_n$ to be 
$$b_\epsilon(GL_n):=\un{o_\epsilon(GL_n)} \quad \text{and}\quad b_\epsilon(SL_n):=\un{o_\epsilon(SL_n)},$$
using the covariantized algebra construction of Definition \ref{defn:cov-alg} and the small quantum function algebras from Definitions~\ref{pres-oqgln} and \ref{pres-oqsln}.

We may also define integral forms over the cylotomic integers $\mO=\mZ[\nu]$ as
$$b_\nu^\intf(GL_n):=\un{o_\nu^\intf(GL_n)} \quad \text{and}\quad b_\nu^\intf(SL_n):=\un{o_\nu^\intf(SL_n)}.$$
\end{defn}

Before giving presentations for the small reflection equation algebras which will display $b_\epsilon(GL_n)$ as a quotient of $B_\epsilon(GL_n)$ and $b_\epsilon(SL_n)$ as a quotient of $B_\epsilon(SL_n)$, we introduce some combinatorial notation.

\begin{notation}[Compositions]\label{notation:comp}
We let $\lambda\models N$ denotes a \emph{composition} of \emph{weight} $N$, i.e., a sequence $\lambda=(\lambda_1,\ldots, \lambda_k)\in \mZ_{\geq 1}^k$ of positive integers such that $\sum_{i=1}^k \lambda_i=N$. We say that the $\lambda_j$ are the \emph{parts} of $\lambda$ and their number, $k=|\lambda|$, is the \emph{length} of $\lambda$.

We also use the notation $\lambda_{-i}:=\lambda_{k+1-i}$ to denote the $i$-th last part of $\lambda$.
Moreover, we let $\lambda_{[1, -i]}$ denote the truncated composition $(\lambda_1, \dots, \lambda_{k+1-i})$ which deletes the final $i-1$ parts.
\end{notation}
We recall that each integer $n\geq 1$ has $2^{n-1}$ distinct compositions.

\begin{defn}[$\scalar_q(\lambda)$, $V^k(\lambda)$]  \label{def:k-extension}
Fix a composition $\lambda\models N$.
\begin{enumerate}
\item[(1)] For a generic parameter $q$, we define the \emph{$q$-scalar} associated to a composition $\lambda\models N$ to be
    \begin{align}
        \scalar_q(\lambda)=\scalar_q\left(\lambda\models N\right) = \frac{\prod_{j=1}^{N-1} \left( 1 - q^{-2(N-j)} \right)}{ \prod_{k=1}^{|\lambda|-1} \left( 1 - q^{-2 \left( N  - \sum_{j=1}^{k} \lambda_{-j} \right)} \right) }.
    \end{align}      
\item[(2)] For an integer $k \in \{1, \dots, n\}$, let $V^k(\lambda)=V^k(\lambda\models N)$ be the set of tuples $\bm{\beta}=(\beta_1,\ldots, \beta_{N+1})$ of length $N+1$ such that $\beta_{\sum_{i=1}^x \lambda_i + 1} = k$ for each $0 \leq x \leq |\lambda|$ and the remaining components $\beta_j$ are arbitrary integers $1\leq \beta_j<k$.
\end{enumerate}
\end{defn}
In other words, to specify an element of $V^k(\lambda)$, there are $|\lambda|+1$ determined components which equal $k$ while the remaining $N-|\lambda|$ components are freely chosen integers from $\{1, \ldots, k-1\}$. Thus, given a fixed positive integer $N$ and a composition $\lambda\models N$, the set $V^k(\lambda\models N)$ contains $(k-1)^{N-|\lambda|}$ distinct elements.

\begin{ex}
    The composition $\lambda = (3, 1, 2)\models 6$  has length $3$. Its $q$-scalar is 
\begin{align*}
\scalar_q(3, 1, 2)&=\frac{\left(1-q^{-2}\right)\left(1-q^{-4}\right)\ldots \left(1-q^{-10}\right)}{\left(1-q^{-8}\right)(1-q^{-6)})}=\left(1-q^{-2}\right)\left(1-q^{-4}\right)\left(1-q^{-10}\right)\\&=\left(1-q^{-10}\right)\scalar_q(3, 1), \qquad \text{where}\\
\scalar_q(3, 1)&=\frac{\left(1-q^{-2}\right)\left(1-q^{-4}\right)(1-q^{-6)})}{\left(1-q^{-6}\right)}=\left(1-q^{-2}\right)\left(1-q^{-4}\right)=\scalar_q(3).
\end{align*}
An example one of the $3^3$ elements of the set $V^4(3, 1, 2)$ is 
$(4, 2, 3, 4, 4, 1, 4)$.
\end{ex}

\begin{lem}\label{lem:lambda-scalar}
The $q$-coefficient associated with a composition $\lambda\models N$ is an element of $\mI=\mZ[q,q^{-1}]$ and satisfies the recursion
$$ \scalar_q\left(\lambda\right)= \left(\prod_{j=1}^{\lambda_{-1}-1} \left( 1 - q^{-2(N-j)} \right) \right) \scalar_q \left( \lambda_{[1,-2]}\right),$$
for the composition $\lambda_{[1,-2]}=(\lambda_1,\ldots, \lambda_{-2})\models N-\lambda_{-1}$.
\end{lem}
\begin{proof}
  To prove the recursive formula, we compute
    \begin{align*}
        \scalar_q\left(\lambda\right) &= \frac{\prod_{j=1}^{N-1} \left( 1 - q^{-2(N-j)} \right)}{ \prod_{k=1}^{|\lambda|-1} \left( 1 - q^{-2 \left( N - \sum_{j=1}^{k} \lambda_{-j} \right)} \right) } \\
        &= \frac{\prod_{j=1}^{\lambda_{-1}} \left( 1 - q^{-2(N-j)} \right) \prod_{j=\lambda_{-1}+1}^{N-1} \left( 1 - q^{-2(N-j)} \right)}{\left( 1 - q^{-2(N- \lambda_{-1})} \right) \prod_{k=2}^{|\lambda|-1} \left( 1 - q^{-2 \left( N - \sum_{j=1}^{k} \lambda_{-j} \right)} \right) } \\
        &= \frac{\prod_{j=1}^{\lambda_{-1}-1} \left( 1 - q^{-2(N-j)} \right) \prod_{j=1}^{N-\lambda_{-1} -1} \left( 1 - q^{-2(N-\lambda_{-1}-j)} \right)}{\prod_{k=1}^{|\lambda|-2} \left( 1 - q^{-2 \left( N - \lambda_{-1} - \sum_{j=2}^{k+1} \lambda_{-j} \right)} \right) } \\
                &= \frac{\prod_{j=1}^{\lambda_{-1}-1} \left( 1 - q^{-2(N-j)} \right) \prod_{j=1}^{N-\lambda_{-1} -1} \left( 1 - q^{-2(N-\lambda_{-1}-j)} \right)}{\prod_{k=1}^{|\lambda|-2} \left( 1 - q^{-2 \left( N - \lambda_{-1} - \sum_{j=1}^{k} \lambda_{-j-1} \right)} \right) } \\
        &= \left(\prod_{j=1}^{\lambda_{-1}-1} \left( 1 - q^{-2(N-j)} \right) \right) \scalar_q \left( \lambda_{[1,-2]}\right).
    \end{align*}
    
Finally, the claim that $\scalar_q(\lambda)\in \mI=\mZ[q,q^{-1}]$ now follows by induction on the length $|\lambda|$ by the formula just proved and that for a one-part composition $(N)\models N$, we have 
\[
\scalar_q(N)= \prod_{j=1}^{N-1}\left(1-q^{-2(N-j)}\right),
\]
with no denominator, since $\lambda_{-1}=N$ is the only part in this composition, i.e., $|\lambda|=1$.
\end{proof}

The above lemma means that the $q$-scalar $\scalar_q(\lambda)$ can be specialized to a root of unity $\epsilon$ or the generator $\nu$ of $\mO$ as $\scalar_\epsilon(\lambda)$ and $\scalar_\nu(\lambda)$ are well-defined.

\begin{thm}\label{thm:main-theorem}
The small reflection equation algebra $b_\epsilon(GL_n)$ is generated by $u^k_l$ for all $1\leq l,k\leq n$ subject to the relations \eqref{eq:BqMn-rel-1}--\eqref{eq:BqMn-rel-4} as well as the additional relations 
\begin{align}
\left(u^k_l\right)^{\br \ell}&=0,\label{eq:oepsilonGLn-rel1}\\
\sum_{\lambda\models \ell} \scalar_\epsilon (\lambda) \sum_{\bm{\beta}\in V^k(\lambda)} u^{\beta_1}_{\beta_2}\br \ldots \br u^{\beta_\ell}_{\beta_{\ell+1}}&=1,\label{eq:oepsilonGLn-rel2}
\end{align}
for all $1\leq k\neq l\leq n$, where the sum is taken over all compositions $\lambda$ of $\ell$ and all $\bm{\beta}=(\beta_1,\ldots, \beta_{\ell+1})\in V^k(\lambda)$. 

Moreover, the algebra $b_\epsilon(SL_n)$ is the quotient of $b_\epsilon(GL_n)$ by the additional relation that
\begin{align}
\brdeps{n}=\sum_{\sigma \in S_n} (-\epsilon)^{l(\sigma)}\epsilon^{e(\sigma)}u^{n}_{\sigma(n)}\br \ldots \br u^1_{\sigma(1)}&=1,\label{eq:oepsilonSLn-rel}
\end{align}
where $l(\sigma)$ is the number of inversions in $\sigma$ and $e(\sigma)=|\Set{i=1,\ldots, n~|~ \sigma(i)>i}|$.
\end{thm}

\begin{proof}
Using Definition \ref{pres-oqgln} and Lemma~\ref{lem:Tak-OGLn-quotient}, we know that we need to twist the relations 
$$\left(x^k_l\right)^\ell=0, \qquad \text{and}\qquad \left(x^k_k\right)^\ell=1,$$
for all $1\leq k\neq l\leq n$. In particular, generators and relations associated with invertibility of the quantum determinant are not needed by Lemma~\ref{lem:Tak-OGLn-quotient}, see \cite{Tak}. 

To show that the twists of the relations in Equation \eqref{eq:small-o-rel} are Equation~\eqref{eq:oepsilonGLn-rel1}--\eqref{eq:oepsilonGLn-rel2}, 
we use results proved in the next subsection but with the generic parameter $q$ replaced with the root of unity $\epsilon$:
Equation~\eqref{eq:oepsilonGLn-rel1} follows from Proposition~\ref{lem:u^k_l-powers} and  Equation~\eqref{eq:oepsilonGLn-rel2} follows from Proposition~\ref{prop:u^k_k-twisted}. Passing to $b_\epsilon(SL_n)$ we twist the additional relation $\dq{n}=1$ in $o_\epsilon(SL_n)$ which, by Theorem~\ref{thm:brdqn} is given by Equation~\eqref{eq:oepsilonSLn-rel}.
\end{proof}

\begin{rmk}
We observe that for the composition $\lambda=(1,\ldots, 1)=\left(1^\ell\right)\models \ell$ we have 
$\scalar_\epsilon\left(1^\ell\right)=1$. Consider, say,  a degree lexicographic order on the monomial basis of $b_\epsilon(GL_n)$ where $u^k_l>u^a_b$ if $(k,l)>(a,b)$. Then the leading term of the relation \eqref{eq:oepsilonGLn-rel2} is always given by $\left(u^k_k\right)^{\br \ell}$ with coefficient $1$. All other terms involve degree $\ell$-monomials in generators $u^a_b$ with $a,b\leq k$. 

We further note that for $k=1$ the relation \eqref{eq:oepsilonGLn-rel2} takes an easy form. It simply becomes $\left(u^k_k\right)^{\br \ell}=1$ as there are no indices $\beta_i<1$. 
\end{rmk}

\begin{rmk}
Working with the same generators and relations (replacing $\epsilon$ by $\nu\in \mO$) give presentations for the integral forms $b_\nu^\intf(GL_n)$ and  $b_\nu^\intf(SL_n)$ as $\mO$-algebras from Theorem~\ref{thm:main-theorem}.
\end{rmk}

\begin{lem}
For any composition $\lambda \models N$, the set $V^2(\lambda)$ is a singleton. 
\end{lem}
\begin{proof}
From the definition we see that for $\bm{\beta}\in V^2(\lambda)$ we necessarily have that
$$ \bm{\beta}=\left(2,1^{\lambda_{1}},2,1^{\lambda_{2}},2,\ldots, 1^{\lambda_{-1}},2\right),$$
showing that $|V^2(\lambda)|=1$.
\end{proof}

Recall that presentations for $B_q(GL_2)$ and $B_q(SL_2)$ were given in Example~\ref{eq:BqMn-GL2-SL2}. We can now add the remaining relations to give a presentation for the finite-dimensional quotients.

\begin{cor}[$b_\epsilon(GL_2)$ and $b_\epsilon(SL_2)$]\label{cor:bepGL2-SL2}
Consider the case when $n=2$. 
\begin{enumerate}
\item
The algebra $b_\epsilon(GL_2)$ is generated by $a=u^1_1$, $b=u^1_2$, $c=u^2_1$, $d=u^2_2$ subject to the relations 
\eqref{eq:BqM2-rel1}--\eqref{eq:BqM2-rel3} and
\begin{align}
a^{\br \ell}&=1, &b^{\,\br \ell}&=0, &c^{\br \ell}&=0,\label{eq:bqGL2-rel1}
\end{align}
as well as the more involved relation
\begin{align}
\sum_{\lambda\models \ell} \scalar_\epsilon (\lambda) \monom(\lambda_{-1})\br \monom(\lambda_{-2})\br \ldots \br \monom(\lambda_1)&=1,\label{eq:bqGL2-rel2}
\end{align}
where for an integer $l\geq 1$,
$$\monom(l)=\begin{cases}
d, & \text{if $l=1$,}\\
c\br a^{\br (l-2)}\br b, & \text{if $l>1$.}
\end{cases}$$
\item The algebra $b_\epsilon(SL_2)$ is the quotient of $b_\epsilon(GL_2)$ by the additional relation 
\begin{align}\label{eq:brdet-1}
a\br d-\epsilon ^2c\br b&=1. 
\end{align}
\end{enumerate}
\end{cor}

\begin{ex}
We compute the more elaborate relation \eqref{eq:bqGL2-rel2} for cases of small values of $\ell$. For $\ell=3$ we have
\begin{align*}
1=d^{\br 3} + \left(1-\epsilon^{-2}\right)d\br c\br b + \left(1-\epsilon^{-4}\right)c\br b\br d + \left(1-\epsilon^{-2}\right)\left(1-\epsilon^{-4}\right)c\br a\br b.
\end{align*}
For $\ell=5$, we find
\begin{align*}
1&=d^{\br 5}+ \left(1-\epsilon^{-2}\right)d^{\br 3}\br c\br b + \left(1-\epsilon^{-4}\right)d^{\br 2}\br c\br b\br d  + \left(1-\epsilon^{-6}\right)d\br c\br b\br d^{\br 2} + \left(1-\epsilon^{-8}\right)c\br b\br d^{\br 3}\\
&+ \left(1-\epsilon^{-2}\right)\left(1-\epsilon^{-4}\right)d^{\br 2}\br c\br a\br b + \left(1-\epsilon^{-4}\right)\left(1-\epsilon^{-6}\right)d \br c\br a\br b \br d+ \left(1-\epsilon^{-6}\right)\left(1-\epsilon^{-8}\right)c\br a\br b\br d^{\br 2}\\
&+ \left(1-\epsilon^{-2}\right)\left(1-\epsilon^{-6}\right)(c\br b)^{\br 2}\br d+ \left(1-\epsilon^{-2}\right)\left(1-\epsilon^{-8}\right)c\br b\br d\br c\br b +\left(1-\epsilon^{-4}\right)\left(1-\epsilon^{-8}\right)d\br (c\br b)^{\br 2}\\
&+\left(1-\epsilon^{-2}\right)\left(1-\epsilon^{-4}\right)\left(1-\epsilon^{-8}\right)c\br a\br b\br c\br b+\left(1-\epsilon^{-2}\right)\left(1-\epsilon^{-6}\right)\left(1-\epsilon^{-8}\right)c\br  b\br c\br a\br b\\
&+\left(1-\epsilon^{-2}\right)\left(1-\epsilon^{-4}\right)\left(1-\epsilon^{-6}\right)d\br c\br a^{\br 2}\br b
+\left(1-\epsilon^{-4}\right)\left(1-\epsilon^{-6}\right)\left(1-\epsilon^{-8}\right)c\br a^{\br 2}\br b\br d\\
&+ \left(1-\epsilon^{-2}\right)\left(1-\epsilon^{-4}\right)\left(1-\epsilon^{-6}\right)\left(1-\epsilon^{-8}\right)c\br a^{\br 3}\br b.
\end{align*}
\end{ex}

\begin{ex}
More generally, consider the case when $n\geq 2$ and $\ell=3$. For $k\leq n$, the sets $V^k(\lambda)$ from Definition \ref{def:k-extension} are
\begin{gather*}
V^k(1,1,1)=\Set{(k,k,k,k)},\qquad V^k(1,2)=\Set{(k,k,i,k)~|~1\leq i< k}, \\ V^k(2,1)=\Set{(k,i,k,k)~|~1\leq i< k},\quad
V^k(3)=\Set{(k,i,j,k)~|~1\leq i,j< k}.
\end{gather*}
Thus, relation \eqref{eq:oepsilonGLn-rel2} specializes to 
\begin{align}
\begin{split}
1=\left(u^k_k\right)^{\br 3}&+\left(1-\epsilon^{-2}\right) \sum_{i<k} u^k_i \br u^i_k\br u^k_k + \left(1-\epsilon^{-4}\right)\sum_{i<k} u^k_k\br u^k_i\br u^i_k \\ &+ \left(1-\epsilon^{-2}\right)\left(1-\epsilon^{-4}\right)\sum_{i,j<k}u^k_i\br u^i_j\br u^j_k.
\end{split}
\end{align}
For instance, we have that
\begin{align*}
1=&\left(u^1_1\right)^{\br 3},\quad\\
1=&\left(u^2_2\right)^{\br 3}+\left(1-\epsilon^{-2}\right) u^2_1 \br u^1_2\br u^2_2 + \left(1-\epsilon^{-4}\right) u^2_2\br u^2_1\br u^1_2+ \left(1-\epsilon^{-2}\right)\left(1-\epsilon^{-4}\right)u^2_1\br u^1_1\br u^1_2.
\end{align*}
The number of monomials with non-zero coefficients in these relations is given by $1+(2^{\ell-1}-1)(k-1),$ which grows exponentially in $\ell$, the order of the root of unity.
\end{ex}


\subsection{Details of the proof of \texorpdfstring{Theorem~\ref{thm:main-theorem}}{Theorem 4.15}}

In the following subsection, we present a number of results leading up to Propositions \ref{lem:u^k_l-powers} and \ref{prop:u^k_k-twisted} which are used in the proof of this paper's main theorem, Theorem~\ref{thm:main-theorem}. Most generally, these statements are valid in the algebra $B^\intf_q(M_n)$ which is based on the dual R-matrix $\cR$ of Equation~\eqref{eq:R-matrix-sln} taken from \cite{Majid95book}*{Example~4.5.7}. We use indexed products 
$$\prod_{i=1}^j x^{k_i}_{l_i}=x^{k_1}_{l_1}\ldots x^{k_j}_{l_j} \qquad \text{or} \qquad \unprod_{i=1}^N u^{k_i}_{l_i}=u^{k_1}_{l_1}\br \ldots \br u^{k_j}_{l_j},$$
which are understood as ordered products in $O^\intf_q(M_n)$, respectively, $B^\intf_q(M_n)$ when using the symbol $\unprod$. We use the standard convention that empty sums are zero and empty products, in algebras, equal the unit element $1$.

\begin{prop} 
    \label{prop:braid_formula}
   For $i=1,\ldots, N+1$ and all $1\leq k_i,l_i\leq n$, we have the following equality in $B^\intf_q(M_n)$:
    \begin{align*}
       \Psi\left( \left( \prod_{i=1}^N x^{k_i}_{l_i} \right) x^{k_{N+1}}_{l_{N+1}}\right) =& \sum_{\alpha_i, \beta_i, \gamma_i, \delta_i} \left( \prod_{i=1}^N \cR \left(S \left(x^{k_{N-i+1}}_{\alpha_{N-i+1}}\right), x^{\gamma_i}_{\gamma_{i+1}} \right) \cR\left( x^{\beta_i}_{l_i}, x^{\delta_i}_{\delta_{i+1}} \right) \right) \\
       &\cdot \Psi\left(\left( \prod_{i=1}^N x^{\alpha_i}_{\beta_i} \right)\right) \br u^{\delta_{N+1}}_{l_{N+1}}\,,
    \end{align*}
where summation is taken over all $1\leq \alpha_{i},\gamma_i,\beta_i,\delta_i\leq n$ satisfying $\gamma_1 = k_{N+1}$ and $\gamma_{N+1} = \delta_1$.
\end{prop}

\begin{proof}
    Let $a = \prod_{i=1}^N x^{k_i}_{l_i}$ and $b = x^{k_{N+1}}_{l_{N+1}}$.  To find the value of the twisting map, $\Psi(ab)$, in the notation of \eqref{eq:twistingmap}, we use the formula in Equation~\eqref{eq:twist-iteration}, and fix notation for the iterated coproduct of $a$ and $b$: 
    \[
        \sum b_{(1)} \otimes b_{(2)} \otimes b_{(3)}  =
        \sum_{\lambda, \mu} x^{k_{N+1}}_{\lambda} \otimes x^{\lambda}_{\mu} \otimes x^{\mu}_{l_{N+1}}
    \]
    and 
    \[ 
        \sum a_{(1)} \otimes a_{(2)} \otimes a_{(3)} 
        = \sum_{\alpha_i, \beta_i} \left( \prod_{i=1}^N x^{k_i}_{\alpha_i} \right) 
        \otimes \left(\prod_{i=1}^N x^{\alpha_i}_{\beta_i} \right) 
        \otimes \left(\prod_{i=1}^N x^{\beta_i}_{l_i}\right)
    \]
    where for the second equality we have used that the coproduct is a morphism of algebras. Substituting $a_{(i)}$ and $b_{(i)}$ into Equation~\eqref{eq:brproduct2} or Equation~\eqref{eq:twist-iteration} gives
    \begin{equation}
    \label{eq: ab inter}
    \begin{split}
    \Psi(a b) =& \sum_{\alpha_i,\beta_i,\lambda,\mu}\cR\left(S \left( \prod_{i=1}^N x^{k_i}_{\alpha_i} \right),  x^{k_{N+1}}_{\lambda}\right)  \cR \left(\left(\prod_{i=1}^N x^{\beta_i}_{l_i}\right), x^{\lambda}_{\mu} \right) \\
    &\cdot \Psi\left(\left(\prod_{i=1}^N x^{\alpha_i}_{\beta_i} \right)\right) \br u^{\mu}_{l_{N+1}}.
    \end{split}
    \end{equation}
Using the axioms of the dual R-matrix, specifically,
    \begin{equation}
        \label{dualR property}
        \cR(xy \otimes z) = \cR\left(x, z_{(1)}\right) \mathcal{R}\left(y, z_{(2)}\right), \end{equation}
we can simplify the middle factor of the right-hand side of \cref{eq: ab inter} to
\[\cR \left(\left(\prod_{i=1}^N x^{\beta_i}_{l_i}\right), x^{\lambda}_{\mu} \right) = \sum_{\delta_i}\prod_{i=1}^N \cR \left( x^{\beta_i}_{l_i}, x^{\delta_i}_{\delta_{i+1}} \right)\]
where $\lambda = \delta_1$ and $\mu = \delta_{N+1}$.
Finally, since the antipode is an anti-algebra map, we can use \cref{dualR property} once more to simplify the first factor in the right-hand side of \cref{eq: ab inter} to 
\begin{align*}
    \cR\left(S \left( \prod_{i=1}^N x^{k_i}_{\alpha_i} \right),  x^{k_{N+1}}_{\lambda}\right) =& \cR\left( \prod_{i=1}^N S \left( x^{k_{N-i+1}}_{\alpha_{N-i+1}} \right),  x^{k_{N+1}}_{\lambda}\right) \\
    =&\sum_{\gamma_i} \prod_i \cR\left(S \left( x^{k_{N-i+1}}_{\alpha_{N-i+1}} \right),  x^{\gamma_i}_{\gamma_i+1}\right)
\end{align*}
where $\gamma_1 = k_{N+1}$ and $\gamma_{N+1} = \lambda$. Combining these simplifications, we arrive at the result.
\end{proof}

\begin{prop}\label{lem:u^k_l-powers}
    If $k \neq l$ then, for all $N\geq 0$,
    \[\Psi\left(\left( x^k_l \right)^N\right) = q^{-\frac{1}{2}N(N-1)} \left( u^k_l \right)^{\br N}\] 
\end{prop}

\begin{proof}
    To prove this identity, we begin by noting that the dual R-matrix $\cR$ is relatively sparse, with non-zero values only on the diagonal, i.e.,
$$\cR(x^i_i\otimes x^j_j)=R^{ij}_{ij}=q^{-\frac{1}{n}}q^{\delta^i_j},$$    
or the single off-diagonal entries 
     $$\cR\left(x^i_j \otimes x^j_i\right)=R^{ij}_{ji}=\left(q-q^{-1}\right)\qquad \text{for}\qquad j>i.$$
This means that in the formula for $\Psi\left(\left( \prod_{i=1}^N x^{k_i}_{l_i} \right) x^{k_{N+1}}_{l_{N+1}}\right)$ given in \Cref{prop:braid_formula}, many of the summands are zero. We proceed by analysing the conditions for which the summands are non-zero to find that with the specific values of $l_i = l$ and $k_i = k$, for $i=1,\ldots, N$, with $l \neq k$, there is a unique non-zero summand appearing in the formula from Proposition \ref{prop:braid_formula}. 
    
   For each $i=1,\ldots, N-1$, if $\gamma_i = k$ then the term     $\cR \left(S \left(x^{k}_{\alpha_{N-i}}\right), x^{\gamma_i}_{\gamma_{i+1}} \right)$ is only non-zero if both $\alpha_{N-i} = k$ and $\gamma_{i+1} = \gamma_i=k$. As $\gamma_1 = k$ this implies that $\alpha_i = k$ and $\gamma_i = k$ for all $i=1,\ldots, N$.
   
Again, for each $i=1,\ldots, N-1$, if $\delta_i = k$ then as $k \neq l$, the other term, $\cR\left( x^{\beta_i}_{l}, x^{\delta_i}_{\delta_{i+1}} \right)$, is only non-zero if both $\beta_i = l$ and $\delta_i = \delta_{i+1}=k$. As $\gamma_1 = \gamma_N =k$ we find that $\gamma_{N}=\delta_1=l$ and this implies that $\beta_i = l$ and $\delta_i = k$ for all $i=1,\ldots, N$.
    Hence, we have 
    \begin{align*}
        \Psi\left(\left( x^{k}_{l} \right)^N\right) =& \sum_{\alpha_i, \beta_i, \gamma_i, \delta_i} 
        \left( \prod_{i=1}^{N-1} \cR \left(S \left(x^{k}_{\alpha_{N-i}}\right), x^{\gamma_i}_{\gamma_{i+1}} \right) \cR\left( x^{\beta_i}_{l}, x^{\delta_i}_{\delta_{i+1}} \right) \right)\\
        &\cdot \Psi\left(\left( \prod_{i=1}^{N-1} x^{\alpha_i}_{\beta_i}\right)\right) \br u^{\delta_{N}}_{l} \\
        =& \left( \cR \left(S \left(x^{k}_{k}\right), x^{k}_{k} \right) \cR\left( x^{l}_{l}, x^{k}_{k} \right) \right)^{N-1}\Psi\left(\left( x^{k}_{l} \right)^{N-1}\right) \br u^{k}_{l} \\
        =& q^{-(N-1)} \Psi\left(\left( x^{k}_{l} \right)^{N-1}\right) \br u^{k}_{l}.
    \end{align*}
An inductive argument on $N$ shows that
$$\Psi\left(\left( x^{k}_{l} \right)^N\right)=q^{-(N-1)}q^{-(N-2)}\ldots q^{-1} \left(u^{k}_{l}\right)^{\br N}=q^{-\frac{1}{2}(N-1)N}\left(u^{k}_{l}\right)^{\br N}
$$
and completes the proof. 
\end{proof} 

\begin{prop}
    \label{prop: u^k_k recursion}
For any $N\geq 1$ and $1\leq k,l \leq  n$ we have 
    \[\Psi\left(\left( x^k_k \right)^N x^k_l\right) = \Psi\left(\left(x^k_k \right)^N\right) \br u^k_l + \left(1 -q^{-2N}\right) \sum_{\beta < k} \Psi\left(\left( x^k_k \right)^{N-1} x^k_{\beta}\right) \br u^{\beta}_l\]
\end{prop} 

\begin{proof}
    In this proof, we also employ the approach of using \cref{prop:braid_formula} and identifying non-zero summands used in the proof of Proposition \ref{lem:u^k_l-powers}. We begin by examining the product
\[
    \prod_{i=1}^N \cR \left(S \left(u^{k}_{\alpha_{N-i+1}}\right), u^{\gamma_i}_{\gamma_{i+1}} \right).    
\]
If $\gamma_i = k$, for the $i$-th factor $\cR \left(S \left(u^{k}_{\alpha_{N-i+1}}\right), u^{\gamma_i}_{\gamma_{i+1}} \right)$ to be non-zero, it must be diagonal with $\gamma_{i+1} = \alpha_{N-i+1} = k$. Since we have that $\gamma_1 = k$, we conclude that
\[\prod_{i=1}^N \cR \left(S\left( u^{k}_{\alpha_{N-i+1}}\right), u^{\gamma_i}_{\gamma_{i+1}} \right) = \prod_{i=1}^N \cR \left(S\left( u^{k}_{k}\right), u^{k}_{k} \right) = q^{N(1/n - 1)}\]
with $\alpha_1=\ldots=\alpha_N = k$ and $\gamma_1=\ldots=\gamma_{N+1} = k$.

Now we consider
\[\prod_{i=1}^N \cR\left( u^{\beta_i}_{k}, u^{\delta_i}_{\delta_{i+1}} \right).\]
There are two cases: 
\begin{enumerate}
\item[(1)] either there is a value of $j \geq 1$ such that $\beta_i = k$ for all $i < j$ and $\beta_j<k$, or 
\item[(2)] there exists no such $j$ and $\beta_i = k$ for all $1 \leq i \leq N$.
\end{enumerate}
Case\ (1): Assume that there exists such a $j$. Given that $\delta_1 = \gamma_{N+1} = k$, the only way for the terms $\cR\left( u^{\beta_i}_{k}, u^{\delta_i}_{\delta_{i+1}} \right)$, for $i<j$, to be non-zero is to have $\delta_i = k$ for all $1 \leq i \leq j$. This forces the  $j$-th term to be $\cR\left( u^{\beta_j}_{k}, u^{k}_{\beta_j} \right)$ with $\delta_{j+1}=\beta_j$. This in turn forces the terms $m > j$ to all be $\cR\left( u^{k}_{k}, u^{\beta_j}_{\beta_j} \right)$ with $\beta_m = k$ and $\delta_m =\beta_j$. Hence, we have
\begin{align*}
    \prod_{i=1}^N \cR\left( u^{\beta_i}_{k}, u^{\delta_i}_{\delta_{i+1}} \right) &=
    \left( 
    \prod_{i=1}^{j-1} \cR\left( u^{k}_{k}, u^{k}_{k} \right)
    \right) 
    \cR\left( u^{\beta_j}_{k}, u^{k}_{\beta_j} \right) 
    \left( 
    \prod_{i=j+1}^{N} \cR\left( u^{k}_{k}, u^{\beta_j}_{\beta_j} \right)
    \right) \\
    &= \left(q - q^{-1}\right) q^{-N/n + j-1}.
\end{align*}
Case\ (2): Now we assume that no such $j$ exists and so we have $\beta_i = k$ for all $1 \leq i \leq N$ we must have 
\[
    \prod_{i=1}^N \cR\left( u^{k}_{k}, u^{\delta_j}_{\delta_{j+1}} \right) = q^{N(-1/n + 1)}
\]
with $\delta_i = k$ for all $1 \leq i \leq N$. This follows as $\delta_1=k$ and non-zero entries only occur if $\delta_i=\delta_{i+1}$. 

By combining all the above information and using the relation $u^k_{\beta} u^k_k = q^{-1} u^k_k u^k_{\beta}$ for $\beta < k$, we conclude that
\begin{align*}
        \Psi\left(\left(x^k_k\right)^{N}x^k_l\right)
        =& q^{N(1/n - 1)} q^{N(-1/n + 1)} \Psi\left(\left(u^k_k\right)^N\right) \br u^k_l \\
        &+ q^{N(1/n - 1)} \left(q - q^{-1}\right) \sum_{\beta < k} \sum_{j=1}^N q^{-N/n + j-1} \Psi\left(\left( x^k_k \right)^{j-1} x^k_{\beta} \left( x^k_k\right)^{N-j} \right) \br u^{\beta}_l\\
        =&  \Psi\left(\left(x^k_k\right)^N\right) \br u^k_l \\
        &+ q^{N(1/n - 1)} \left(q - q^{-1}\right) \sum_{\beta < k} \sum_{j=1}^N q^{-N/n + j-1 - N+j} \Psi\left(\left( x^k_k \right)^{N-1} x^k_{\beta} \right)\br u^{\beta}_l \\
        =& \Psi\left( \left(x^k_k\right)^N\right) \br u^k_l \\
        &+ q^{-2N -1} \left(q - q^{-1}\right) \sum_{\beta < k} \left(\sum_{j=1}^N q^{2j}\right) \Psi\left(\left( x^k_k \right)^{N-1} x^k_{\beta} \right)\br u^{\beta}_l \\
        =&  \Psi\left(\left(x^k_k\right)^N\right)\br u^k_l + \left(1-q^{-2N}\right) \sum_{\beta < k} \Psi\left(\left( x^k_k \right)^{N-1} x^k_{\beta}\right) \br u^{\beta}_l,
\end{align*}
where the first equality uses Proposition~\ref{prop:braid_formula}. This completes the proof.
\end{proof}

\begin{lem} For all $N\geq 1$, $k,l=1,\ldots, n$, we have
    \label{lem:1}
    \[\Psi\left(\left( x^k_k \right)^N x^k_l \right)= \sum_{i=0}^N \sum_{\bm{\beta}}\left( \prod_{j=1}^i \left( 1- q^{-2(N+1-j)} \right) \right) 
   \Psi\left( \left( x^k_k \right)^{N-i}\right) \br \unprod_{j=1}^{i} u^{\beta_j}_{\beta_{j+1}} \br u^{\beta_{i+1}}_l\,,
    \]
    where  the sum is taken over all $\bm{\beta}=(\beta_1,\ldots, \beta_{i+1})$ with $\beta_1 = k$  and $\beta_j < k$ for all $1<j \leq i+1$.
\end{lem}
\begin{proof}
    We proceed by induction on $N$. In the base case $N=1$ the formula specializes to
$$    
\Psi\left(x^k_kx^k_l \right)=u^k_k\br u^k_l + \sum_{\beta<k}\left(1-q^{-2}\right)u^k_\beta \br u^\beta_l,
$$  
which holds by Equation~\eqref{eq:twist-quadratic}. Now assume $N\geq 2$ and that the statement holds for smaller powers. 
     By \cref{prop: u^k_k recursion} we have 
    \begin{align*}
        \Psi\left(\left( x^k_k \right)^N x^k_l \right)&= \Psi\left(\left(x^k_k \right)^N\right) \br u^k_l + \left(1 -q^{-2N}\right) \sum_{\beta < k} \Psi\left(\left( x^k_k \right)^{N-1} x^k_{\beta}\right) \br u^{\beta}_l
    \end{align*}
    and using the induction hypothesis then gives
    \begin{align*}
        \Psi\left(\left( x^k_k \right)^N x^k_l\right) =& \Psi\left(\left(x^k_k \right)^N\right) \br u^k_l + \left(1 -q^{-2N}\right) \sum_{\bm{\beta}}\sum_{i=0}^{N-1}\sum_{\beta_{i+2}<k} \left( \prod_{j=1}^i \left( 1- q^{-2(N-j)} \right) \right) \\&
        \cdot\Psi\left(\left( x^k_k \right)^{N-1-i}\right) \br \unprod_{j=1}^{i+1} u^{\beta_j}_{\beta_{j+1}} \br u^{\beta_{i+2}}_l.
    \end{align*}
    Changing the indexing to increase $i$ by $1$, by appending $\beta_{i+2}$ to $\bm{\beta}$, yields
    \begin{align*}
        \Psi\left(\left( x^k_k \right)^N x^k_l\right) =& \Psi\left(\left(x^k_k \right)^N\right) \br u^k_l + \sum_{\bm{\beta}}\sum_{i=1}^{N} \left( \prod_{j=1}^{i} \left( 1- q^{-2(N+1-j)} \right) \right) \\
        &\cdot
        \Psi\left(\left( x^k_k \right)^{N-i}\right) \br \unprod_{j=1}^{i} u^{\beta_j}_{\beta_{j+1}} \br u^{\beta_{i+1}}_l \\
        =& \sum_{\bm{\beta}}\sum_{i=0}^{N} \left( \prod_{j=1}^{i} \left( 1- q^{-2(N+1-j)} \right) \right) 
        \Psi\left(\left( x^k_k \right)^{N-i}\right) \br \unprod_{j=1}^{i} u^{\beta_j}_{\beta_{j+1}} \br u^{\beta_{i+1}}_l.
    \end{align*}
    and proves the claim.
\end{proof}

We now recall Notation~\ref{notation:comp} for compositions $\lambda\models N$ of $N$ and Definition~\ref{def:k-extension} for the associated $q$-scalars $\scalar_q(\lambda)$ and the set $V^k(\lambda)$.

\begin{prop}\label{prop:u^k_k-twisted}
    For all integers $N\geq 2$ and $1\leq k\leq n$, the equality
    \[
        \Psi\left(\left( x^k_k \right)^N\right) = \sum_{\lambda\models N} \scalar_q(\lambda) \sum_{\bm{\beta}\in V^k(\lambda)} \unprod_{i=1}^{N} u^{\beta_j}_{\beta_{j+1}},
    \]
holds in $B_q^\intf(M_n)$. Here, the sum is taken over all compositions $\lambda$ of $N$ and all elements $\bm{\beta}=(\beta_1,\ldots, \beta_{N+1})\in V^k(\lambda)$, see Definition~\ref{def:k-extension}.
\end{prop}
\begin{proof}
    We proceed by induction on $N$. We start with the induction base $N=2$. In this case, 
    $\lambda=(1,1), (2)$ are the only possible partitions and $\scalar_q(1,1)=1$, $\scalar_q(2)=1-q^{-2}$. Moreover, we have that 
    $$V^k(1,1)=\Set{(k,k,k)}, \qquad V^k(2)=\Set{(k,\beta,k)~|~ 1\leq \beta<k}.$$
    Thus, the claimed formula gives 
    $$
    \Psi\left(\left(x^k_k\right)^2\right)=u^k_k\br u^k_k + \left(1-q^{-2}\right)\sum_{\beta<k}u^k_\beta\br u^\beta_k.
    $$
    This matches the expression for $\Psi\left(\left(x^k_k\right)^2\right)$ obtained directly from Equation~\eqref{eq:twist-quadratic}.
    
    Now assume $N \geq 3$ is an integer and that the claim holds for smaller powers. By \cref{lem:1}, we have
\begin{align*}
 \Psi\left(\left( x^k_k \right)^N\right)= \sum_{\bm{\mu}}\sum_{i=1}^N \left( \prod_{j=1}^{i-1} \left( 1- q^{-2(N-j)} \right) \right) 
   \Psi\left( \left( x^k_k \right)^{N-i}\right) \br \unprod_{j=1}^{i-1} u^{\mu_j}_{\mu_{j+1}} \br u^{\mu_{i}}_k\,.
\end{align*}
where $\mu_1 = k$  and $\mu_j < k$ for all $1<j \leq i$.
Applying the induction hypothesis gives 
\begin{align*}
 \Psi\left(\left( x^k_k \right)^N\right)= & \sum_{\bm{\mu}}\sum_{i=1}^N \left( \prod_{j=1}^{i-1} \left( 1- q^{-2(N-j)} \right) \right) \\
& \cdot 
   \left(\sum_{\lambda'\models N-i} \scalar_q(\lambda') \sum_{\bm{\beta}\in V^k(\lambda')} \unprod_{i=1}^{N-i} u^{\beta_j}_{\beta_{j+1}}\right) \br \unprod_{j=1}^{i-1} u^{\mu_j}_{\mu_{j+1}}\,
\end{align*}
    where the first summation runs over all $\bm{\mu}=(\mu_1,\ldots, \mu_i)$ with $\mu_j < k$ for $1<k<i$ and $\mu_1 = k = \mu_{i}$. 

Observe that the space of compositions is stratified as the disjoint union
    \[
        \left\{ \lambda ~|~ \lambda\models N\right\} =\bigsqcup_{i=1}^N  \left\{ (p_1,\ldots, p_k, i) ~|~ k\geq 1, \text{ and } (p_1,\ldots, p_k)\models N-i\right\}.
    \]
    Thus, we can extend the given composition $\lambda'$ of $N-i$ by appending $i$ and get all compositions $\lambda\models N$ of $N$ ending with $i$. Using this bijection in the above formula for re-indexing, we have that 
    \begin{align*}
        \Psi\left(\left( u^k_k \right)^N \right) 
        =& 
        \sum_{\bm{\mu}}\sum_{\lambda\models N} \left( \prod_{j=1}^{\lambda_{-1}-1} \left( 1- q^{-2(N-j)} \right) \right) \scalar_q\left(\lambda_{[1, -2]}\right)  \\
        &\cdot \sum_{\bm{\beta}\in V^k\left(\lambda_{[1,-2]}\right)} \unprod_{j=1}^{N-\lambda_{-1}} u^{\beta_j}_{\beta_{j+1}} \br \unprod_{j=1}^{\lambda_{-1}-1} u^{\mu_j}_{\mu_{j+1}}\\
        =& \sum_{\lambda\models N} \left( \prod_{j=1}^{\lambda_{-1}-1} \left( 1- q^{-2(N-j)} \right) \right) \scalar_q\left(\lambda_{[1, -2]}\right) \sum_{\bm{\beta}\in V^k\left(\lambda\right)} \unprod_{j=1}^{N} u^{\beta_j}_{\beta_{j+1}}.
    \end{align*}
Here, the second equality we re-index again, by extending the given $\bm{\beta}\in V^k(\lambda_{[1,-2]})$ by the given $\bm{\mu}$. That is, we set $\beta_{N-\lambda_{-1}+1+j}:=\mu_j$, for $j=1,\ldots, \lambda_{-1}$, i.e.
$$(\beta_1,\ldots, \beta_{N-\lambda_{-1}+1},\mu_1,\ldots, \mu_{\lambda_{-1}})=(\beta_1,\ldots, \beta_{N+1})=\bm{\beta}\in V^k(\lambda\models N),$$
where now $\beta_{N-\lambda_{-1}+2}=\mu_1=k$ and $\beta_{N+1}=\mu_{\lambda_{-1}}=k$. Thus, the extended string on the right hand side defines an element in $V^k(\lambda)$. We observe that any element  of $V^k(\lambda)$ is obtained as such an extension from an element $V^k(\lambda_{[1,-2]})$, since all possible $\beta_{N-\lambda_{-1}+3}=\mu_2, \ldots, \beta_{N}=\mu_{\lambda_{-1}-1}<k$ appeared in the summation over $\mu$. Finally, 
$$\Psi\left(\left( u^k_k \right)^N \right) 
        = \sum_{\lambda\models N} \scalar_q\left(\lambda\right) \sum_{\bm{\beta}\in V^k\left(\lambda\right)} \unprod_{j=1}^{N} u^{\beta_j}_{\beta_{j+1}},$$
using Lemma~\ref{lem:lambda-scalar}. This completes the proof. 
\end{proof}

\appendix

\section{Cocompletion of \texorpdfstring{$R$}{R}-linear tensor categories}
\label{sec:appendix}

In this section, we collect some facts on locally finitely presentable categories with focus on cocompletions of $R$-linear monoidal categories with cocontinuous tensor products building on work of \cites{Lyu3}. For the general theory of locally finitely presentable categories, we refer to \cite{AR}, with \cite{Kel2} for the enriched case. 
We assume that all categories are skeletally small, Kar\-ou\-bi\-an (i.e., additive and idempotent complete), and $R$-linear over a non-zero commutative ring $R$. The symmetric monoidal category of $R$-modules is denoted by $\lMod{R}$.

\subsection{Categories enriched over a commutative ring}

The category $\lMod{R}$ is closed with respect to the internal hom functor 
$$\Hom_R(V\otimes_R W, U)\cong \Hom_R(V,\iHom(W,U)),$$
where $\iHom(W,U)$ is the space of maps of abelian groups $f\colon W\to U$, with $R$-action given by $(r\cdot f)(w)=rf(w),$ and $\Hom_R$ denotes the abelian group of $R$-linear maps. Thus, $\lMod{R}$ is a cartesian closed, complete and cocomplete, symmetric monoidal category.

Let $\cC$ be a category enriched over $R$-modules. An appropriate concept of (co)limit taking into account the enriched structure is that of a \emph{conical (co)limit} \cites{Kel, Rie}. Given (unenriched) indexing category $\cD$ with an (unenriched) diagram functor $D\colon \cD\to \cC$, the \emph{(conical) limit} $\lim D$ is an object in $\cC$ such that  
for every object $X$ in $\cC$, there is a natural isomorphism of $R$-modules 
$$\Hom_\cC(X,\lim D)\cong \Hom_{\Fun(\cD,\cC)}(R,\Hom_\cC(X,F(-)).$$
This means that there is an $R$-linear isomorphism between morphisms $\eta\colon X\to \lim D$ in $\cC$ and diagrams $(\eta_d)_{d\in \cD}$ of morphisms $\eta_d\colon X\to D(d)$ such that for any morphism $f\colon d\to d'$ in $\cD$, the diagram 
$$\xymatrix{
X\ar[rr]^{\eta_d}\ar[dr]^{\eta_{d'}}&&D(d)\\
&D(d')\ar[ur]^{D(f)}&
}$$
commutes. Conical colimits are defined dually.

Assume that the $R$-linear category $\cC$ is \emph{tensored} over $\lMod{R}$, i.e., that $\cC$ comes with a categorical action 
$$\lMod{R}\times \cC\to \cC, \qquad (V,X)\mapsto V\otimes_R X, $$
with natural isomorphisms
$$\Hom_\cC(V\otimes_R X,Y)\cong \Hom_{R}(V,\Hom_\cC(X,Y)).$$
Further, we assume that that $\cC$ is \emph{cotensored} over $\lMod{R}$, i.e., that $\cC$ comes with a functor 
$$\lMod{R}\times \cC\to \cC, \quad (V,X)\mapsto \iHom_R(V,X),$$
such that there are natural isomorphisms 
$$ \Hom_{R}(V,\Hom_\cC(X,Y))
\cong \Hom_\cC(X,\iHom_R(V,Y)).
$$

We note that if $\cC$ is both tensored and cotensored over $\lMod{R}$, any (co)limit in $\cC$ is a \emph{conical} (co)limit in the above sense \cite{Rie}*{Theorem~7.5.3}. The $R$-linear category $\Fun_R(\cC^\vee, \lMod{R})$ of $R$-linear functors is both tensored and cotensored over $\lMod{R}$ with 
$$(V\otimes_R M)(W):=V\otimes_R M(W), \quad \iHom_R(V,M)(W):=\iHom_R(V,M(W)),$$
for $R$-modules $V,W$ and any $R$-linear functor $M\colon \cC^\vee\to \lMod{R}$.

\subsection{Locally finitely presentable categories}

Assume that $\cC$ is an $R$-linear category with directed colimits, tensored and cotensored over $\lMod{R}$. An object $X$ of $\cC$ is \emph{compact} (or \emph{finitely presented}) if the functor $\Hom_\cC(X, -)\colon \cC\to \lMod{R}$ commutes with directed colimits.\footnote{Equivalently, directed colimit can be replaced by filtered colimits \cite{AR}.}
We denote by $\cC_\circ$ the full subcategory of compact objects.
Every representable object $\Hom_\cC(-,X)$ is compact in $\Fun(\cC^\vee,\lMod{R})$ as it preserves colimits \cite{AR}*{1.A}.
The finitely presented objects of the presheaf category are precisely the finite colimits of representables.

The category $\cC$ is \emph{compactly generated} if every object is a directed colimit  of compact objects. If $\cC$ is compactly generated and cocomplete then $\cC$ is \emph{locally finitely presentable (LFP)}.\footnote{In the enriched case, this definition is found in \cite{Kel2}*{(2.1)} and uses \emph{conical} colimits. Since we assume $\cC$ is tensored and cotensored over $\lMod{R}$, it suffices to consider ordinary colimits.}

Recall that a colimit is directed if the diagram category $D$ is a poset $(I,\leq)$ so that for any two elements $i,j$ there exists $k$ such that $i\leq k$ and $j\leq k$. We write 
\begin{equation}\label{eq:X-colim}
X=\colim_{i\in I}X_i
\end{equation}
when $X$ is such a directed colimit over the diagram $X_I\colon I\to \cC_\circ, i\mapsto X_i$. A morphism between two such objects 
\begin{equation}
f\colon \colim_{i\in I}X_i\to \colim_{j\in J}Y_j\label{eq:Ind-morphism}
\end{equation}
corresponds to a map of posets $\phi\colon I\to J$ and morphisms $f_i\colon X_i\to Y_{\phi(i)}$ such that, for all $i\leq j$, the 
diagrams
$$\xymatrix{X_i\ar[rr]^{f_i}\ar[d]&&Y_{\phi(i)}\ar[d]\\
X_{j}\ar[rr]^{f_j}&&Y_{\phi(j)}
}$$
commute. This data describes a functor $\phi\colon (I,\leq) \to (J,\leq)$ and a natural transformation between the diagram
\begin{equation}\label{eq:restrict-diagram}
\phi^*X_I\colon I\to \cC_\circ,\quad  i \mapsto Y_{\phi(i)},
\end{equation}
of shape $(I,\leq)$ given by the $X_i$ and the restriction $\phi^*X_I$ of the diagram of shape $(J,\leq)$ given by the $Y_i$ along the functor $\phi$. 

\begin{ex}
\begin{enumerate}
\item 
A vector space is finitely presented if and only if it is finite-dimensional. 
\item More generally, if $R$ is a ring, then the compact $R$-modules are precisely the finitely presentable modules. 
\item If $R$ is a Noetherian ring, then finitely generated modules are finitely presented. 
\item If $A$ is a compact $R$-algebra, then an $A$-module is compact if and only if it is compact as an $R$-module.
\end{enumerate}
\end{ex}

\subsection{Completion under directed colimits}\label{sec:cocompletion}

In the following, we always assume that $\cC$ is a compactly generated $R$-linear category,  tensored and cotensored over $\lMod{R}$. 
Consider the Yoneda embedding 
$$h\colon \cC\to \Fun_R(\cC^\vee,\lMod{R}), \qquad X\mapsto h_X=\Hom_\cC(-,X).$$
We will identify $\cC$ with its image under $h$.
It follows that $h_X$ is a compact object in the LFP category $\Fun_R(\cC^\vee,\lMod{R})$. In fact, all compact objects in the latter category are finite colimits of objects in $\cC$.

Since $\cC$ is compactly generated, the Yoneda embedding factors through the full subcategory $$\wcC:=\Fun_R^\lex(\cC_\circ^\vee,\lMod{R})$$ of \emph{left exact} or \emph{finitely concontinous functors}, i.e., contravariant, additive, $R$-linear functors that preserve finite limits that exist in $\cC_\circ^\vee$ (i.e., send finite colimits in $\cC$ to the corresponding limits in $\lMod{R}$). The restricted Yoneda embedding $h\colon \cC\hookrightarrow \wcC$ is right exact, i.e., preserves finite colimits, \cite{AR}*{1.45~Proposition}.

\begin{defn}
We call $\wcC$ the \emph{cocompletion} of $\cC$. 
\end{defn}

The cocompletion $\wcC$ is a LFP category \cite{AR}*{Section~1.46}. In particular, $\wcC$ is cocomplete  and complete and every object in $\wcC$ is a (countable) directed colimit of objects in $\cC_\circ$. Such objects can be displayed as $X=\colim_{i\in I}X_i$, with $X_i\in \cC_\circ$, as described above in Equation~\eqref{eq:X-colim}.

The cocompletion $\wcC$ has the following universal property \cite{AR}*{1.45 Definition}. If $\cD$ is a cocomplete category and $F\colon \cC_\circ\to \cD$ right exact functor, then there exists an extension to a cocontinuous (i.e., arbitrary colimit preserving) functor 
$\wF\colon \wcC\to\cD$, which is unique up to natural isomorphism. Thus, on an object $X=\colim_{i\in I} X_i$ we have 
$\wF(X)=\colim_{i\in I} F(X_i)$.

If $\cC$ is already LFP then the inclusion $\cC\hookrightarrow \wcC$ is an equivalence \cite{AR}*{1.45--1.46},  \cite{BCJF}*{Proposition~2.2}. In particular, $\wcC$ is equivalent to its cocompletion. We note that a LFP $R$-linear category $\cC$ is not equivalent to the \emph{free} cocompletion $\Fun_R(\cC,\lMod{R})$ since in the latter all representable functors are finitely presentable.

\begin{ex}\label{ex:Rmod-compl}
If $\cC=\lMod{R}$, then the subcategory of compact objects $\cC_\circ$ is the category of finitely presented $R$-modules \cite{Len} which we denote by $\fplmod{R}$.

The category $\cC=\lMod{R}$ itself is LFP so every $R$-module is a directed colimit of finitely presented $R$-modules and $\cC\simeq \wcC$.

We may consider the full subcategory $\cP=\flmod{R}$ of finitely-generated projective $R$-modules. The category $\wcC$ is co-complete so it contains all finite colimits of finitely-generated projective $R$-modules. In particular, all finitely-presented $R$-modules. As every $R$-module is a directed union of finitely-presented ones, $\wcP=\lMod{R}$. 
\end{ex}

Most examples of $R$-linear categories $\cC$ and their completions considered in the main text are of the form discussed in the following examples.

\begin{ex}\label{ex:Cocompletion-A-projR-finite}
Let $A$ be an $R$-algebra that is a finitely generated projective $R$-module. Consider the category $\cC=\flRmod{A}{R}$ which consists of $A$-modules that are finitely generated projective as $R$-modules. By the universal property and the Yoneda embedding, the cocompletion $\wcC$ is a full subcategory of $\lMod{A}$.  As $A$ is a finitely generated projective $R$-module, any finitely presented $A$-module is finitely presented as an $R$-module and hence contained in $\wcC$. Thus, $\wcC$ is equivalent to $\lMod{A}$.
\end{ex}

\begin{ex}\label{ex:Cocompletion-A-projR}
Let $A$ be an $R$-algebra with a filtration $A=\cup_{i\in \mN}A_i$, with $A_i\subset A_j$, where each $A_i$ is finitely generated projective over $R$ and the multiplication $m_A$ satisfies $m_A\colon A_i\otimes A_j\to A_{i+j}$. Then $A/A_i$ is contained in  $\cC=\flRmod{A}{R}$ and $A$ is the directed colimit of the $A/A_i$. Thus, any finitely presented $A$-module is contained in $\wcC$ and this cocompletion is equivalent to $\lMod{A}$.
\end{ex}

\begin{rmk}[The abelian case]
Note that if $\cC$ is abelian and compactly generated, then $\wcC$ is in fact a \emph{Grothendieck category} and $\cC\hookrightarrow \wcC$ is an exact functor \cite{BD}*{Theorem~5.40}.
\end{rmk}


\begin{ex}
If $\cC$ is a finite $\Bbbk$-linear abelian category \cite{EGNO}*{Definition~1.8.5}, then $\cC$ is equivalent to finite-dimensional modules over a $\Bbbk$-algebra $A$. In this case, $\wcC=\lMod{A}$, the category of all $A$-modules.
\end{ex}

\subsection{Tensor products of cocompletions}\label{sec:tensor-cocompletion}

In this section, we review the external tensor product of LFP categories. We work with $R$-linear categories. 

Given two  $R$-linear categories $\cC$, $\cD$ we define $\cC\otimes_R\cD$ to be the category with objects $X\otimes Y$ where $X$, $Y$ are objects of $\cC$, $\cD$, respectively. Morphism spaces are given by tensor products over $R$, i.e.,
$$\Hom_{\cC\otimes_R\cD}(X\otimes Y,X'\otimes Y')=\Hom_\cC(X,X')\otimes_R \Hom_\cD(Y,Y').$$

There is a tensor product $\cC\boxtimes_R \cD$ of $R$-linear LFP categories $\cC$, $\cD$ which is part of a cartesian closed structure in \cite{BCJF}. It satisfies the universal property that for any cocomplete $R$-linear category $\cE$ and any $R$-linear functor 
$F\colon \cC\otimes_R \cD\to \cE$ which is cocontinuous in each component  there exists a 
$R$-linear cocontinuous functor $\hat{F}\colon \cC\boxtimes_R \cD\to \cE$ which extends $F$ under the inclusion $\cC\otimes_R \cD\hookrightarrow \cC\boxtimes_R \cD$. Moreover, $\hat{F}$ is unique up to natural isomorphism. 

The compact objects in $\cC\boxtimes_R \cD$ are finite colimits of objects in $\cC_\circ\otimes_R \cD_\circ$. Thus, a model for 
$\cC\boxtimes_R \cD$ is given by $\Fun_R^\lex(\cC_\circ^\vee\otimes_R \cD_\circ^\vee, \lMod{R})$. In particular, $\cC\boxtimes \cD$ is the cocompletion $\widehat{\cT}$ of $\cT=\cC_\circ\boxtimes^{\mathrm{f}}_R \cD_\circ$, where $\boxtimes^{\mathrm{f}}_R$ denotes the Kelly tensor product \cite{Kel}*{Section~6.5}, the cocompletion of $\cC_\circ\otimes_R \cD_\circ$ under \emph{finite} colimits (enriching over $\lMod{R}$).

More generally, if  $\cC$ and $\cD$ are compactly generated $R$-linear categories, it makes sense to consider $\wcC\boxtimes_R \wcD$, the full subcategory of $\Fun_R(\cC_\circ^\vee\otimes_R \cD_\circ^\vee, \lMod{R})$ which are left exact in both components, also equivalent to the cocompletion of $\cC_\circ\boxtimes^{\mathrm{f}}_R\cD_\circ$. Then, by \cite{AR}*{1.49}, $\wcC\boxtimes_R \wcD$ is a LFP category. The functor $\wcC\otimes_R \wcD\to \wcC\boxtimes_R \wcD$ is cocontinuous  in both components and  $\wcC\boxtimes \wcD$ satisfies the following universal property.  
For any cocomplete category $\cE$ and any functor $\cC_\circ\otimes_R \cD_\circ\to \cE$ which is $R$-linear and right exact in both compontents there exists an extension to an $R$-linear cocontinuous functor $\wcC\boxtimes_R \wcD\to \cE$ which is unique up to natural isomorphism.

\begin{ex}
If $\cC=\lMod{A}$ and $\cD=\lMod{B}$ for $R$-algebras $A,B$, then $\cC\boxtimes_R \cD\simeq \lMod{A\otimes_R B}$.
\end{ex}

If it is clear from context, we will omit the subscript  $R$ from tensor products, i.e, denote $\otimes=\otimes_R$ and $\boxtimes=\boxtimes_R$.

\subsection{Cocompletions of \texorpdfstring{$R$}{R}-linear tensor categories}

In this section, we generalize \cite{Lyu3}*{Section~3.4} to equip the cocompletion $\wcC$ of an $R$-linear tensor category $\cC$ with the structure of a monoidal category with cocontinuous tensor product. We will use the following class of monoidal categories.

\begin{defn}\label{def:tensor}
Let $R$ be a commutative ring. An \emph{$R$-linear tensor category} is a monoidal category which is: 
\begin{enumerate}
\item[(i)] $R$-linear, additive, and Karoubian (i.e.~idempotent complete),
\item[(ii)] tensored and cotensored over $\lMod{R}$,
\item[(iii)] compactly generated,
\item[(iv)] equipped with a tensor product $\otimes \colon \cC\otimes \cC\to \cC$ that is cocontinuous in both components.
\end{enumerate}
\end{defn}

We note that the assumptions used here are weaker than the concept of tensor category in \cite{EGNO}, where $\cC$ is required to be abelian and rigid, having left and right duals. In particular, in \cite{EGNO}, the tensor product is exact rather than just right exact.

\begin{prop}\label{prop:wcC}
If $\cC$ is an $R$-linear tensor category, then so is $\wcC$, and the canonical fully faithful functor $\cC\to \wcC$ is a right exact functor of monoidal categories.
\end{prop}
\begin{proof}
Properties (i)--(iii) are inherited by the cocompletion.
It remains to check (iv) for $\wcC$. Fix an object $X$ of $\cC$. Then the functor $(-)\otimes X\colon \cC\to \cC\hookrightarrow \wcC$ extends to a cocontinuous functor $(-)\otimes X\colon \wcC\to \wcC$ by the universal property of the cocompletion. The same argument works for the second tensor slot. Together, we obtain an extension of the tensor product to a functor 
$$\otimes\colon \wcC\otimes \wcC\to \wcC.$$
Following \cite{Lyu3}*{Section~3.4}, we define the tensor product of objects $X=\colim_{i\in I}X_i$ and $Y=\colim_{j\in J}Y_j$. The product $I\otimes J$ is a directed poset with $(i,j)\leq (i',j')$ if and only if $i\leq i'$ and $j\leq j'$. Thus, we can define  
$$X\otimes Y=\colim_{(i,j)\in I\times J}X_i\otimes Y_j.$$
The associativity and unitality isomorphisms and their coherences naturally extend to $\wcC$. 
\end{proof}

In the abelian case, we obtain from \cite{BD}*{Theorem~5.40} that $\wcC$ is a Grothen\-dieck category. The above proposition now yields the following.

\begin{cor}
If $\cC$ is an abelian monoidal category with right exact tensor product, then $\wcC$ is an abelian monoidal category with right exact tensor product and $\cC\hookrightarrow \wcC$ is an exact functor.
\end{cor}

We are mostly interested in examples coming from Hopf algebras over $R$.

\begin{ex}\label{ex:Cocompletion-H-projR}
Let $A$ be an $R$-algebra with a filtration $A=\cup_{i\in \mN}A_i$, with $A_i\subset A_j$, where each $A_i$ is finitely generated projective over $R$ as in Example~\ref{ex:Cocompletion-A-projR}. Then $\cC=\flRmod{H}{R}$ is an $R$-linear tensor category with the tensor product $\otimes_R$, where $H$ acts via the coproduct. It follows that $\wcC$ is an $R$-linear tensor category and equivalent to $\lMod{H}$. 
\end{ex}

Passing to cocompletions is $2$-functorial. In particular, if $\cC$ has a braiding then so does the cocompletion $\wcC$.

\bibliography{main}

\begin{bibdiv}
\begin{biblist}

\bib{AR}{book}{
      author={Ad\'{a}mek, Ji\v{r}\'{\i}},
      author={Rosick\'{y}, Ji\v{r}\'{\i}},
       title={Locally presentable and accessible categories},
      series={London Mathematical Society Lecture Note Series},
   publisher={Cambridge University Press, Cambridge},
        date={1994},
      volume={189},
        ISBN={0-521-42261-2},
         url={https://doi.org/10.1017/CBO9780511600579},
}

\bib{BCJF}{article}{
      author={Brandenburg, Martin},
      author={Chirvasitu, Alexandru},
      author={Johnson-Freyd, Theo},
       title={Reflexivity and dualizability in categorified linear algebra},
        date={2015},
     journal={Theory Appl. Categ.},
      volume={30},
       pages={Paper No. 23, 808\ndash 835},
}

\bib{BD}{book}{
      author={Bucur, Ion},
      author={Deleanu, Aristide},
       title={Introduction to the theory of categories and functors},
      series={Pure and Applied Mathematics, Vol. XIX},
   publisher={John Wiley \& Sons, Ltd., London-New York-Sydney},
        date={1968},
        note={With the collaboration of Peter J. Hilton and Nicolae Popescu, A
  Wiley Interscience Publication},
}

\bib{BG}{book}{
      author={Brown, Ken~A.},
      author={Goodearl, Ken~R.},
       title={Lectures on algebraic quantum groups},
      series={Advanced Courses in Mathematics. CRM Barcelona},
   publisher={Birk\-h\"{a}user Verlag, Basel},
        date={2002},
        ISBN={3-7643-6714-8},
         url={https://doi.org/10.1007/978-3-0348-8205-7},
}

\bib{BBJ}{article}{
      author={Ben-Zvi, David},
      author={Brochier, Adrien},
      author={Jordan, David},
       title={Integrating quantum groups over surfaces},
        date={2018},
        ISSN={1753-8416},
     journal={J. Topol.},
      volume={11},
      number={4},
       pages={874\ndash 917},
         url={https://doi.org/10.1112/topo.12072},
}

\bib{BBJ2}{article}{
      author={Ben-Zvi, David},
      author={Brochier, Adrien},
      author={Jordan, David},
       title={Quantum character varieties and braided module categories},
        date={2018},
        ISSN={1022-1824},
     journal={Selecta Math. (N.S.)},
      volume={24},
      number={5},
       pages={4711\ndash 4748},
         url={https://doi.org/10.1007/s00029-018-0426-y},
}

\bib{Che}{article}{
      author={Cherednik, I.~V.},
       title={Factorizing particles on a half line, and root systems},
        date={1984},
        ISSN={0564-6162},
     journal={Teoret. Mat. Fiz.},
      volume={61},
      number={1},
       pages={35\ndash 44},
}

\bib{DKM}{article}{
      author={Donin, J.},
      author={Kulish, P.~P.},
      author={Mudrov, A.~I.},
       title={On a universal solution to the reflection equation},
        date={2003},
        ISSN={0377-9017},
     journal={Lett. Math. Phys.},
      volume={63},
      number={3},
       pages={179\ndash 194},
         url={https://doi.org/10.1023/A:1024438101617},
}

\bib{DL}{article}{
      author={Domokos, M.},
      author={Lenagan, T.~H.},
       title={Quantized trace rings},
        date={2005},
        ISSN={0033-5606},
     journal={Q. J. Math.},
      volume={56},
      number={4},
       pages={507\ndash 523},
         url={https://doi.org/10.1093/qmathj/hah051},
}

\bib{DM}{article}{
      author={Donin, J.},
      author={Mudrov, A.},
       title={Reflection equation, twist, and equivariant quantization},
        date={2003},
        ISSN={0021-2172},
     journal={Israel J. Math.},
      volume={136},
       pages={11\ndash 28},
         url={https://doi.org/10.1007/BF02807191},
}

\bib{DGGPR}{article}{
      author={De~Renzi, Marco},
      author={Gainutdinov, Azat~M.},
      author={Geer, Nathan},
      author={Patureau-Mirand, Bertrand},
      author={Runkel, Ingo},
       title={3-{D}imensional {TQFT}s from non-semisimple modular categories},
        date={2022},
        ISSN={1022-1824},
     journal={Selecta Math. (N.S.)},
      volume={28},
      number={2},
       pages={Paper No. 42, 60},
         url={https://doi.org/10.1007/s00029-021-00737-z},
}

\bib{EGNO}{book}{
      author={Etingof, Pavel},
      author={Gelaki, Shlomo},
      author={Nikshych, Dmitri},
      author={Ostrik, Victor},
       title={Tensor categories},
      series={Mathematical Surveys and Monographs},
   publisher={American Mathematical Society, Providence, RI},
        date={2015},
      volume={205},
}

\bib{FRT}{incollection}{
      author={Faddeev, L.~D.},
      author={Reshetikhin, N.~Yu.},
      author={Takhtajan, L.~A.},
       title={Quantization of {L}ie groups and {L}ie algebras},
        date={1988},
   booktitle={Algebraic analysis, {V}ol. {I}},
   publisher={Academic Press, Boston, MA},
       pages={129\ndash 139},
}

\bib{FS}{incollection}{
      author={Fuchs, J\"{u}rgen},
      author={Schweigert, Christoph},
       title={Coends in conformal field theory},
        date={2017},
   booktitle={Lie algebras, vertex operator algebras, and related topics},
      series={Contemp. Math.},
      volume={695},
   publisher={Amer. Math. Soc., Providence, RI},
       pages={65\ndash 81},
         url={https://doi.org/10.1090/conm/695/13996},
}

\bib{Jan}{book}{
      author={Jantzen, Jens~Carsten},
       title={Lectures on quantum groups},
      series={Graduate Studies in Mathematics},
   publisher={American Mathematical Society, Providence, RI},
        date={1996},
      volume={6},
        ISBN={0-8218-0478-2},
         url={https://doi.org/10.1090/gsm/006},
}

\bib{JW}{article}{
      author={Jordan, David},
      author={White, Noah},
       title={The center of the reflection equation algebra via quantum
  minors},
        date={2020},
        ISSN={0021-8693},
     journal={J. Algebra},
      volume={542},
       pages={308\ndash 342},
         url={https://doi.org/10.1016/j.jalgebra.2019.08.038},
}

\bib{Kel2}{article}{
      author={Kelly, G.~M.},
       title={Structures defined by finite limits in the enriched context.
  {I}},
        date={1982},
        ISSN={0008-0004},
     journal={Cahiers Topologie G\'{e}om. Diff\'{e}rentielle},
      volume={23},
      number={1},
       pages={3\ndash 42},
        note={Third Colloquium on Categories, Part VI (Amiens, 1980)},
}

\bib{Kel}{book}{
      author={Kelly, Gregory~Maxwell},
       title={Basic concepts of enriched category theory},
      series={London Mathematical Society Lecture Note Series},
   publisher={Cambridge University Press, Cambridge-New York},
        date={1982},
      volume={64},
        ISBN={0-521-28702-2},
}

\bib{Kol}{article}{
      author={Kolb, Stefan},
       title={Braided module categories via quantum symmetric pairs},
        date={2020},
        ISSN={0024-6115},
     journal={Proc. Lond. Math. Soc. (3)},
      volume={121},
      number={1},
       pages={1\ndash 31},
         url={https://doi.org/10.1112/plms.12303},
}

\bib{KoSt}{article}{
      author={Kolb, Stefan},
      author={Stokman, Jasper~V.},
       title={Reflection equation algebras, coideal subalgebras, and their
  centres},
        date={2009},
        ISSN={1022-1824},
     journal={Selecta Math. (N.S.)},
      volume={15},
      number={4},
       pages={621\ndash 664},
         url={https://doi.org/10.1007/s00029-009-0007-1},
}

\bib{KuS}{article}{
      author={Kulish, P.~P.},
      author={Sklyanin, E.~K.},
       title={Algebraic structures related to reflection equations},
        date={1992},
        ISSN={0305-4470},
     journal={J. Phys. A},
      volume={25},
      number={22},
       pages={5963\ndash 5975},
         url={http://stacks.iop.org/0305-4470/25/5963},
}

\bib{KS}{book}{
      author={Klimyk, Anatoli},
      author={Schm\"{u}dgen, Konrad},
       title={Quantum groups and their representations},
      series={Texts and Monographs in Physics},
   publisher={Springer-Verlag, Berlin},
        date={1997},
        ISBN={3-540-63452-5},
         url={https://doi.org/10.1007/978-3-642-60896-4},
}

\bib{Kul}{incollection}{
      author={Kulish, P.~P.},
       title={Yang-{B}axter equation and reflection equations in integrable
  models},
        date={1996},
   booktitle={Low-dimensional models in statistical physics and quantum field
  theory ({S}chladming, 1995)},
      series={Lecture Notes in Phys.},
      volume={469},
   publisher={Springer, Berlin},
       pages={125\ndash 144},
         url={https://doi.org/10.1007/BFb0102555},
}

\bib{Len}{article}{
      author={Lenzing, Helmut},
       title={Endlich pr\"{a}sentierbare {M}oduln},
        date={1969},
        ISSN={0003-889X},
     journal={Arch. Math. (Basel)},
      volume={20},
       pages={262\ndash 266},
         url={https://doi.org/10.1007/BF01899297},
}

\bib{LM}{article}{
      author={Lyubashenko, Volodimir},
      author={Majid, Shahn},
       title={Braided groups and quantum {F}ourier transform},
        date={1994},
        ISSN={0021-8693},
     journal={J. Algebra},
      volume={166},
      number={3},
       pages={506\ndash 528},
         url={https://doi.org/10.1006/jabr.1994.1165},
}

\bib{Lus2}{article}{
      author={Lusztig, George},
       title={Finite-dimensional {H}opf algebras arising from quantized
  universal enveloping algebra},
        date={1990},
        ISSN={0894-0347},
     journal={J. Amer. Math. Soc.},
      volume={3},
      number={1},
       pages={257\ndash 296},
         url={https://doi.org/10.2307/1990988},
}

\bib{LWY}{article}{
      author={Laugwitz, Robert},
      author={Walton, Chelsea},
      author={Yakimov, Milen},
       title={Reflective centers of module categories and quantum
  {K}-matrices},
        date={2025},
     journal={Forum Math. Sigma (to appear)},
       pages={arXiv:2307.14764},
}

\bib{Lyu2}{article}{
      author={Lyubashenko, V.},
       title={Modular transformations for tensor categories},
        date={1995},
        ISSN={0022-4049},
     journal={J. Pure Appl. Algebra},
      volume={98},
      number={3},
       pages={279\ndash 327},
         url={https://doi.org/10.1016/0022-4049(94)00045-K},
}

\bib{Lyu3}{article}{
      author={Lyubashenko, Volodimir},
       title={Tangles and {H}opf algebras in braided categories},
        date={1995},
        ISSN={0022-4049},
     journal={J. Pure Appl. Algebra},
      volume={98},
      number={3},
       pages={245\ndash 278},
         url={https://doi.org/10.1016/0022-4049(95)00044-W},
}

\bib{Lyu1}{article}{
      author={Lyubashenko, Volodymyr~V.},
       title={Invariants of {$3$}-manifolds and projective representations of
  mapping class groups via quantum groups at roots of unity},
        date={1995},
        ISSN={0010-3616},
     journal={Comm. Math. Phys.},
      volume={172},
      number={3},
       pages={467\ndash 516},
         url={http://projecteuclid.org/euclid.cmp/1104274312},
}

\bib{Majid95book}{book}{
      author={Majid, Shahn},
       title={Foundations of quantum group theory},
   publisher={Cambridge University Press, Cambridge},
        date={2000},
        ISBN={0-521-64868-8},
         url={http://dx.doi.org/10.1017/CBO9780511613104},
        note={Paperback Edition, originally published 1995},
}

\bib{Maj1}{article}{
      author={Majid, Shahn},
       title={Braided groups},
        date={1993},
        ISSN={0022-4049},
     journal={J. Pure Appl. Algebra},
      volume={86},
      number={2},
       pages={187\ndash 221},
         url={https://doi.org/10.1016/0022-4049(93)90103-Z},
}

\bib{PW}{article}{
      author={Parshall, Brian},
      author={Wang, Jian~Pan},
       title={Quantum linear groups},
        date={1991},
        ISSN={0065-9266},
     journal={Mem. Amer. Math. Soc.},
      volume={89},
      number={439},
         url={https://doi.org/10.1090/memo/0439},
}

\bib{Rie}{book}{
      author={Riehl, Emily},
       title={Categorical homotopy theory},
      series={New Mathematical Monographs},
   publisher={Cambridge University Press, Cambridge},
        date={2014},
      volume={24},
        ISBN={978-1-107-04845-4},
         url={https://doi.org/10.1017/CBO9781107261457},
}

\bib{RTF}{article}{
      author={Reshetikhin, N.~Yu.},
      author={Takhtadzhyan, L.~A.},
      author={Faddeev, L.~D.},
       title={Quantization of {L}ie groups and {L}ie algebras},
        date={1989},
        ISSN={0234-0852},
     journal={Algebra i Analiz},
      volume={1},
      number={1},
       pages={178\ndash 206},
}

\bib{Sch}{book}{
      author={Schauenburg, Peter},
       title={Tannaka duality for arbitrary {H}opf algebras},
      series={Algebra Berichte},
   publisher={Verlag Reinhard Fischer, Munich},
        date={1992},
      volume={66},
        ISBN={3-88927-100-6},
}

\bib{Tak}{article}{
      author={Takeuchi, Mitsuhiro},
       title={Some topics on {${\rm GL}_q(n)$}},
        date={1992},
        ISSN={0021-8693},
     journal={J. Algebra},
      volume={147},
      number={2},
       pages={379\ndash 410},
         url={https://doi.org/10.1016/0021-8693(92)90212-5},
}

\bib{TV}{book}{
      author={Turaev, Vladimir},
      author={Virelizier, Alexis},
       title={Monoidal categories and topological field theory},
      series={Progress in Mathematics},
   publisher={Birkh\"{a}user/Springer, Cham},
        date={2017},
      volume={322},
        ISBN={978-3-319-49833-1; 978-3-319-49834-8},
         url={https://doi.org/10.1007/978-3-319-49834-8},
}

\bib{Zh}{article}{
      author={Zhang, R.~B.},
       title={Howe duality and the quantum general linear group},
        date={2003},
        ISSN={0002-9939},
     journal={Proc. Amer. Math. Soc.},
      volume={131},
      number={9},
       pages={2681\ndash 2692},
         url={https://doi.org/10.1090/S0002-9939-02-06892-2},
}

\end{biblist}
\end{bibdiv}
\bibliographystyle{amsrefs}%

\end{document}